%% file: main_v2.tex
\documentclass[11pt,a4paper]{amsart}
\usepackage[utf8]{inputenc}
\usepackage{amssymb,amsmath}

\usepackage{lipsum}
\usepackage{amsfonts}
\usepackage{amssymb,amsmath}
\usepackage{graphicx}
\usepackage{bm,bbm}
\usepackage{color}
\usepackage{tikz}
\usepgflibrary{decorations.pathmorphing}
\usetikzlibrary{patterns}
\usepackage{subcaption}
\usepackage{ifthen}

\usepackage{pgfplots}
\usepackage{enumitem}
\usepackage{array,blkarray}
\usepackage{bigdelim}
\usepackage{listings} 
\usepackage[hidelinks]{hyperref}
\usepackage{booktabs}
\usepackage{isotope}
\usepackage{stmaryrd}
\usepackage[]{placeins}
\usepackage{algorithm}
\usepackage{algpseudocode}
\pgfplotsset{compat=1.18} 

\ifpdf
\DeclareGraphicsExtensions{.eps,.pdf,.png,.jpg}
\else
\DeclareGraphicsExtensions{.eps}
\fi

\newtheorem{theorem}{Theorem}
\newtheorem{proposition}[theorem]{Proposition}
\newtheorem{corollary}[theorem]{Corollary}
\theoremstyle{definition}

\newtheorem{lemma}[theorem]{Lemma}

\theoremstyle{remark}
\newtheorem{remark}[theorem]{Remark}

\numberwithin{theorem}{section}
\numberwithin{equation}{section}
\numberwithin{table}{section}
\numberwithin{figure}{section}
\input{def}

\title[Riemannian Optimisation Methods for Multicomponent BEC]
{Riemannian Optimisation Methods for Ground States of Multicomponent Bose--Einstein Condensates}

\author[R. Altmann]{Robert Altmann}
\address[R. Altmann]{Institute of Analysis and Numerics, Otto von Guericke University Magdeburg, Universit\"atsplatz 2, 39106 Magdeburg, Germany}
\email{robert.altmann@ovgu.de}

\author[M. Hermann]{Martin Hermann}
\address[M. Hermann]{Institute of Mathematics, University of Augsburg, Universit\"atsstra{\ss}e~12a, 86159 Augsburg, Germany}
\email{martin.hermann@uni-a.de}

\author[D. Peterseim]{Daniel Peterseim}
\author[T. Stykel]{Tatjana Stykel}
\address[D. Peterseim, T. Stykel]{Institute of Mathematics \& Centre for Advanced Analytics and Predictive Sciences (CAAPS), University of Augsburg, Universit\"atsstra{\ss}e~12a, 86159 Augsburg, Germany}
\email{daniel.peterseim@uni-a.de}
\email{tatjana.stykel@uni-a.de}

\thanks{The work of M.~Hermann and D.~Peterseim is part of a project that has received funding from the European Research Council (ERC) under the European Union's Horizon 2020 research and innovation programme (Grant agreement No.~865751 -- RandomMultiScales).\\
\indent
This paper will appear in IMA Journal of Numerical Analysis.}
\begin{document}
%
%
\begin{abstract}
	This paper addresses the computation of ground states of multicomponent Bose--Einstein condensates, defined as the global minimiser of an energy functional on an infinite-dimensional ge\-ne\-ra\-lised oblique manifold. We establish the existence of the ground state, prove its uniqueness up to scaling, and characterise it as the solution to a coupled nonlinear eigenvector problem. By equipping the manifold with several Riemannian metrics, we introduce a suite of Riemannian gradient descent and Riemannian Newton methods. Metrics that incorporate first- or second-order information about the energy are particularly advantageous, effectively preconditioning the resulting methods. For a~Riemannian gradient descent method with an~energy-adaptive metric, we provide a qualitative global and quantitative local convergence analysis, confirming its reliability and robustness with respect to the choice of the spatial discretisation. Numerical experiments highlight the computational efficiency of both the Riemannian gradient descent and Newton methods.
\end{abstract}
%
\maketitle

%
%
%
%
\section{Introduction}
This paper is devoted to the numerical approximation of ground states of multicomponent Bose--Einstein condensates (BECs), where $p$ different species of matter can condense into one single internal state with $p$ components; see~\cite{Bao04,ChaKLS05,CheW03} for an~introduction. For the special case of two species and two components, we refer to~\cite{BaoC11,LinW06} and \cite[Sect.~9.2]{BaoC13}. Multicomponent systems can also be obtained, under some simplifying assumptions, from spinor BECs which provide a~valuable tool for exploring complex quantum phenomena \cite{BaoC18,BaoCZ13}.

Mathematically, a stationary quantum state of a $p$-component condensate can be modelled as a~wave function $\varphibf=(\varphi_1,\ldots,\varphi_p)$, defined in a domain $\Omega \subset \R^d$, $d=1,2,3$. The components of the wave function $\varphibf$ are square integrable with square integrable weak derivatives. Their squared $L^2$-norms represent the masses of the components which are subject to the constraints
\begin{equation}\label{eq:constraints}
\|\varphi_j\|_{L^2(\Omega)}^2 = N_j > 0, \qquad j=1,\ldots,p,
\end{equation}
where $N_j$ is the given number of particles in the $j$-th condensate component, and $N_1+\ldots+N_p$ is the total number of particles of the condensate. The set of admissible quantum states of the condensates thus forms an~infinite-dimensional generalised oblique manifold. Given the finite mass, the wave function is well approximated in a sufficiently large but bounded domain $\Omega$ with homogeneous Dirichlet boundary conditions. Furthermore, the linear combinations $\rho_1(\varphibf),\ldots,\rho_p(\varphibf)$ of the density functions $|\varphi_i|^2$ are given by
\[
\rho_j(\varphibf)=\sum_{i=1}^p \kappa_{ij} |\varphi_i|^2, \qquad j=1,\ldots,p,
\]
where the parameters $\kappa_{ij} \in \R$, $\kappa_{ij} = \kappa_{ji}$, characterise the strength of particle interactions between the $i$-th and $j$-th condensate components. These interactions can be either attractive \mbox{($\kappa_{ij}<0$)} or repulsive \mbox{($\kappa_{ij}>0$)}.  

The energy of the condensate is given by the functional 
\begin{equation}\label{eq:energy}
\calE(\varphibf)
= \sum_{j=1}^p \int_\Omega \frac{1}{2}\, \|\nabla\varphi_j\|^2 
+ \frac{1}{2}\, \V_j(x)\, |\varphi_j|^2 
+ \frac{1}{4}\, \rho_j(\varphibf) \,|\varphi_j|^2 \, \dx,
\end{equation}
where $\V_1, \ldots, \V_p$ denote the external potentials, which can be, for example, harmonic traps or periodic potentials created by optical lattices or realisations of random potentials created by laser speckles. The first term in \eqref{eq:energy} represents the kinetic energy of the $j$-th component, the second term is the potential energy, and the last term characterises the energy resulting from particle interactions. The coupling is encoded in the nonlinearities $\rho_j$ and driven by the parameters $\kappa_{ij}$, whose relative variation can significantly impact the ground state. In the two-component case $p=2$, for example, the parameter 
$$
\Delta_{\text{misc}} = \frac{\kappa_{11}\kappa_{22}}{\kappa_{12}^2} - 1
$$ 
describes the miscibility of the condensate. For $\Delta_{\text{misc}} > 0$, ground states with substantially overlapping components are expected, whereas for $\Delta_{\text{misc}} < 0$, the components are mostly separated in space according to the Gross--Pitaevskii model; see, e.g., \cite{Tri00}. In this paper, we focus on the regime where $\Delta_{\text{misc}} > 0$. 

We are interested in global minimisers of the energy \eqref{eq:energy} subject to the mass constraints \eqref{eq:constraints} on the components. These minimisers are the ground states that represent the most stable configurations of the system.
The existence of a ground state for $p=2$ in the full-space problem has been shown in \cite{BaoC11}. Building on the approach used to establish the existence and uniqueness in the single-component case \cite{CanCM10,LieSY00}, we extend these results to general $p$ and bounded domains under suitable assumptions on the potentials $V_j$ and the interaction matrix $K=[\kappa_{ij}]_{i,j=1}^p$. 

The ground state of a~multicomponent BEC with $p$ interacting components is further cha\-rac\-terised as the solution to the nonlinear eigenvector problem (NLEVP) given by the coupled (dimensionless) Gross--Pitaevskii equations 
\begin{equation}\label{eq:multiBEC}
	-\Delta \varphi_j + \V_j \varphi_j + \rho_j(\varphibf)\varphi_j 
    = \lambda_j\varphi_j, \qquad j=1,\ldots,p,
\end{equation}
where the eigenvalues $\lambda_1,\ldots,\lambda_p \in \R$ represent the chemical potentials of the components. We show that, similar to the single-component case \cite{CanCM10}, each component of the ground state is associated with the eigenfunction corresponding to the smallest eigenvalue for the corresponding condensate component. This relationship is crucial in proving both the uniqueness of the ground state, up to a~global sign change of its components, and a second-order necessary optimality condition.

Various numerical methods have been proposed for computing ground states of the Gross--Pitaevskii equation in the single-component case, based on either energy minimisation or eigenvalue characterisation. These methods include self-consistent field iterations~\cite{CanKL21,DioC07}, discrete normalised gradient flows~\cite{BaoCL06,BaoD04}, (projected) Sobolev gradient methods~\cite{ChenLLZ24,DanK10,GarP01,HenP20,KaE10}, the \mbox{$J$-me}\-thod \cite{AltHP21,JarKM14}, Riemannian optimisation methods in both discrete and continuous settings~\cite{AltPS24,AntLT17,DanP17}, and Newton-type methods~\cite{BaoT03,DuL22,WuWB17}. For more details and an extended overview, we refer to~\cite{HenJ23}.
Although the computation of ground states for multicomponent BECs is also important in practice, it has received significantly less attention. Among the methods proposed for these problems are gradient flow-based approaches~\cite{AntD14,Bao04}, a Newton-like method with an approximate line search strategy~\cite{CalORT09}, a regularised Riemannian Newton method \cite{TiaCWW20}, and more recently Newton-based alternating methods~\cite{HuaY24}.

This paper promotes Riemannian gradient descent and Newton-type methods for multicomponent BECs, emphasising their robustness across different spatial discretisation methods by working within an infinite-dimensional framework. Consequently, the convergence properties of selected optimisation schemes are asymptotically independent when the resolution in the spatial discretisation is increased, such as mesh refinement in finite element methods. In contrast, conventional convergence proofs in Riemannian optimisation for discrete models often lack an explicit dependence on the dimension of the problem, limiting their applicability to discretised partial differential equations where the dimension is critical and increases with increasing spatial resolution. 
Furthermore, we provide a rigorous global convergence analysis and quantify the local convergence behaviour of a~selected Riemannian gradient descent scheme, demonstrating its effectiveness in reliably identifying accurate initial values for empirically more efficient Riemannian Newton methods within sufficiently close neighbourhoods of the ground state.

The paper is organised as follows. Section~\ref{sec:setting} introduces the functional analytical framework for constrained energy minimisation and proves the existence of a ground state. Section~\ref{sec:eigenvector} establishes the connection to NLEVPs and shows the uniqueness of the ground state. In Section~\ref{sec:oblique}, we characterise the tangent and normal spaces of the infinite-dimensional generalised oblique manifold, deriving formulae for Riemannian gradients and Hessians based on the chosen metric. Section~\ref{sec:methods} presents the Riemannian gradient descent and Newton methods, followed by the convergence analysis of the energy-adaptive Riemannian gradient descent method in Section~\ref{sec:convergence}. Section~\ref{sec:FEM} discusses the finite element discretisation and the discrete geometric ingredients. Finally, Section~\ref{sec:numerics} reports on numerical experiments.\medskip

\textbf{Notation.}
The set of $p\times p$ real diagonal matrices is denoted by $\Dp$. For $M\in\R^{p\times p}$, the trace of~$M$ is denoted by $\trace M$. In addition, we denote by $\diag(v)$ the diagonal matrix with components of a~vector $v\in\R^p$ on the diagonal and by $\ddiag(M)$ the diagonal matrix whose diagonal elements are the same as that of $M$. The $p\times p$ identity and zero matrices are denoted by $I_p$ and $0_p$, respectively. The Euclidean vector norm is denoted by $\|\cdot\|$ and the spectral matrix norm is denoted by $\|\cdot\|_2$.
%
%
\section{Constrained energy minimisation and existence of a ground state}\label{sec:setting}
%
%
The ground states are precisely defined by a mathematical model of constrained energy minimisation. Their existence is then proved under appropriate assumptions on the model parameters.

\subsection{Energy-related mathematical model}\label{sec:setting:energy}
Let $\Omega\subset\R^d$ with $d=1,2,3$ be a~bounded convex Lipschitz domain. Consider the Hilbert space $L^2(\Omega)$ with the inner product $(\,\cdot\,,\cdot\,)_{L^2(\Omega)}$ and the Sobolev space $H_0^1(\Omega)$ with the inner product $(\,\cdot\,,\cdot\,)_{H^1(\Omega)}$. For $p\geq 1$, we define the Hilbert spaces $\Lsp=[L^2(\Omega)]^p$ and $\Hsp=[H_0^1(\Omega)]^p$ of $p$-frames which form a~Gelfand triple $\Hsp \subset \Lsp \subset \Hsp^*$, where $\Hsp^*=[H^{-1}(\Omega)]^p$ denotes the dual space of $\Hsp$. 
On the pivot space $\Lsp$, we define an~inner product
\begin{equation*}
	(\vbf, \wbf)_{\Lsp} 
	= \sum_{j=1}^p (v_j, w_j)_{L^2(\Omega)} 
\end{equation*}
for $\vbf = (v_1, \dots, v_p), \wbf = (w_1, \dots, w_p) \in \Lsp$, which induces the norm $\|\vbf\|_{\Lsp} = \sqrt{(\vbf, \vbf)_{\Lsp}}$. In addition, we introduce the diagonal matrix
\[
    \out{\vbf}{\wbf}
    =\diag\big((v_1, w_1)_{L^2(\Omega)},\ldots,(v_p, w_p)_{L^2(\Omega)}\big)
\]
of the component-wise $L^2$-inner products of $\vbf,\wbf\in\Lsp$.
Due to the symmetry of the $L^2$-inner product,
we have \mbox{$\out{\vbf}{\wbf}=\out{\wbf}{\vbf}$}. Moreover, it holds
$(\vbf, \wbf)_{\Lsp} = \trace\, \out{\vbf}{\wbf}$. For $1\le q < \infty$, we also define the norms
\[
\|\vbf\|_{[L^q(\Omega)]^p}
    = \Big(\sum_{j=1}^p \|v_j\|_{L^q(\Omega)}^q\Big)^{1/q}
\]
on the Lebesgue spaces $[L^q(\Omega)]^p$ of $p$-frames. 

Next, we consider the energy functional $\calE\colon \Hsp \rightarrow \R$ from~\eqref{eq:energy}, which can equivalently be written as
\begin{equation}\label{eq:energy_short}
\calE(\varphibf) = \int_\Omega \frac{1}{2} \trace\big((\nabla \varphibf)^T\nabla \varphibf\big)
+ \frac{1}{2}(\varphibf \circ \varphibf) V(x) + \frac{1}{4}(\varphibf \circ \varphibf) K(\varphibf \circ \varphibf)^T \,\dx,
\end{equation}
where 
$\varphibf \circ \vbf = (\varphi_1 v_1, \ldots, \varphi_p v_p)$ denotes the Hadamard (component-wise) product of two $p$-frames, $V = [\V_1, \ldots, \V_p]^T$, and $K = [\kappa_{ij}]_{i,j=1}^{p} \in \R^{p \times p}$. 
Our aim is to compute a global minimiser of $\calE$, a~so-called {\em ground state}, on the manifold
\begin{equation}\label{eq:obligman}
	\OB =  \big\{ \varphibf \in \Hsp \,:\, \out{\varphibf}{\varphibf} = N \big\}
\end{equation}
of admissible states that respect the prescribed masses of the individual components, encoded in the diagonal matrix $N = \diag(N_1, \ldots, N_p)$; see also \eqref{eq:constraints}. This manifold is known as the \emph{infinite-dimensional generalised oblique manifold}. The resulting constrained energy minimisation problem can then shortly be written as
\begin{equation}\label{eq:min}
	\min_{\varphibf\in\OB} \calE(\varphibf).
\end{equation} 
For simplicity, the restriction of $\,\calE$ to $\OB$ will also be denoted by $\,\calE$ in what follows.
\begin{remark}[Scaling invariance of $\calE$ and non-uniqueness of ground state]\label{rm:invariance}
The energy functional satisfies
$\calE(\varphibf \, \Sigma_{\pm 1}) = \calE(\varphibf)$ 
for all $\varphibf \in \OB$ and $\Sigma_{\pm 1} \in \Dp$ with $\pm 1$ on the diagonal. Thus, any global or local minimiser of \eqref{eq:min}, if it exists, is not unique in the strict sense. A proper concept of uniqueness will be introduced later in Section~\ref{sec:eigenvector}.
\end{remark}
Throughout this paper, we make the following assumptions on the model parameters $\V_j$ and~$K$: 
\vspace{0.1em}
\begin{itemize}[itemsep=0.2em]
	\item[\bf{A1:}] The external potentials satisfy $\V_j \in L^\infty(\Omega)$ with $\V_j(x) \geq 0$ for almost all $x \in \Omega$.
	\item[\bf{A2:}] The interaction matrix $K$ is symmetric and positive definite.
\end{itemize}
\vspace{0.4em}
Assumption~{\bf A1} allows us to equip the space $\Hsp$ with the potential-dependent inner product
\[
(\vbf, \wbf)_\Hsp 
    = \sum_{j=1}^p \int_\Omega (\nabla v_j)^T \nabla w_j 
        + V_j(x)\, v_j w_j \, \dx
\]
and the induced norm $\|\vbf\|_\Hsp = \sqrt{(\vbf, \vbf)_\Hsp}$, which is equivalent to the canonical norm of $\Hsp$.
Note that for $p=2$, Assumption~{\bf A2} corresponds to the miscible regime $\Delta_{\text{misc}} > 0$. 
In Section~\ref{sec:setting:groundstate} below, we show that these assumptions ensure the existence of a ground state. In the proof and throughout the paper, we will frequently use several Sobolev embedding inequalities \cite[Sec.~10.9]{Alt16} generally given by  
\begin{equation}\label{eq:qSobolev}
\|\vbf\|_{[L^q(\Omega)]^p} \leq C_q\, \|\vbf\|_\Hsp \qquad \text{ for all } \vbf \in \Hsp
\end{equation}
with $1\leq q\leq 6$ and constants $C_q>0$ depending on $q$, $d$, and $\Omega$.

\subsection{Properties of the derivative of the energy}
The energy functional $\,\calE$ is Fr\'echet differentiable on~$\Hsp$. Given Assumption~{\bf A2}, its directional derivative at $\varphibf \in \Hsp$ along $\wbf \in \Hsp$ is given by 
\[
	\Drm \calE(\varphibf)[\wbf] = a_\varphibf(\varphibf, \wbf)
\]
with the bilinear form $a_\varphibf\colon \Hsp\times \Hsp\to\R$, 
\begin{align}\label{eq:aphi}
	a_\varphibf(\vbf, \wbf) 
	& = \sum_{j=1}^p \int_\Omega (\nabla v_j)^T \nabla w_j  
	   + \V_j(x)\, v_j w_j + \rho_j(\varphibf)\, v_j w_j \, \dx\\
    & = \int_\Omega \trace\big((\nabla \vbf)^T \nabla \wbf\big)  
	   + (\vbf \circ \wbf)\V(x) + (\varphibf \circ \varphibf)K(\vbf \circ \wbf)^T \, \dx. \notag
\end{align}
The following proposition establishes some useful properties of $a_{\varphibf}$. 
\begin{proposition}[Properties of $a_\varphibf$]\label{prop:aphiGarding}
Let Assumption~{\bf A1} be fulfilled and let $\varphibf\in \Hsp$. Then the bilinear form~$a_\varphibf$ defined in \eqref{eq:aphi} is symmetric, bounded, and satisfies a~G{\aa}rding inequality. 
\end{proposition}
\begin{proof}
The proofs of symmetry and boundedness are straightforward. 
In order to derive the G{\aa}rding inequality, we first observe that the interpolation inequality \cite[App.~B.2.h]{Eva10} and the Sobolev embedding inequality \eqref{eq:qSobolev} with $q=6$ imply that 
\begin{equation}\label{eq:estL4}
\|\vbf\|_{[L^4(\Omega)]^p} 
\leq \|\vbf\|_{\Lsp}^{1/4} \|\vbf\|_{[L^6(\Omega)]^p}^{3/4}
\leq C_6^{3/4} \|\vbf\|_{\Lsp}^{1/4} \|\vbf\|_{\Hsp}^{3/4}.
\end{equation}
Using the Cauchy--Schwarz and Young's inequalities~\cite[App.~B.2]{Eva10} as well as \eqref{eq:estL4}, we then obtain 
\begin{align*}
\|\vbf\|_{\Hsp}^2 
    & = a_{\varphibf}(\vbf, \vbf)
        - \int_\Omega (\varphibf \circ \varphibf) K (\vbf \circ \vbf)^T \,\dx 
    \leq 
    a_{\varphibf}(\vbf, \vbf)
        + \|K\|_2\, \|\varphibf \circ \varphibf\|_\Lsp \|\vbf \circ \vbf\|_\Lsp \\  
    & \leq a_{\varphibf}(\vbf, \vbf)
        + C_6^{3/2} \|K\|_2\, \|\varphibf\|_{[L^4(\Omega)]^p}^2 \|\vbf\|_{\Lsp}^{1/2} \|\vbf\|_{\Hsp}^{3/2} \\
    & \leq a_{\varphibf}(\vbf, \vbf) +
     \tfrac{1}{4}\, C_6^{6}\, \|K\|_2^4 \|\varphibf\|_{[L^4(\Omega)]^p}^8 \|\vbf\|_{\Lsp}^2 + \tfrac{3}{4}\, \|\vbf\|^2_{\Hsp}.
\end{align*}
This implies the G{\aa}rding inequality 
\begin{equation}\label{eq:Garding}
a_{\varphibf}(\vbf, \vbf)
    \geq \tfrac{1}{4}\, \|\vbf\|_{\Hsp}^2 -c_\varphibf \|\vbf\|_\Lsp^2
\end{equation}
with $c_\varphibf=\tfrac{1}{4}\,C_6^{6}\, \|K\|^4_2\, \|\varphibf\|_{[L^4(\Omega)]^p}^8$.
\end{proof}

Note that for the G{\aa}rding inequality~\eqref{eq:Garding}, we do not need any assumption on the matrix~$K$.

\begin{remark}[Coercivity of $a_\varphibf$ for non-negative $K$]\label{re:non-negative}
If all entries of the interaction matrix $K$ are non-negative, then the integrand $(\varphibf \circ \varphibf) K (\vbf \circ \vbf)^T$ in the nonlinear term of $a_\varphibf(\vbf, \vbf)$ is non-negative almost everywhere. Hence, $a_\varphibf$ is coercive with a coercivity constant of $1$, which means that for all $\vbf \in \Hsp$, we have $a_{\varphibf}(\vbf, \vbf) \geq \|\vbf\|_\Hsp^2$.
\end{remark}

The bilinear form $a_\varphibf$ defines the \emph{Gross--Pitaevskii Hamiltonian} $\calA_\varphibf\colon \Hsp\to \Hsp^*$ given by
\[
\langle \calA_\varphibf \vbf,\wbf\rangle = a_\varphibf(\vbf,\wbf)  
\qquad \text{for all }\vbf,\wbf\in \Hsp,
\]
where $\langle\,\cdot\,,\cdot\,\rangle$ denotes the duality pairing on $H^*\times \Hsp$. 
In particular, we have $\Drm \calE(\varphibf)[\wbf] = \langle \calA_\varphibf \varphibf, \wbf\rangle$ for all $\wbf \in \Hsp$. Due to the additive structure of $a_\varphibf$ in \eqref{eq:aphi}, the operator $\calA_\varphibf$ can be represented as 
\begin{subequations}
\label{eq:calA}
\begin{equation}\label{eq:calAsum}
\langle \calA_\varphibf \vbf,\wbf\rangle 
= \sum_{j=1}^p \langle \calA_{\varphibf,j} v_j,w_j\rangle
\qquad \text{for all }\vbf,\wbf\in \Hsp,
\end{equation}
where the component operators $\calA_{\varphibf,j}\colon H_0^1(\Omega)\to H^{-1}(\Omega)$ are given by
\begin{equation}\label{eq:calAj}
    \langle\calA_{\varphibf,j}v_j,w_j\rangle 
    = \int_\Omega (\nabla v_j)^T \nabla w_j  
	+ \V_j(x)\, v_j w_j
    + \rho_j(\varphibf)\, v_j w_j \dx, \quad j=1,\ldots,p.
\end{equation}
\end{subequations}
Note that, for fixed $\varphibf$, the operator $\calA_{\varphibf,j}$ is simply the weak form of the Laplacian shifted by a~bounded potential $V_j+\rho_j(\varphibf)$. Formula~\eqref{eq:calAsum} shows that the operator $\calA_\varphibf$ acts component-wise on a $p$-frame,~i.e., 
\begin{equation}\label{eq:calAcomp}
\calA_\varphibf\vbf=(\calA_{\varphibf,1}v_1,\ldots,\calA_{\varphibf,p}v_p).
\end{equation}
Hence, for all $\vbf \in \Hsp$ and $\Lambda \in \Dp$, we have $\calA_\varphibf(\vbf\Lambda) = (\calA_\varphibf\vbf)\Lambda$. Such properties are also present in the Hartree--Fock and Kohn--Sham Hamiltonians studied in \cite{AltPS22,SchRNB09}. Unlike these Hamiltonians, however, the individual operators~$\calA_{\varphibf,j}$ in \eqref{eq:calAcomp} vary between the components, making the analysis of multicomponent BECs more difficult.
%
%
\subsection{Existence of a ground state}\label{sec:setting:groundstate}
The following theorem generalises the existence of ground states for two-component BECs~\cite[Th.~2.6]{BaoC11} to the multicomponent case. However, even in the case $p=2$, there are slight differences in the assumptions. Specifically, \cite{BaoC11} considers the entire space $\R^d$, whereas we focus on bounded spatial domains.
\begin{theorem}[Existence of a~ground state]
\label{thm:existence}
Let Assumptions~{\bf A1} and~{\bf A2} be satisfied. Then there exists a~ground state $\varphibf_* \in \OB$, i.e., a global minimiser of the constrained minimisation problem \eqref{eq:min}.
\end{theorem}
\begin{proof} The proof essentially follows the approach of \cite[Lem.~2]{CanCM10}, to which we refer for details that we omit here for the sake of brevity. Due to Assumptions~{\bf A1} and~{\bf A2}, the energy functional $\calE$ in~\eqref{eq:energy_short} is bounded from below (by zero). Hence, there exists a~minimising sequence~$\varphibf^n$ in~\mbox{$\OB \subset \Hsp$}. Since this sequence is bounded in~$\Hsp$, there exists a weak limit~$\varphibf^\infty \in \Hsp$ such that (up to a subsequence) for each component $j=1,\dots,p$, we have 
\[
    \varphi_j^n \rightharpoonup \varphi_j^\infty \enskip\text{ in } H^1_0(\Omega).
\]
The compact embedding $H^1_0(\Omega) \hookrightarrow L^2(\Omega)$ then implies strong convergence $\varphi_j^n \rightarrow \varphi_j^\infty$ of the components in $L^2(\Omega)$ and, hence,
\[
    \|\varphi_j^\infty\|^2_{L^2(\Omega)}
    = \lim_{n \to \infty} \|\varphi_j^n\|^2_{L^2(\Omega)}
    = N_j, \qquad j=1,\ldots,p,  
\]
proving the feasibility of $\varphibf^\infty \in \OB$. The weak lower semi-continuity of the functional $\calE$ on~$\Hsp$ shows that~$\varphibf_* = \varphibf^\infty$ is a global minimiser. 
\end{proof}
Note that in Theorem~\ref{thm:existence}, Assumption~{\bf A2} may be replaced by the condition that $K$ is element-wise non-negative, since this as well guarantees the boundedness (from below) and weak lower semi-continuity of $\calE$.
\begin{remark}[Non-negative ground state]
If $\varphibf_* \in \OB$, then $|\varphibf_*| \in \OB$, where~$|\varphibf_*|$ is the $p$-frame with absolute values in all the components. Moreover, we have \mbox{$\calE(|\varphibf_*|) = \calE(\varphibf_*)$}. This implies that with $\varphibf_*$ also $|\varphibf_*|$ is a ground state. Hence, there exists a~non-negative ground state $\varphibf_* \geq 0$, where the inequality is understood to be component-wise.    
\end{remark}
%
%
\section{Nonlinear eigenvector problem and uniqueness of the ground state}
\label{sec:eigenvector}
This section aims to derive the first- and second-order necessary optimality conditions of the ground state. In particular, we will characterise the constrained critical points of the energy as solutions to the NLEVP~\eqref{eq:multiBEC} representing the critical points of the Lagrangian.
\subsection{Critical points and nonlinear eigenvector problem}
\label{sec:eigenvector:Hessian}
To establish a~relation between the coupled Gross--Pitaevskii eigenvalue problem~\eqref{eq:multiBEC} and the energy minimisation problem~\eqref{eq:min}, we consider the Lagrange functional 
\begin{equation}\label{eq:Lagrange}
	\calL(\varphibf,\Lambda) 
	= \calE(\varphibf) - \frac{1}{2} \trace\Big(\Lambda \big(\out{\varphibf}{\varphibf}-N\big) \Big)
\end{equation}
with a~Lagrange multiplier $\Lambda\in\Dp$. The directional derivative of $\calL$ with respect to $\varphibf\in\Hsp$ along $\wbf\in\Hsp$ is given by 
\begin{equation*}
	\Drm_\varphibf \calL(\varphibf,\Lambda)[\wbf]
	 = \Drm\calE(\varphibf)[\wbf]-(\varphibf\,\Lambda,\wbf)_\Lsp = \langle\calA_\varphibf\,\varphibf,\wbf\rangle - (\varphibf\,\Lambda,\wbf)_\Lsp. 
\end{equation*}
A $p$-frame $\varphibf \in \Hsp$ is called a \emph{constrained critical point} of the energy functional~$\,\calE$ if $\varphibf \in \OB$ is feasible and if it is a critical point of the Lagrangian, i.e., there exists a Lagrange multiplier $\Lambda \in \Dp$ such that $\Drm_\varphibf \calL(\varphibf,\Lambda)[\wbf] = 0$ for all $\wbf \in \Hsp$, or, equivalently,
\begin{equation}\label{eq:nlevp}
 \langle\calA_\varphibf\,\varphibf, \wbf\rangle = (\varphibf\,\Lambda,\wbf)_\Lsp \qquad \text{for all } \wbf \in \Hsp.
\end{equation}
Note in particular that a ground state $\varphibf_*$ is a constrained critical point of $\calE$. The computation of critical points for the constrained energy minimisation problem~\eqref{eq:min} is therefore linked to the NLEVP~\eqref{eq:nlevp}.
Using \eqref{eq:calAsum}, the NLEVP~\eqref{eq:nlevp} can equivalently be written component-wise as
\begin{equation}\label{eq:levp1}
    \langle \calA_{\varphibf,j} \varphi_j, w_j \rangle 
    = \lambda_j (\varphi_j, w_j)_{L^2(\Omega)} \qquad \text{for all } w_j \in H_0^1(\Omega).
\end{equation}
This equals the weak formulation of the coupled Gross--Pitaevskii eigenvalue problem previously stated in~\eqref{eq:multiBEC}.
Given a pair consisting of a $p$-frame $\varphibf$ and a diagonal matrix $\Lambda = \diag(\lambda_1, \ldots, \lambda_p)$ that solves~\eqref{eq:nlevp}, we call $\varphibf$ an \emph{eigenvector} corresponding to \emph{eigenvalues} $\lambda_1, \ldots, \lambda_p$.

For a~constrained critical point $\varphibf \in \OB$, the corresponding Lagrange multiplier~$\Lambda$ can be determined as $\Lambda = \out{\calA_\varphibf\,\varphibf}{\varphibf}N^{-1}$. This particularly implies that the diagonal entries $\lambda_j$ of $\Lambda$ are given by the Rayleigh quotients
\[
\lambda_j
= \frac{\langle \calA_{\varphibf,j} \varphi_j, \varphi_j \rangle}{(\varphi_j, \varphi_j)_{L^2(\Omega)}}, \qquad j=1,\ldots,p,
\]
 of the component operators~$\calA_{\varphibf,j}$.
\begin{lemma}[Eigenvalues and eigenfunctions of $\calA_{\varphibf,j}$]
\label{lem:spectralGap}
Let Assumption~{\bf A1} be fulfilled. Then for fixed $\varphibf\in \OB$ and \mbox{$j=1,\ldots,p$}, the operator $\calA_{\varphibf,j}$ introduced in~\eqref{eq:calAj} has 
an $L^2$-or\-tho\-go\-nal basis of eigenfunctions $v_{j,1},v_{j,2},\ldots$ in~$H_0^1(\Omega)$ corresponding to the real eigenvalues \mbox{$\lambda_{1}(\calA_{\varphibf,j})< \lambda_{2}(\calA_{\varphibf,j})\leq \ldots$}  ordered increasingly. Furthermore, the smallest eigenvalue $\lambda_{1}(\calA_{\varphibf,j})$ is simple and the corresponding eigenfunction satisfies $|v_{j,1}|>0$ in $\Omega$.
\end{lemma}
\begin{proof}
For fixed $\varphibf\in \OB$, the operator~$\calA_{\varphibf,j}$ corresponds to the linear Schr\"odinger operator with external potential $\V_j+\rho_j(\varphibf)$. By Proposition~\ref{prop:aphiGarding}, we know that there exists a~constant $c_\varphibf>0$ such that $\widehat{a}_{\varphibf,j}(v ,w)  
    := \langle \calA_{\varphibf,j} v,w\rangle 
        + c_\varphibf\, (v ,w)_{L^2(\Omega)}$
defines a symmetric, bounded, and coercive bilinear form. This implies that $\calA_{\varphibf,j}=\calA_{|\varphibf|,j}$ has a~countably infinite number of real eigenvalues such that the smallest eigenvalue is simple. The corresponding eigenfunctions form an $L^2$-orthogonal basis in $H_0^1(\Omega)$. Moreover, according to the proof of~\cite[Lem.~2]{CanCM10}, $v_{j,1}$ is H\"older continuous and can be chosen to be positive in $\Omega$. 
\end{proof}
The subsequent result connects the eigenvalues of the component operators with the diagonal elements of the Lagrange multiplier corresponding to the ground state.
\begin{proposition}[Characterisation of Lagrange multipliers as eigenvalues of components]\label{prop:LagrEigVal}
Let Assumptions~\textup{\textbf{A1}} and~\textup{\textbf{A2}} be fulfilled and let $\varphibf_*\in\OB$ with $\varphibf_*\geq 0$ be a~ground state of the energy functional $\,\calE$ in \eqref{eq:energy} with the Lagrange multiplier \mbox{$\Lambda_*=\diag(\lambda_{*,1},\ldots,\lambda_{*,p})$}. Then, for all $j=1,\ldots,p$,
the smallest eigenvalue $\lambda_{1}(\calA_{\varphibf_*,j})$ of $\calA_{\varphibf_*,j}$ coincides with $\lambda_{*,j}$, i.e.,
$\lambda_{1}(\calA_{\varphibf_*,j})=\lambda_{*,j}$. 
\end{proposition}
\begin{proof}
For a~ground state $\varphibf_*=(\varphi_{*,1},\ldots,\varphi_{*,p})\in\OB$ with $\varphibf_*\geq 0$, we obtain from~\eqref{eq:levp1} that $\lambda_{*,j}$ is an~eigenvalue of the operator $\calA_{\varphibf_*,j}$ and $\varphi_{*,j}$ is the corresponding eigenfunction.  Then the required result follows from the proof of \cite[Lem.~2]{CanCM10}.  
\end{proof}
\subsection{Uniqueness of the ground state}
Given that the ground state as a global minimiser of \eqref{eq:min} is also an eigenvector of the NLEVP~\eqref{eq:nlevp} with component-wise minimality properties as established in Proposition~\ref{prop:LagrEigVal}, we can now state a~uniqueness result for the ground state of a~multicomponent BEC. 
\begin{theorem}[Uniqueness of the~ground state]
\label{thm:uniquedensity}
Let Assumptions~{\bf A1} and~{\bf A2} be fulfilled. Then the ground state is unique up to the global signs of its components. We can, therefore, choose the ground state $\varphibf_*$ to be component-wise positive in $\Omega$.
\end{theorem}
\begin{proof}
Similarly to the single- and two-component cases in \cite{CanCM10,LieSY00} and \cite[Lem.~2.2]{BaoC11}, respectively, we introduce a convex problem by considering the modified functional $\calE(\sqrt{\bullet})$ acting on the convex set of densities
\begin{equation*}
\Upsilon = \big\{\rhobf\in \Lsp \enskip:\enskip \rhobf\geq 0 \text{ and } \sqrt{\rhobf}\in \OB\big\},
\end{equation*}
where $\rhobf\geq 0$ means that all components of $\rhobf$ are non-negative point-wise almost everywhere, and the square root is taken component-wise.
Given $\rhobf^{[1]},\rhobf^{[2]}\in \Upsilon$, define $\varphibf^{[i]} = \sqrt{\rhobf^{[i]}}\in \OB$ for $i=1,2$.  
First note that $\varphibf=(\varphi_1,\ldots,\varphi_p)$ with components
\[
    \varphi_j 
    = \sqrt{ \alpha\, \rho^{[1]}_j + (1-\alpha)\, \rho^{[2]}_j}, 
    \qquad\alpha\in(0,1),
\]
satisfies $\out{\varphibf}{\varphibf}=N$ and, hence, $\varphibf\in\OB$. 
Then \cite[Lem.~A.1]{LieSY00} implies that 
\[
    \big\|\nabla \varphi_j\big\|^2 
    \le \alpha\, \big\|\nabla \varphi^{[1]}_j \big\|^2 + (1-\alpha)\, \big\|\nabla \varphi^{[2]}_j \big\|^2, \qquad j=1,\ldots, p.
\]
Therefore, the first term of $\calE(\sqrt{\bullet})$ involving the gradients is convex on $\Upsilon$. In addition, the second term containing the potential $V$ is linear and the last term with the positive definite interaction matrix $K$ is quadratic.  
As the sum of three (strictly) convex functionals, $\calE(\sqrt{\bullet})$ is therefore strictly convex on the convex set~$\Upsilon$. For~$\varphibf_*\ge 0$, its density $\rhobf_* = \varphibf_* \circ \varphibf_* \in \Upsilon$ minimises the functional~$\calE(\sqrt{\bullet})$ on $\Upsilon$. The strict convexity of~$\calE(\sqrt{\bullet})$ on the convex feasible set~$\Upsilon$, as shown above, implies the uniqueness of the density. By Lemma~\ref{lem:spectralGap}, the components of any ground state have no sign change in $\Omega$ by a maximum principle. This, in turn, shows that the only remaining degree of freedom is a~global sign change in the components, which is the strongest possible form of uniqueness according to Remark~\ref{rm:invariance}.
\end{proof}
It is worth noting that Assumption~\textbf{A2} is essential for the uniqueness (up to sign) of the ground state. Replacing it by the condition that $K$ is element-wise non-negative does not guarantee uniqueness; see \cite[Rem.~2.1(ii)]{BaoC11}.

Based on Proposition~\ref{prop:LagrEigVal}, under Assumptions~\textup{\textbf{A1}} and~\textup{\textbf{A2}}, the unique positive ground state $\varphibf_*$ is an eigenvector of the NLEVP~\eqref{eq:nlevp} corresponding to the $p$ smallest eigenvalues $\lambda_{*,1}, \ldots, \lambda_{*,p}$ of the component operators~$\calA_{\varphibf,1},\ldots,\calA_{\varphibf,p}$, respectively. There exists a~spectral gap in the following sense.

\begin{remark}[Spectral gap]\label{lm:GroundEigVal}
Generalising the proof for the single-component BEC given in~\textup{\cite[p.~117]{CanCM10}} to the multicomponent case, one can show that if $\varphibf \in \OB$ is an eigenvector of the NLEVP~\eqref{eq:nlevp} corresponding to the eigenvalues $\lambda_1, \ldots, \lambda_p$, then either $\lambda_j = \lambda_{*,j}$ for $j = 1, \ldots, p$ and $\varphibf = \varphibf_*\Sigma_{\pm 1}$, or there exists an index $j_0 \in \{1, \ldots, p\}$ such that $\lambda_{j_0} > \lambda_{*,j_0}$.
\end{remark}
\subsection{Second-order derivatives and necessary optimality condition}
The second-order (directional) derivative of $\,\calE$ at $\varphibf$ in the direction of $\vbf, \wbf \in \Hsp$ can be computed as
\begin{align}\label{eq:HessE}
	\Drm^2 \calE(\varphibf)[\vbf,\wbf] 
	& = \lim_{t\to 0} \tfrac{1}{t}\, \big\langle \calA_{\varphibf+t\vbf}(\varphibf+t\vbf)-\calA_\varphibf\, \varphibf ,\wbf  \big\rangle 
 = \big\langle \calA_\varphibf \,\vbf +\calB_{\varphibf}(\vbf,\varphibf), \wbf\big\rangle,
\end{align}
where the operator $\calB_\varphibf\colon \Hsp\times \Hsp \to \Hsp^*$ is given by 
\begin{subequations}
\label{eq:calB}
\begin{equation}\label{eq:calBsum}
\langle\calB_{\varphibf}(\vbf,\ubf), \wbf\rangle 
=  \sum_{i,j=1}^p \langle\calB_{\varphibf,ij}(v_i,u_j), w_j\rangle 
\qquad \text{for all }\ubf,\vbf,\wbf\in \Hsp
\end{equation}
with 
\begin{equation}\label{eq:calBij}
\langle\calB_{\varphibf,ij}(v_i,u_j), w_j\rangle
= 2 \int_\Omega \kappa_{ij}\, \varphi_i v_i \, u_j w_j \dx, \qquad i,j=1,\ldots,p. 
\end{equation}
\end{subequations}
Similar to (\ref{eq:energy_short}), we can also express $\calB_\varphibf$ compactly as
\begin{align}\label{eq:opB}
    \big\langle \calB_\varphibf(\vbf, \ubf), \wbf \big\rangle 
    = 2 \int_\Omega (\varphibf \circ \vbf) K (\ubf \circ \wbf)^T \, \dx.
\end{align}

We now turn to the second-order derivative of the Lagrangian $\calL$ in \eqref{eq:Lagrange} with respect to the first variable. It reads 
\begin{align}
	\Drm^2_{\!\varphibf\varphibf} \calL(\varphibf,\Lambda)[\vbf,\wbf]
	= \Drm^2 \calE(\varphibf)[\vbf,\wbf] - (\vbf\,\Lambda,\wbf)_{\Lsp}
	= \langle\calA_\varphibf\,\vbf,\wbf\rangle +\langle\calB_\varphibf (\vbf,\varphibf),\wbf\rangle - (\vbf\, \Lambda,\wbf)_\Lsp
    \label{eq:HessL}
\end{align}
for all $\vbf,\wbf \in \Hsp$.
  
The following theorem  shows that the second-order derivative of the Lagrangian with respect to the first argument at a~ground state and the corresponding Lagrange multiplier is positive definite on the~tangent space to the set of constraints described by~\eqref{eq:constraints}. This result is stronger than a~second-order necessary optimality condition which requires only positive semi-definiteness. 
\begin{theorem}
\label{thm:second-orderOpt}
Let Assumptions~\textup{\textbf{A1}} and \textup{\textbf{A2}} be fulfilled and let \mbox{$\varphibf_*\in\OB$} with $\varphibf_*\geq 0$ be a~ground state of $\,\calE$ with the corresponding Lagrange multiplier $\Lambda_* = \out{\varphibf_*}{\calA_{\varphibf_*}\varphibf_*}N^{-1}$. Then
\[
    \Drm^2_{\!\varphibf\varphibf} \calL(\varphibf_*,\Lambda_*)[\zbf,\zbf] > 0
\]
for all non-zero $\zbf\in \Hsp$ such that $\out{\varphibf_*}{\zbf}=0$. 
\end{theorem}
\begin{proof}
Let $\varphibf_*=(\varphi_{*,1},\ldots,\varphi_{*,p})\geq 0$ be a~ground state of the functional $\,\calE$ with the Lagrange multiplier $\Lambda_*=\diag(\lambda_{*,1},\ldots,\lambda_{*,p})$. Furthermore, let $v_{j,1}, v_{j,2}, \ldots$ be the eigenfunctions of $\calA_{\varphibf_*,j}$ which form an~$L^2$-orthogonal basis of~$H_0^1(\Omega)$ and let \mbox{$\lambda_{1}(\calA_{\varphibf_*,j}) < \lambda_{2}(\calA_{\varphibf_*,j}) \leq \ldots$} be the corresponding eigenvalues. By Proposition~\ref{prop:LagrEigVal}, we know that $v_{j,1}=\pm\varphi_{*,j}$ and \mbox{$\lambda_{1}(\calA_{\varphibf_*,j}) = \lambda_{*,j}$} for $j=1,\ldots,p$. For any $\zbf\in\Hsp$, we obtain from Assumption~{\bf A2} and the relation \eqref{eq:opB} that $\langle\calB_{\varphibf_*}(\zbf,\varphibf_*),\zbf\rangle \geq 0$. Together with 
\eqref{eq:HessL}, this yields
\begin{align*}
    \Drm^2_{\!\varphibf\varphibf} \calL(\varphibf_*,\Lambda_*)[\zbf,\zbf]
    & = \langle\calA_{\varphibf_*}\zbf,\zbf\rangle 
        + \langle\calB_{\varphibf_*}(\zbf,\varphibf_*),\zbf\rangle
        - (\zbf\, \Lambda_*,\zbf)_\Lsp 
    \geq \langle\calA_{\varphibf_*}\zbf,\zbf \rangle 
        - \trace\, \out{\zbf\, \Lambda_*}{\zbf}.
\end{align*}
If $\out{\varphibf_*}{\zbf}=0$, then each component $z_j$ of $\zbf$
can be represented as $z_j=\sum_{\ell\geq 2} \alpha_{j,\ell}v_{j,\ell}$ with $\alpha_{j,\ell}\in\R$. For non-zero $\zbf$, this finally implies that
\begin{align*}
    \Drm^2_{\!\varphibf\varphibf} \calL(\varphibf_*,\Lambda_*)[\zbf,\zbf]
    &\geq \sum_{j=1}^p \langle\calA_{\varphibf_*,j}z_j,z_j\rangle  - \lambda_{*,j} (z_j,z_j)_{L^2(\Omega)} \notag\\
    &= \sum_{j=1}^p\, \sum_{\ell\ge 2} \lambda_{\ell}(\calA_{\varphibf_*,j})\,(\alpha_{j,\ell}v_{j,\ell},z_j)_{L^2(\Omega)}  - \lambda_{1}(\calA_{\varphibf_*,j}) (z_j,z_j)_{L^2(\Omega)} \notag\\
    &\geq \sum_{j=1}^p \big(\lambda_{2}(\calA_{\varphibf_*,j})-\lambda_{1}(\calA_{\varphibf_*,j}) \big)(z_j,z_j)_{L^2(\Omega)} > 0. \qedhere
\end{align*}
\end{proof}
It immediately follows from the proof of Theorem~\ref{thm:second-orderOpt} that $\Drm^2_{\!\varphibf\varphibf} \calL(\varphibf_*,\Lambda_*)$ is coercive on the tangent space to the constraint set with the coercivity constant $\gamma = \min\limits_{j=1,\ldots,p}(\lambda_{2}(\calA_{\varphibf_*,j})-\lambda_{1}(\calA_{\varphibf_*,j})) > 0$.
%
%
\section{Generalised oblique manifold}
\label{sec:oblique}
The ground state problem of multicomponent BECs is naturally connected to the infinite-dimen\-sio\-nal generalised oblique manifold $\OB$ defined in \eqref{eq:obligman}. In the finite-dimensional case, such a~manifold with $N = I_p$ has been investigated in \cite{AbsG06,BouAMC20,SelAGQCLA12} in the context of the approximate joint diagonalisation problem, which frequently arises in data science and engineering. Here, we extend the results on the manifold geometry from these works to the infinite-dimensional setting.

First of all, note that $\OB$ can be represented as the product manifold
\begin{equation}\label{eq:prodS}
    \OB = \Sphere_{N_1} \times \ldots \times \Sphere_{N_p}
\end{equation} 
with the spheres $\,\Sphere_{N_j} = \{\varphi \in H_0^1(\Omega) \, : \, \|\varphi\|_{L^2(\Omega)}^2 = N_j\}$ for $j = 1, \ldots, p$. 
The geometric structure of~$\,\Sphere_{N_j}$, which is a particular case of the infinite-dimensional Stiefel manifold, has been studied in~\cite{AltPS22}.
%
%
\subsection{Tangent and normal spaces}
The generalised oblique manifold $\OB$ admits a submanifold structure on the ambient space $H$.
\begin{proposition}\label{prop:submanifold}
The generalised oblique manifold $\OB$ is a~closed embedded subma\-ni\-fold of the Hilbert space $\Hsp$ of co-dimension~$p$.
\end{proposition}
\begin{proof}
We consider the map $F\colon \Hsp \to \Dp$ given by 
$F(\varphibf) = \out{\varphibf}{\varphibf} - N$. 
Then \mbox{$\OB = F^{-1}(0_p)$}. This implies that $\OB$ is closed, since it is the preimage of the closed set $\{0_p\}$ for the continuous map $F$. We now show that $F$ is a~submersion or, equivalently, that for all $\varphibf\in\OB$, the Fr\'echet derivative of $F$ at~$\varphibf$ given by \mbox{$\Drm \!F (\varphibf)[\vbf] = 2\, \out{\varphibf}{\vbf}$}
is surjective. Let $\Sigma \in\Dp$. Then, for \mbox{$\vbf = \frac{1}{2} \varphibf\,\Sigma N^{-1}\in \Hsp$}, we obtain 
$$
    \Drm \!F (\varphibf)[\vbf] 
    = \out{\varphibf}{\varphibf\,\Sigma N^{-1}} 
    = \Sigma.
$$ 
Thus, by the preimage theorem \cite[Th.~73.C]{Zei88}, $\OB$ is an embedded submanifold of $H$ and its co-dimension coincides with $\dim(\Dp) = p$.	
\end{proof}

The kernel of $\Drm \!F (\varphibf)$ defines the \emph{tangent space} $T_{\varphibf}\,\OB$ to $ \OB$ at $\varphibf$. It is given~by 
\begin{align*}
T_{\varphibf}\,\OB = \Kern\big(\Drm \!F (\varphibf)\big) = \big\{ \zbf \in \Hsp\enskip:\enskip \out{\varphibf}{\zbf} =0_p\big\}.
\end{align*}
In order to introduce a Riemannian metric on $\OB$, we consider a~symmetric, bounded, and coercive bilinear form $g_\varphibf\colon \Hsp\times \Hsp\to \R$ which defines an~inner product on~$H$. Furthermore, let \mbox{$\calG_\varphibf\colon \Hsp\to \Hsp^*$} be the linear bounded operator given by 
$$
  \langle \calG_\varphibf \vbf, \wbf\rangle 
  = g_\varphibf( \vbf, \wbf)\qquad \text{for all } \vbf, \wbf\in \Hsp.
$$
Due to the coercivity of~$g_\varphibf$, there exists the inverse $\calG_\varphibf^{-1}\colon \Hsp^* \to \Hsp$ of $\calG_\varphibf$, which is again linear. Note that due to the canonical inclusion $\Hsp \subset \Hsp^*$, the application of $\calG_\varphibf^{-1}$ to functions in $\Hsp$ is well-defined.
 
Since the generalised oblique manifold $\OB$ is an embedded submanifold of $H$, the restriction $g_\varphibf\colon T_{\varphibf}\,\OB\times T_{\varphibf}\,\OB\to \R$, provided that it depends smoothly on $\varphibf\in \Hsp$, defines a~Riemannian metric on~$\OB$ which turns $\OB$ into a~Riemannian manifold. Then the \emph{normal space} to $\OB$ at $\varphibf$ with respect to this metric is defined as 
\begin{align*} 
\big(T_{\varphibf}\,\OB\big)_{g_\varphibf}^\perp = \big\{ \etabf \in \Hsp\enskip:\enskip g_\varphibf(\etabf,\zbf) =0\enskip \text{for all }\zbf\in T_{\varphibf}\,\OB \big\}.
\end{align*}
Using the inverse operator $\calG_\varphibf^{-1}$, this space can be characterised as follows.

\begin{proposition}[Characterisation of the normal space]\label{prop:normal_space}
The normal space to the generalized oblique manifold $\OB$ at \mbox{$\varphibf\in\OB$} with respect to the metric~$g_\varphibf$ is given by 
$$
	\big(T_{\varphibf}\,\OB\big)_{g_\varphibf}^\perp = \big\{ \calG_\varphibf^{-1}\big(\varphibf\,\Sigma\big) \in \Hsp\enskip:\enskip \Sigma\in\Dp\big\}.
$$
\end{proposition}
\begin{proof}
For all $\zbf\in T_{\varphibf}\,\OB$ and $\Sigma\in\Dp$, we see that 
$$
	g_\varphibf( \calG_\varphibf^{-1}\big(\varphibf\,\Sigma\big),\zbf) 
	= (\varphibf\,\Sigma, \zbf)_{\Lsp} 
	= \trace\, \out{\varphibf\,\Sigma}{\zbf}
	= \trace\big(\Sigma\,\out{\varphibf}{\zbf}\big) 
	= 0.
$$
This implies that $\{ \calG_\varphibf^{-1}(\varphibf\,\Sigma)\in \Hsp \enskip:\enskip \Sigma\in\Dp \} \subseteq (T_{\varphibf}\,\OB )_{g_\varphibf}^\perp$. The equality follows since both spaces have dimension $p$.  
\end{proof}
The orthogonal projection $\calP_{\varphibf}\colon H\to T_{\varphibf}\,\OB$ with respect to the metric $g_\varphibf$ will be called the $g_{\varphibf}$-{\em orthogonal projection} and is characterised in terms of the inverse operator below.  
 \begin{proposition}\label{prop:proj}
	Let $\varphibf\in\OB$. Then the $g_{\varphibf}$-orthogonal projection onto the tangent space
	$T_{\varphibf}\,\OB$ is given by 
	\begin{equation}\label{eq:proj2}
		\calP_\varphibf(\ubf) = \ubf- \calG_\varphibf^{-1} \big(\varphibf \,\Sigma_{\varphibf,\ubf}\big), \qquad \ubf\in \Hsp,
	\end{equation}
where $\Sigma_{\varphibf,\ubf}\in\Dp$ is the~unique solution to the equation 
\begin{equation}\label{eq:Sigma}
\out{\varphibf}{\calG_\varphibf^{-1}( \varphibf \,\Sigma_{\varphibf,\ubf})}
=\out{\varphibf}{\ubf}.
\end{equation}
\end{proposition}
\begin{proof}
First, we show that equation \eqref{eq:Sigma} is uniquely solvable and provide its explicit solution. Let $\varphibf\in\OB$ and $\Sigma_{\varphibf,\ubf}=\diag(\sigma_1,\ldots,\sigma_p)$. For every component $\varphi_j$ of $\varphibf$, we define the \mbox{$p$-fra}me \mbox{$\widehat{\varphibf}_j=(0,\ldots,0,\varphi_j,0\ldots,0)$} for $j=1,\ldots,p$. Then, due to the linearity of $\calG_\varphibf^{-1}$, we obtain that $\calG_\varphibf^{-1}(\varphibf\,\Sigma_{\varphibf,\ubf})=\sum_{i=1}^p \sigma_i\, \calG_\varphibf^{-1}\widehat{\varphibf}_i$. In this case, equation \eqref{eq:Sigma} can equivalently be written as a~linear system $\widehat{G}\sigma = \widehat{u}$ with an unknown vector $\sigma=[\sigma_1,\ldots,\sigma_p]^T$ and given 
\begin{equation}\label{eq:hatG}
\widehat{G} = \begin{bmatrix} 
\big(\varphi_1, (\calG_\varphibf^{-1}\widehat{\varphibf}_1)_1\big)_{L^2(\Omega)} & \cdots & \big(\varphi_1, (\calG_\varphibf^{-1}\widehat{\varphibf}_p)_1\big)_{L^2(\Omega)} \\
\vdots & \ddots & \vdots \\
\big(\varphi_p, (\calG_\varphibf^{-1}\widehat{\varphibf}_1)_p\big)_{L^2(\Omega)} & \cdots & \big(\varphi_p, (\calG_\varphibf^{-1}\widehat{\varphibf}_p)_p\big)_{L^2(\Omega)} 
\end{bmatrix}, 
\qquad 
\widehat{u}=\begin{bmatrix}(\varphi_1,u_1)_{L^2(\Omega)}\\[1.5mm] \vdots \\[1.5mm]
(\varphi_p,u_p)_{L^2(\Omega)}\end{bmatrix}.
\end{equation}
We now verify that the matrix $\widehat{G}$ is nonsingular. Assume that there exists a vector $\alpha=[\alpha_1,\ldots,\alpha_p]^T$ such that $\widehat{G}\alpha=0$. This implies that 
\[
\big(\varphi_i, (\calG_\varphibf^{-1}(\varphibf\diag(\alpha)))_i\big)_{L^2(\Omega)}= \sum_{j=1}^p \alpha_j\, \big(\varphi_i, (\calG_\varphibf^{-1}\widehat{\varphibf}_j)_i\big)_{L^2(\Omega)} = 0, \qquad i=1,\ldots,p,
\]
and, hence, $\calG_\varphibf^{-1}(\varphibf\diag(\alpha))\in T_\varphibf\,\OB$.
On the other hand, due to Proposition~\ref{prop:normal_space}, we have $\calG_\varphibf^{-1}(\varphibf\diag(\alpha))\in \big(T_\varphibf\,\OB\big)^\perp_{g_\varphibf}$ and thus $\calG_\varphibf^{-1}(\varphibf\diag(\alpha))=\zerobf$. This yields that $\alpha=0$ which implies that $\widehat{G}$ is nonsingular and $\Sigma_{\varphibf,\ubf}=\diag(\sigma)=\diag(\widehat{G}^{-1}\widehat{u})$ is the unique solution of \eqref{eq:Sigma}.

Next, we prove that $\calP_\varphibf$ maps on $T_\varphibf\,\OB$. Indeed, it follows from \eqref{eq:Sigma} that
\[
\out{\varphibf}{\calP_\varphibf(\ubf)}
=\out{\varphibf}{\ubf- \calG_\varphibf^{-1}( \varphibf \,\Sigma_{\varphibf,\ubf})} =  0
\] 
and, hence, $\calP_\varphibf(\ubf)\in T_{\varphibf}\,\OB$ for all $\ubf\in \Hsp$.

Further, taking into account that for all $\ubf\in T_{\varphibf}\,\OB$, equation \eqref{eq:Sigma} has only the trivial solution $\Sigma_{\varphibf,\ubf}=0_p$, we 
obtain that
\begin{align*}
\calP_\varphibf\big(\calP_\varphibf(\ubf)\big)
 = \calP_\varphibf(\ubf)- \calG_\varphibf^{-1} \big(\varphibf \,\Sigma_{\varphibf,\calP_\varphibf(\ubf)}\big)
 = \calP_\varphibf(\ubf).
\end{align*}
This shows that $\calP_{\varphibf}$ is a~projection. Finally, it follows from Proposition~\ref{prop:normal_space} that 
\[\ubf-\calP_\varphibf(\ubf) = \calG_\varphibf^{-1} (\varphibf \,\Sigma_{\varphibf,\ubf})\in (T_{\varphibf}\,\OB)_{g_\varphibf}^\perp
\] 
implying the $g_\varphibf$-orthogonality property.
\end{proof}

The computation of the projection $\calP_\varphibf$ in \eqref{eq:proj2} requires multiple inversions of the operator~$\calG_{\varphibf}$, which can be computationally expensive in practice. To reduce the computational cost, we exploit the product structure of the generalised oblique manifold in~\eqref{eq:prodS} and, inspired by \cite{GaoPY23}, equip it with a Riemannian metric defined as a sum of the metrics $g_{\varphibf,j}$ on each component $\Sphere_{N_j}$, i.e.,
\begin{equation}\label{eq:blockdiagonal}
g_\varphibf^\times(\zbf, \ybf) 
   = g_{\varphibf,1}(z_{1}, y_{1}) + \ldots + g_{\varphibf,p}(z_{p}, y_{p})
\end{equation}
for $\zbf = (z_{1}, \ldots, z_{p})$, $\ybf = (y_{1}, \ldots, y_{p}) \in T_{\varphibf}\,\OB$. Here,
\[
    g_{\varphibf,j}(z_{j}, y_{j}) 
    = \langle \calG_{\varphibf,j}^\times\, y_{j}, z_j\rangle
    \qquad \text{for all } z_{j}, y_{j} \in T_{\varphi_j}\Sphere_{N_j}
\]
with appropriate operators $\calG_{\varphibf,j}^\times\colon H_0^1(\Omega) \to H^{-1}(\Omega)$, $j = 1, \ldots, p$, smoothly depending on $\varphibf \in \Hsp$. The metric $g_{\varphibf}^\times$ has an advantage over $g_{\varphibf}$ in that all computations can be performed component-wise, which facilitates parallel computing. As a consequence of Proposition~\ref{prop:proj}, by exploiting the additive structure of the metric $g_\varphibf^\times$ in~\eqref{eq:blockdiagonal}, we obtain the following representation for the $g_\varphibf^\times$-orthogonal projection onto $T_\varphibf\,\OB$.
\begin{corollary}\label{cor:timesProj}
Define the operator $\calG_\varphibf^\times\vbf=(\calG_{\varphibf,1}^\times v_1,\ldots,\calG_{\varphibf,p}^\times v_p)$ for $\vbf\in \Hsp$. Then the $g_\varphibf^\times$-or\-tho\-go\-nal projection onto the tangent space $T_\varphibf\,(\OB$ is given by $\calP_\varphibf^\times (\ubf) = \ubf- (\calG_\varphibf^\times)^{-1} \varphibf \,\Sigma_{\varphibf,\ubf}^\times$ with the diagonal matrix $\Sigma_{\varphibf,\ubf}^\times = \out{\varphibf}{(\calG_\varphibf^\times)^{-1}\varphibf}^{-1}\out{\varphibf}{\ubf}$.
\end{corollary}
\begin{proof}
The result follows from Proposition~\ref{prop:proj} by using the fact that for the metric $g_\varphibf^\times$ in~\eqref{eq:blockdiagonal}, the matrix~$\widehat{G}$ given in \eqref{eq:hatG} is diagonal with entries $(\varphi_j,(\calG_{\varphibf,j}^\times)^{-1}\varphi_j)_{L^2(\Omega)}$ on the diagonal.
\end{proof}
%
%
\subsection{Riemannian gradients and Hessians}

The \emph{Riemannian gradient} of a~Fr\'echet differentiable functional $\calE$ defined on the manifold $\OB$ with respect to the metric $g_\varphibf$ is the unique element $\grad \calE(\varphibf)\in T_\varphibf\,\OB$ which satisfies the condition
\[
g_\varphibf(\grad \calE(\varphibf), \zbf)=\Drm \calE(\varphibf)[\zbf] \qquad \text{ for all }\zbf\in T_\varphibf\,\OB.
\]
The Riemannian gradient of the energy functional $\,\calE$ in \eqref{eq:energy} with respect to $g_\varphibf$ can be represented~as 
\begin{equation}\label{eq:grad}
    \grad \calE(\varphibf) 
    = \calP_\varphibf\big(\calG_\varphibf^{-1} \Drm{\calE}(\varphibf)\big) 
    = \calG_\varphibf^{-1}\calA_\varphibf^{}\,\varphibf-\calG_\varphibf^{-1}\big(\varphibf\,\Sigma_{\varphibf,\calG_\varphibf^{-1}\calA_\varphibf\varphibf}\big),
\end{equation}
where $\Sigma_{\varphibf,\calG_\varphibf^{-1}\calA_\varphibf\varphibf}$ is as in Proposition~\ref{prop:proj} with $\ubf=\calG_\varphibf^{-1}\calA_\varphibf^{}\,\varphibf$. 

Consider the linear operator $\calR_\varphibf\colon \Hsp \to \Hsp^* $ defined by 
\[
    \calR_\varphibf\, \vbf 
    = \calA_\varphibf\,\vbf - \varphibf\out{\varphibf}{\calA_\varphibf\,\vbf}N^{-1}.
\]
Just as the Hamiltonian~$\calA_\varphibf$, the operator $\calR_\varphibf$ acts component-wise with the component operators denoted by~$\calR_{\varphibf,j}$. 
For all $\vbf\in\Hsp$ with $\calA_\varphibf\,\vbf\in\Hsp$, we have $\out{\varphibf}{\calR_\varphibf\, \vbf}=0$. This implies that \mbox{$\calR_\varphibf\, \vbf \in T_\varphibf\OB$}. Note that $\calR_\varphibf\, \varphibf$ gives the~residual of the NLEVP \eqref{eq:nlevp}. It can also be used to obtain an alternative representation for the Riemannian gradient, namely
\[
\grad \calE(\varphibf) 
= \calP_\varphibf\big(\calG_\varphibf^{-1}\calR_\varphibf\, \varphibf \big)
= \calG_\varphibf^{-1} \calR_\varphibf^{}\, \varphibf   - \calG_\varphibf^{-1}(\varphibf \, \Sigma_{\varphibf,\,\calG_\varphibf^{-1}\calR_\varphibf\varphibf}),  
\]
where $\Sigma_{\varphibf,\calG_\varphibf^{-1}\calR_\varphibf\varphibf}$ is as in Proposition~\ref{prop:proj} with $\ubf=\calG_\varphibf^{-1}\calR_\varphibf^{}\,\varphibf$. 

Let the Riemannian manifold $(\OB,g_\varphibf)$ be endowed with the~Riemannian connection $\bm{\nabla}$. The \emph{Riemannian Hessian} of a~twice differentiable functional $\calE$, denoted by $\Hess\calE(\varphibf)$, is the linear mapping $\Hess\calE(\varphibf)\colon T_\varphibf\,\OB\to T_\varphibf\,\OB$ given by
$$
\Hess \calE(\varphibf)[\zbf]=\bm{\nabla}_{\!\zbf} \grad\calE(\varphibf) \qquad \text{ for all }\zbf \in T_\varphibf\,\OB,
$$
where $\bm{\nabla}_{\!\zbf}$ denotes the covariant derivative along $\zbf$ with respect to the connection $\bm{\nabla}$. If the metric does not depend on $\varphibf$, the Riemannian Hessian admits the expression 
\begin{equation}\label{eq:RHess}
\Hess \calE(\varphibf)[\zbf] = \calP_\varphibf\big( \Drm\grad\calE(\varphibf)[\zbf]\big),
\end{equation}
which can be shown similarly to the finite-dimensional case \cite[Prop.~5.3.2]{AbsiMS08}.
%
%
\subsection{Choices of metric}\label{sec:metrics}
We introduce several metrics on the manifold $\OB$ and discuss the corresponding geometric concepts. In particular, using first- and second-order information of the energy, we construct specific metrics that possess a preconditioning effect when used in Riemannian optimisation.
%
\subsubsection{\texorpdfstring{$L^2$}--metric}\label{ssec:L2metric}
Let $\calG_\Lsp\colon \Hsp\to\Hsp^*$ denote the canonical identification operator defined by 
\[
\langle \calG_\Lsp\vbf,\wbf\rangle 
= (\vbf,\wbf)_\Lsp\qquad \text{for all } \vbf,\wbf\in\Hsp.
\]
This allows us to equip the generalised oblique manifold $\OB$ with the {\em $L^2$-metric}
\begin{equation}\label{eq:metricH}
g_{\Lsp}(\zbf,\ybf) 
= \langle\calG_\Lsp\zbf,\ybf\rangle \qquad \text{for all }\zbf,\ybf\in T_\varphibf\,\OB,
\end{equation}
which is independent of $\varphibf$. Note that this metric violates our assumptions above, as it is not coercive with respect to the $\Hsp$-norm and, in particular, the tangential space $T_\varphibf\OB$ is not complete with respect to the metric $g_{\Lsp}$. As a~consequence, the operator $\calG_\Lsp$ is not necessarily invertible and the Riemannian gradient with respect to $g_\Lsp$ may not exist for all $\varphibf$. In this subsection, however, we will always assume that all formulas are well-defined and, in particular, that $\varphibf\in \Hsp$ is such that $\calA_\varphibf\, \varphibf\in \Hsp$. 

As the metric $g_{\Lsp}$ is of the additive form~\eqref{eq:blockdiagonal}, the $g_L$-orthogonal projection onto $T_\varphibf\,\OB$ reads
\begin{equation}\label{eq:projH}
\calP_{\varphibf,L} (\ubf) = \ubf - \varphibf \out{\varphibf}{\ubf}N^{-1},
\qquad \ubf\in \Hsp.
\end{equation}
The Riemannian gradient of $\,\calE$ in~\eqref{eq:energy} with respect to $g_{\Lsp}$ then takes the form
\begin{align}\label{eq:gradH}
    \grad_{\Lsp} \calE(\varphibf) 
    = \calP_{\varphibf,L}\big(\Drm{\calE}(\varphibf)\big)
    = \calA_\varphibf\,\varphibf-\varphibf\,\out{\varphibf}{\calA_\varphibf\,\varphibf}N^{-1}
    = \calR_\varphibf\,\varphibf.
\end{align}
Furthermore, using \eqref{eq:HessE}, we obtain from \eqref{eq:RHess} the Riemannian Hessian 
\begin{align}
	\Hess_{\Lsp} \calE(\varphibf)[\zbf] 
       & = \calP_{\varphibf,L} \big(\calA_\varphibf\,\zbf 
		+ \calB_{\varphibf}(\zbf,\varphibf) 
        - \zbf \out{\varphibf}{\calA_\varphibf\,\varphibf}N^{-1}\big) \notag\\
		& =  \calA_\varphibf\,\zbf 
		+ \calB_{\varphibf}(\zbf,\varphibf)
  - \varphibf\out{\varphibf}{\calA_\varphibf\,\zbf + \calB_{\varphibf}(\zbf,\varphibf)}N^{-1} -\zbf\, \out{\varphibf}{\calA_\varphibf\,\varphibf}N^{-1}. \label{eq:HessH}
\end{align}
Note that for a~constrained critical point $\varphibf^*\in\OB$ of $\,\calE$ and the corresponding Lagrange multiplier~$\Lambda^*$, the Riemannian Hessian of $\,\calE$ and the second-order derivative of the Lagrange functional~$\calL$ are related to each other via $\Hess_{\Lsp} \calE(\varphibf^*)[\zbf] =\calP_{\varphibf^*\!,\,\Lsp}\big(\Drm^2_{\!\varphibf\varphibf}\calL(\varphibf^*,\Lambda^*)[\zbf]\big)$ for all tangent vectors $\zbf\in T_{\varphibf^*}\OB$. Then it follows from Theorem~\ref{thm:second-orderOpt} that for a~ground state $\varphibf_*\geq 0$, the Riemannian Hessian 
$\Hess_{\Lsp} \calE(\varphibf_*)$ is positive definite on the tangent space $T_{\varphibf_*}\OB$.
%
%
\subsubsection{Energy-adaptive metric}\label{sec:energymetric}
Consider the bilinear form $a_\varphibf$ introduced in~\eqref{eq:aphi}. If it is not coercive, then, based on Proposition~\ref{prop:aphiGarding}, we replace it by the coercive bilinear form
\[
\ahat_\varphibf(\vbf,\wbf)=a_\varphibf(\vbf, \wbf) + \widehat{c}_\varphibf (\vbf, \wbf)_\Lsp,
\]
where the~constant $\widehat{c}_\varphibf \geq c_\varphibf>0$ with $c_\varphibf$ as in the  G{\aa}rding inequality \eqref{eq:Garding} can be chosen to depend smoothly on $\varphibf$. This allows us to define an~alternative Riemannian metric on the generalised oblique manifold $\OB$, namely
\begin{equation}\label{eq:metr_a}
g_{\varphibf,a}(\zbf, \ybf)=\ahat_\varphibf(\zbf,\ybf) \qquad \text{for all } \zbf,\ybf\in T_\varphibf\,\OB.
\end{equation}
It is referred to as the \emph{energy-adaptive metric}. 

For $\varphibf\in\OB$, let $\calAhat_\varphibf:H\to H^*$ be defined by
\begin{equation}\label{eq:calAhat}
\langle\calAhat_\varphibf\vbf,\wbf\rangle=\ahat_\varphibf(\vbf,\wbf)\qquad \text{for all } \vbf,\wbf\in H.
\end{equation}
Then it follows from Corollary~\ref{cor:timesProj} with $\calG_\varphibf^\times = \calAhat_\varphibf$, which obviously acts component-wise, that the {\em $\ahat_\varphibf$-orthogonal projection} onto the tangent space $T_\varphibf\,\OB$ is given by 
$$
\calP_{\varphibf,a}(\ubf) = \ubf-\calAhat_\varphibf^{-1}\varphibf \out{\varphibf}{\ubf}
\out{\varphibf}{\calAhat_{\varphibf}^{-1}\varphibf}^{-1}.
$$
Observe that due to $\langle\calA_\varphibf \varphibf,\zbf\rangle = \langle\calAhat_\varphibf\varphibf, \zbf\rangle$ for all $\zbf \in T_\varphibf\OB$, we have 
\[
\calP_{\varphibf,a}(\calAhat_\varphibf^{-1}\calA_\varphibf\varphibf) = \calP_{\varphibf,a}(\varphibf).
\]
Therefore, using \eqref{eq:grad}, the Riemannian gradient of~$\,\calE$ in 
\eqref{eq:energy} with respect to the energy-adaptive metric~$g_{\varphibf,a}$ is determined as
\begin{align}\label{eq:grad_a} 
\grad_a\calE(\varphibf) 
& =\calP_{\varphibf,a}(\varphibf)
=\varphibf-\calAhat_\varphibf^{-1}\varphibf\, N\out{\varphibf}{\calAhat_\varphibf^{-1}\varphibf}^{-1}
= \calP_{\varphibf,a}\big(\calAhat_\varphibf^{-1}\calR_\varphibf^{}\,\varphibf\big). 
\end{align}
Using the relation $\calAhat_\varphibf^{-1}\varphibf = \big(\varphibf - \calAhat_\varphibf^{-1}\calR_\varphibf^{}\,\varphibf\big)\out{\varphibf}{\calAhat_\varphibf\,\varphibf}^{-1}N$, which follows directly from the definition of $\calR_\varphibf\,\varphibf$, it can further be represented as
\begin{equation}\label{eq:grad_a_R}
\grad_a\calE(\varphibf) 
= \varphibf - \big(\varphibf - \calAhat_\varphibf^{-1}\calR_\varphibf^{}\,\varphibf\big)\big(I_p - N^{-1}\out{\varphibf}{\calAhat_\varphibf^{-1}\calR_\varphibf^{}\,\varphibf}\big)^{-1}.
\end{equation}

The Riemannian Hessian with respect to $g_{\varphibf,a}$ can be calculated with the help of \cite[Th.~3.1]{Ngu23}. This, however, is beyond the scope of the present paper.
%
%
\subsubsection{Lagrangian-based metric} 
\label{ssec:LagrangeMetric}
In \cite{KreSV16, MisS16, ShuA23}, different strategies for constructing Riemannian metrics have been proposed in the context of Riemannian preconditioning for solving various optimisation problems on matrix manifolds. They all aim to speed up the convergence of Riemannian optimisation methods by exploiting the second-order  information of the cost functional and possibly constraints. Here, we adapt the approaches from \cite{MisS16} and \cite{GaoPY23} to the infinite-dimensional setting with product manifold structure and develop a~new family of Riemannian metrics by considering the second-order derivative of a~regularised Lagrange functional and neglecting its off-diagonal blocks with a~goal to trade-off between the computational cost and efficiency.   

Let us start with introducing a regularised Lagrange functional 
\[
\calL_{\omega}(\varphibf,\Lambda)
=\calE(\varphibf)- \frac{\omega}{2}\trace\Big(\Lambda \big(\out{\varphibf}{\varphibf}-N\big)\Big) 
\]
with an~appropriately chosen regularisation parameter $\omega \geq 0$. Note that for $\omega=1$, $\calL_{\omega}$ coincides with the Lagrangian $\calL$ in \eqref{eq:Lagrange}. The second-order derivative of $\calL_\omega$ with respect to $\varphibf$  has the form
\begin{equation}\label{eq:L2HessLreg}
\Drm_{\varphibf\varphibf}^2\calL_\omega(\varphibf,\Lambda)[\vbf] 
= \calA_\varphibf \vbf+\calB_\varphibf(\vbf,\varphibf)- \omega\,\vbf\,\Lambda, \qquad \vbf\in\Hsp,
\end{equation}
with $\calA_\varphibf$ and $\calB_\varphibf$ given in \eqref{eq:calA} and \eqref{eq:calB}, respectively. We define an~operator $\calG_{\varphibf,\omega}\colon H\to H^*$ acting component-wise $\calG_{\varphibf,\omega}\vbf=(\calG_{\varphibf,\omega,1}v_1,\ldots,\calG_{\varphibf,\omega,p}v_p)$, where  
\[
\calG_{\varphibf,\omega,j} v_j 
=  \calA_{\varphibf,j}v_j+\calB_{\varphibf,jj}(v_j,\varphi_j)-\omega\,\lambda_{j}v_j, \qquad j=1,\ldots, p,
\]
are obtained from \eqref{eq:L2HessLreg} by neglecting $\calB_{\varphibf,ij}$ with $i\neq j$ and setting the diagonal elements of~$\Lambda$ to \mbox{$\lambda_{j}=\langle\calA_{\varphibf,j}\varphi_j,\varphi_j\rangle/N_j$}, $j=1,\ldots,p$. 
Choosing $\omega$ such that the bilinear form $g_{\varphibf,\omega}(\vbf,\wbf)=\langle\calG_{\varphibf,\omega}\vbf,\wbf\rangle$ is coercive, we can define a~new metric on the generalised oblique manifold $\OB$~via  
\begin{equation}\label{eq:Lagr_metric}
g_{\varphibf,\omega}(\zbf,\ybf) 
= \langle\calG_{\varphibf,\omega}\,\zbf,\ybf\rangle
= \sum_{j=1}^p \langle\calG_{\varphibf,\omega,j}z_{j},y_{j}\rangle \qquad 
\text{for all } \zbf, \ybf\in T_{\varphibf}\,\OB.
\end{equation}
Due to the additive structure of this metric, the corresponding Riemannian gradient of the energy functional $\,\calE$ is given by 
\[
\grad_{\omega} \calE(\varphibf) = 
\calG_{\varphibf,\omega}^{-1}\calR_\varphibf^{}\,\varphibf- \calG_{\varphibf,\omega}^{-1}\,\varphibf\,\out{\varphibf}{\calG_{\varphibf,\omega}^{-1}\calR_\varphibf^{}\,\varphibf}\out{\varphibf}{\calG_{\varphibf,\omega}^{-1}\,\varphibf}^{-1}. 
\]

In summary, we see that different metrics result in distinct Riemannian gradients. The choice of the metric, as we will see in what follows, significantly influences the performance of the Riemannian optimisation schemes. A carefully chosen metric enables efficient optimisation by ensuring that the geometry of the generalised oblique manifold~$\OB$ is respected in every optimisation step. 
%
%
\section{Riemannian optimisation methods}
\label{sec:methods}
This section presents several Riemannian optimisation schemes for solving the energy minimisation problem \eqref{eq:min}. To facilitate this, we need a~retraction, which allows movement in a~tangent direction while remaining on the manifold. This can be realised by using a~component-wise normali\-sation operator $\calN\colon \Hsp\to\OB$ defined~as
\[
\calN(\vbf) = \vbf \out{\vbf}{\vbf}^{-1/2} N^{1/2}, \qquad \vbf \in \Hsp.
\]
Clearly, for all $\varphibf \in \OB$ and $\zbf \in T_\varphibf\,\OB$, $\calN(\varphibf+\zbf)$ is well-defined, as the matrix $\out{\varphibf + \zbf}{\varphibf + \zbf} = N + \out{\zbf}{\zbf}$ is invertible.
Moreover, for the origin $\zerobf_\varphibf \in T_\varphibf\,\OB$, we have $\calN(\varphibf+\zerobf_\varphibf) = \varphibf$, and for all $\zbf \in T_\varphibf\,\OB$, the local rigidity condition $\frac{{\rm d}}{{\rm d}t} \calN(\varphibf+t\zbf)\big|_{t=0} = \zbf$ is satisfied. This shows that the mapping $(\varphibf,\zbf)\mapsto \calN(\varphibf+\zbf)$ on the tangent bundle of $\OB$ indeed defines a~retraction on this manifold. 
%
%
\subsection{Preconditioned Riemannian gradient descent methods}\label{sec:methods:RGD}
A simple first-order optimisation scheme is the gradient descent method. In the context of Riemannian optimisation, this iterative procedure consists of the following steps: move in the tangent space along a~descent direction given by the negative Riemannian gradient with a certain step size and then retract the resulting tangent vector back onto the manifold. In a general setting, the {\em Riemannian gradient descent} (RGD) {\em method} is described in Algorithm~\ref{alg:RGD}.

\begin{algorithm}[ht]
		\caption{Riemannian gradient descent method}
		\label{alg:RGD}
	\begin{algorithmic}
			\Require Metric $g_\varphibf$ on $\OB$, 
                    initial guess $\varphibf_0\in\OB$.
            \vspace{0.5em}
			\For{$k=0, 1,\ldots$} 
			\State Compute a~search direction $\zbf_k=-\grad \calE(\varphibf_k)$. 
            \State Choose a~step size $\tau_k>0$. 
			\State Update $\varphibf_{k+1} = \calN(\varphibf_k+\tau_k \zbf_k)$.
			\EndFor
	\end{algorithmic}
\end{algorithm}
The convergence of the RGD method is strongly influenced by the step size, which can be determined using, e.g., the non-monotone line search procedure combined with the alternating Barzilai--Borwein step size strategy \cite{WenY13,ZhaH04}. It should also be noted that the RGD method depends on the choice of retraction and, most importantly, on the selected metric. In what follows, we briefly describe the resulting RGD schemes for the energy minimisation problem~\eqref{eq:min} based on the metrics considered in Section~\ref{sec:metrics} and highlight relationships between them.
%
%
\subsubsection{Preconditioned \texorpdfstring{$L^2$}--Riemannian gradient descent method}\label{ssec:pl2RGD}
The choice of the $L^2$-metric $g_{\Lsp}$ defined in \eqref{eq:metricH} in Algorithm~\ref{alg:RGD} yields the $L^2$-RGD method which can be interpreted as a power iteration with shifting. In electronic structure calculations, this method is also known as the direct minimisation algorithm \cite{PayTAAJ92,SchRNB09}. 
However, due to its explicit nature, it generally requires preconditioning which leads to the iteration
\[
	\varphibf_{k+1} = \calN\Big(\varphibf_k -\tau_k\,\calP_{\varphibf_k,L}\big(\calC_{\varphibf_k}^{-1} \big( \calA_{\varphibf_k}\varphibf_k - \varphibf_k\out{\varphibf_k}{\calA_{\varphibf_k}\varphibf_k}N^{-1} \big) \big) \Big)
\]
with a preconditioner $\calC_{\varphibf_k}$. Formally, this scheme requires a sufficiently smooth starting value to be well-defined. 
Different physics-based preconditioners related to the Laplace and Thomas–Fermi approximations have been developed in the literature \cite{AntD14a,AntLT17}. Alternatively, following \cite{AltPS22}, we propose the preconditioner \mbox{$\calC_{\varphibf_k}=\calAhat_{\varphibf_k}$}. Then using expression~\eqref{eq:projH}, we obtain the {\em energy-precon\-di\-tio\-ned $L^2$-RGD method} 
\begin{equation}\label{eq:pL2RGD}
	\varphibf_{k+1} = \calN\Big(\varphibf_k -\tau_k\, \big(\varphibf_k-\calAhat_{\varphibf_{k}}^{-1}\varphibf_kN\out{\varphibf_k}{\calAhat_{\varphibf_k}^{-1}\varphibf_k}^{-1} \big) \Theta_{\varphibf_k} \Big)
 \end{equation}
 with the diagonal matrix $\Theta_{\varphibf_k} = N^{-1}\out{\varphibf_k}{\calAhat_{\varphibf_k}^{-1}\varphibf_k}\out{\varphibf_k}{\calAhat_{\varphibf_k}\varphibf_k}N^{-1}$.

\subsubsection{Energy-adaptive Riemannian gradient descent method}\label{ssec:eaRGD}
Endowing $\OB$ with the energy-adaptive metric $g_{\varphibf,a}$ defined in \eqref{eq:metr_a} and using \eqref{eq:grad_a}  and~\eqref{eq:grad_a_R}, we obtain the \emph{energy-adaptive Riemannian gradient descent {\em (eaRGD)} method} 
\begin{align}
	\varphibf_{k+1} & = \calN\Big(\varphibf_k -\tau_k\, \big(\varphibf_k-\calAhat_{\varphibf_{k}}^{-1}\varphibf_k N \out{\varphibf_k}{\calAhat_{\varphibf_k}^{-1}\varphibf_k}^{-1} \big) \Big) \label{eq:eaRGD}\\
 & = \calN\Big(\varphibf_k -\tau_k\, \big(\varphibf_k - \big(\varphibf_k - \calAhat_{\varphibf_k}^{-1}\calR_{\varphibf_k}^{}\varphibf_k\big)\big(I_p - N^{-1}\out{\varphibf_k}{\calAhat_{\varphibf_k}^{-1}\calR_{\varphibf_k}^{}\varphibf_k}\big)^{-1}\big) \Big)\label{eq:eaRGD_R}.
\end{align}
This method can be interpreted as the inverse subspace iteration for $\calAhat_\varphibf$ enhanced with adaptive damping, where $\tau_k$ represents the damping parameter; see~\cite{Hen23,HenP20} for the case $p=1$. For coercive~$a_\varphibf$ and $\tau_k=1$, it reduces to the inverse subspace iteration method, known also as the \mbox{$\calA$-method}. Very similar schemes, such as the discrete normalised gradient flows~\cite{BaoD04}, result from an implicit--explicit time discretisation of the continuous $L^2$-Riemannian gradient flow; see also~\cite{DanP17,HenJ23}.

The eaRGD method \eqref{eq:eaRGD} differs from the energy-preconditioned $L^2$ -RGD iteration \eqref{eq:pL2RGD} only by the diagonal matrix $\Theta_{\varphibf_k}$. 
Since at a critical point, $\calAhat_\varphibf\,\varphibf = \varphibf\,(\Lambda+\widehat{c}_\varphibf I_p)$ and $\calAhat_\varphibf^{-1}\,\varphibf = \varphibf\,(\Lambda+\widehat{c}_\varphibf I_p)^{-1}$, the RGD methods~\eqref{eq:pL2RGD} and~\eqref{eq:eaRGD} are even asymptotically equivalent. In the numerical experiments reported in Section~\ref{sec:numerics},  we implement \eqref{eq:eaRGD_R}, as it explicitly includes the residual $\calR_{\varphibf_k}\varphibf_k$.
%
%
\subsubsection{Lagrangian-based Riemannian gradient descent method}\label{ssec:LgrRGD}
By considering the Lagran\-gian-based metric $g_{\varphibf,\omega}$ as defined in \eqref{eq:Lagr_metric}, we derive the \emph{Lagrangian-based Riemannian gradient descent {\em (LgrRGD)} method} 
\begin{equation}\label{eq:LgrRGD}
	\varphibf_{k+1} = \calN\Big(\varphibf_k -\tau_k\,\big(\calG_{\varphibf_k,\omega_k}^{-1} \calR_{\varphibf_k}^{}\varphibf_k - \calG_{\varphibf_k,\omega_k}^{-1}\varphibf_k	\,\Sigma_{\varphibf_k,\omega_k} \big) \Big)
\end{equation}
with the diagonal matrix $\Sigma_{\varphibf_k,\omega_k} 
 =\out{\varphibf_k}{\calG_{\varphibf_k,\omega_k}^{-1}\calR_{\varphibf_k}^{}\varphibf_k}
\out{\varphibf_k}{\calG_{\varphibf_k,\omega_k}^{-1}\varphibf_k}^{-1}$ and the regularisation parameter $\omega_k>0$ guaranteeing that $\calG_{\varphibf_k,\omega_k}$ is coercive and, in particular, invertible. A~possible choice for $\omega_k$ can be found in~\cite{MisS16}. In practice, however, we observed that after initialising the algorithm sufficiently close to a~ground state (e.g.~by applying a~few steps of eaRGD first), simply choosing $\omega_k = 1$ usually works well for the case where the interaction matrix $K$ has only positive entries. Comparing the computational complexity of the eaRGD method~\eqref{eq:eaRGD_R} and the LgrRGD method~\eqref{eq:LgrRGD}, we see that each iteration of the latter requires the solution of two linear operator equations with the same operator $\calG_{\varphibf_k,\omega}$ and the right-hand sides $\calR_{\varphibf_k}\varphibf_k$ and~$\varphibf_k$, while the former involves the solution of a single equation with the operator~$\calA_{\varphibf_k}$ and the right-hand side~$\calR_{\varphibf_k}\varphibf_k$.
%
%
\subsection{Alternating Riemannian gradient descent method}
In the above RGD schemes, the metric is of additive form \eqref{eq:blockdiagonal}, so that the components $\varphi_{k+1,j}$ of the new iteration~$\varphibf_{k+1}$ can be calculated independently, giving rise to possible parallelisation.  
Alternatively, inspired by the alternating approach presented in \cite{HuaY24}, the components of the new iteration can be computed sequentially using the components most recently calculated. This leads to the {\em alternating RGD method} presented in Algorithm~\ref{alg:alternatingRGD}, which is formulated for a~general metric $g_\varphibf^\times$ of the form~\eqref{eq:blockdiagonal} induced by an~operator~$\calG_\varphibf^\times$ acting component-wise. Taking either $\calG_\varphibf^\times=\calAhat_\varphibf$ or $\calG_\varphibf^\times=\calG_{\varphibf,\omega}$, we get the corresponding alternating versions of the eaRGD and LgrRGD methods, where for the eaRGD method, we can use \eqref{eq:eaRGD} or \eqref{eq:eaRGD_R} to reduce the number of linear systems to be solved per alternating step by one. The alternating LgrRGD method is similar to the alternating Newton--Noda method presented in~\cite{HuaY24}, but with a different choice for the Lagrange multiplier.

\begin{algorithm}[ht]
		\caption{Alternating Riemannian gradient descent method}
		\label{alg:alternatingRGD}
	\begin{algorithmic}
			\Require Operator $\calG_\varphibf^\times$ defining the metric $g_\varphibf^\times$ on~$\OB$, initial guess \mbox{$\varphibf_0\in\OB$}.
            \vspace{0.5em}
            \For{$k=0, 1,\ldots$}
                \State Set $\varphibf_{k+1}=\varphibf_k$. 
                \For{$j=1,\ldots, p$}
                    \State Solve $\calG_{\varphibf_{k+1},j}^\times \,v= \calR_{\varphibf_{k+1},j}\,\varphi_{k,j}$ for $v$.
                    \State Solve $\calG_{\varphibf_{k+1},j}^\times \,w=\varphi_{k,j}$ for $w$. 
                	\State Compute $z=v-\frac{(\varphi_{k,j},v)_{L^2(\Omega)}}{(\varphi_{k,j},w)_{L^2(\Omega)}}\, w$. 
            \State Choose a~step size $\tau_k>0$. 
			\State Update $\varphi_{k+1,j} = \sqrt{N_j}\,\frac{\varphi_{k,j}+\tau_k z}{\|\varphi_{k,j}+\tau_k z\|_{L^2(\Omega)}}$.  
			\EndFor
   \EndFor
	\end{algorithmic}
\end{algorithm}
%
%
\subsection{Riemannian Newton and Newton-type methods}
\label{sec:methods:Newton}
In the \emph{Riemannian Newton~{\em(RN)} method}, for given iterate $\varphibf_{k}\in\OB$, the search direction \mbox{$\zbf_k\in T_{\varphibf_k}\OB$} is computed by solving the Newton equation 
\begin{equation}\label{eq:NewtonEq}
\Hess\calE(\varphibf_k)[\zbf_k]=-\grad\calE(\varphibf_k),
\end{equation}
leading to the next iterate $\varphibf_{k+1} = \calN\big(\varphibf_k+\zbf_k\big)$. The resulting method is given in Algorithm~\ref{alg:Newton}.
\begin{algorithm}[ht]
		\caption{Riemannian Newton method}
		\label{alg:Newton}
	\begin{algorithmic}
			\Require Metric $g_\varphibf$ on $\OB$, 
                     initial guess \mbox{$\varphibf_0\in\OB$}.
            \vspace{0.5em}
			\For{$k=0, 1,\ldots$}             
            \State Solve $\Hess\calE(\varphibf_k)[\zbf_k]=-\grad\calE(\varphibf_k)$ for $\zbf_k\in T_{\varphibf_k}\OB$. 
            \State Update $\varphibf_{k+1} = \calN\big(\varphibf_k+\zbf_k\big)$.
            \EndFor
	\end{algorithmic}
\end{algorithm}

In the $L^2$-metric, due to \eqref{eq:gradH} and \eqref{eq:HessH}, the Newton equation \eqref{eq:NewtonEq} takes the form
\begin{equation}\label{eq:NewtonEqL2}
\calP_{\varphibf_k,L}\big(\calA_{\varphibf_k}\zbf_k + \calB_{\varphibf_k}(\zbf_k,\varphibf_k) -\zbf_k \out{\varphibf_k}{\calA_{\varphibf_k}\varphibf_k}N^{-1}\big) = -\calR_{\varphibf_k}\varphibf_k.
\end{equation}
As in Section~\ref{ssec:pl2RGD}, it needs to hold $\calA_{\varphibf_k} \varphibf_k \in \Hsp$. Similarly to~\cite{AltPS24}, we can show that the resulting RN method is equivalent to the Lagrange--Newton method \cite{AltM95} and the Newton--Noda iteration \cite{DuL22} with modified update $\varphibf_{k+1}=\calN\big(\varphibf_k+\zbf_k\big)$ and $\Lambda_{k+1}=\out{\varphibf_{k+1}}{\calA_{\varphibf_{k+1}}\varphibf_{k+1}}N^{-1}$. 

It is well-known that Newton methods have a local nature. Hence, the convergence to a~minimal solution can only be expected in a~small neighbourhood of this solution, where the Hessian is positive definite. To circumvent this difficulty, inspired by the second-order derivative of the regularised Lagrangian $\calL_{\omega}$ given in \eqref{eq:L2HessLreg}, the Riemannian Hessian can be parametrised by $\omega\geq 0$ yielding the operator
\begin{equation*}
    \calH_{\varphibf,\omega} \zbf
    = \calP_{\varphibf,L}\big(\calA_{\varphibf}\zbf + \calB_{\varphibf}(\zbf,\varphibf) - \omega \zbf \out{\varphibf}{\calA_{\varphibf}\varphibf}N^{-1}\big), \qquad \zbf\in T_{\varphibf_k}\OB.
\end{equation*}
For $\omega = 1$, we recover the Riemannian Hessian $\Hess_{\Lsp}\calE(\varphibf)$. If $K$ satisfies Assumption~\textbf{A2} and, additionally, has non-negative entries, then for $\omega = 0$, the operator $\calH_{\varphibf,\omega}$ is positive definite on $T_{\varphibf_k}\OB$. Therefore, we can view $\omega$ as a~regularisation parameter. Solving the equation
\[
\calH_{\varphibf_k,\omega_k}\zbf_k=-\calR_{\varphibf_k}\varphibf_k
\]
instead of the Newton equation \eqref{eq:NewtonEqL2} leads to the  \textit{regularised Riemannian Newton {\em (regRN)} method}. In practice, we found that even for $\omega$ very close to $1$, the convergence radius of the regRN method is much larger than that of the RN iteration. Comparing the regRN method with the LgrRGD iteration~\eqref{eq:LgrRGD}, in the case of weak inter-component interactions, the operator~$\calG_{\varphibf,\omega}$ can be considered as an~approximation to $\calH_{\varphibf,\omega}$, and, hence, \eqref{eq:LgrRGD} can be interpreted as a Riemannian quasi-Newton method. It should also be noted that our regRN method differs from the adaptive regularised Newton method presented in \cite{TiaCWW20,WuWB17}, which approximates the Hessian by a quadratic functional. 

\medskip
In all the above methods, scalar parameters such as the shift $\widehat{c}_\varphibf$, the step size~$\tau_k$, or the regularisation parameter~$\omega_k$ can be chosen for each component individually, resulting in vector-valued parameters. Moreover, other Riemannian optimisation methods such as the Riemannian conjugate gradient, trust region, or quasi-Newton methods, see, e.g. \cite[Chap.~7 \&~8]{AbsiMS08}, can be extended to multicomponent infinite-dimensional BEC models in a~similar way by introducing a~vector transport between different tangent spaces.
%
%
\section{Convergence analysis of energy-adaptive Riemannian gradient descent}
\label{sec:convergence}
This section aims to demonstrate the reliability that arises from the problem-adaptive choice of the metric in RGD schemes. For the eaRGD method, we adapt the global qualitative and local quantitative convergence results for the single-component case~\cite{Hen23,HenP20,HenY25} to the multicomponent model under Assumptions~\textbf{A1} and~\textbf{A2}. This reliability makes this method a~prime choice for the globalisation of the other schemes, in particular, the potentially locally faster RN~methods.
%
%
\subsection{Global convergence} 
\label{ssec:GlobConv}
To verify the convergence, we start by noting that for any fixed $\varphibf \in \Hsp$, the coercive bilinear form~$\ahat_\varphibf$ introduced in \eqref{eq:metr_a}
induces the norm $\|\vbf\|_{\ahat_\varphibf} = \sqrt{\ahat_\varphibf(\vbf, \vbf)}$ on $\Hsp$. Due to \eqref{eq:Garding}, it is equivalent to the $\Hsp$-norm in the sense that there exists a~constant $C_\Hsp > 0$ (which may depend on $\varphibf$) such that for all $\vbf \in \Hsp$, 
\[
    C_\Hsp \|\vbf\|_{\ahat_\varphibf} \leq \|\vbf\|_{\Hsp} 
    \leq 2\, \|\vbf\|_{\ahat_\varphibf},
\]
provided that we choose $\widehat{c}_\varphibf \geq c_\varphibf$. In addition, for the point $\varphibf \in H$, we have $\|\varphibf\|_\Hsp \leq \|\varphibf\|_{\ahat_\varphibf}$ due to the positive definiteness of the interaction matrix $K$.

We now collect some technical results that are needed to prove the global convergence of the eaRGD method \eqref{eq:eaRGD}. 
\begin{lemma}\label{lm:estEnergy}
    Let Assumption~\textup{\textbf{A1}} be fulfilled. For all $\varphibf, \vbf\in\Hsp$ and $\Sigma = \diag(\sigma_1, \dots, \sigma_p)$ with $|\sigma_j| \leq 1$ for $j = 1, \dots, p$,
it holds that
\begin{align*}
            \calE(\varphibf) - \calE(\varphibf-\vbf) 
            & \geq a_\varphibf(\varphibf,\vbf) - \frac{1}{2}\, a_\varphibf(\vbf, \vbf) - C_K\|\vbf\|_\Hsp^2\Big(\|\varphibf\|_\Hsp^{}+ \tfrac{1}{2}\, \|\vbf\|_\Hsp\Big)^2,\\
            \calE(\varphibf) \,- \,\calE(\varphibf\,\Sigma) 
            & \,\geq - \frac{1}{2}\, C_K\|I_p - \Sigma^2\|_2\|\varphibf\|_\Hsp^4
            \end{align*}
with $C_{K}=C_4^4\,\|K\|_2$ and $C_4$ as in \eqref{eq:qSobolev} with $q=4$.
\end{lemma}
\begin{proof}
For $\varphibf = (\varphi_1, \dots, \varphi_p)$ and $\vbf\in\Hsp$, we have 
    \begin{align*}
      \calE(\varphibf) - \calE(\varphibf-\vbf)
    & = a_\varphibf(\varphibf,\vbf)
    - \frac{1}{2}\, a_\varphibf(\vbf, \vbf) \\
    &\qquad - \int_\Omega (\varphibf\circ\vbf)K(\varphibf\circ\vbf)^T
        - (\varphibf\circ\vbf)K(\vbf\circ\vbf)^T
        + \frac{1}{4}\, (\vbf\circ\vbf)K(\vbf\circ\vbf)^T\dx
    \end{align*}
as well as  
    \begin{align*}
    \calE(\varphibf) - \calE(\varphibf\,\Sigma) 
    & = \sum_{j=1}^p \int_\Omega \frac{1}{2} \big(1 - \sigma_j^2\big)\big(\|\nabla\varphi_j\|^2 + V_j(x)|\varphi_j|^2\big) \dx\\
    & \qquad  + \frac{1}{4} \int_\Omega (\varphibf \circ \varphibf)(K-\Sigma^2K\Sigma^2)(\varphibf \circ \varphibf)^T \dx. 
    \end{align*}
Using the Cauchy--Schwarz inequality and the Sobolev embedding \eqref{eq:qSobolev} with $q=4$ as well as the facts that $1 - \sigma_j^2 \geq 0$ for $j = 1, \dots, p$ and $\|K-\Sigma^2 K\Sigma^2\|_2\leq 2\,\|K\|_2\|I_p-\Sigma^2\|_2$, we obtain the claimed estimates.~
\end{proof}
It is easy to see that if $K$ has non-negative entries, we have $\calE(\varphibf) - \calE(\varphibf\,\Sigma) \geq 0$ in the above lemma.
\begin{lemma}\label{lm:positivity}
Let Assumptions~\textup{\textbf{A1}} and \textup{\textbf{A2}} be fulfilled. 
For $\varphibf\in\Hsp$, let~$\calAhat_\varphibf$ be the operator defined in~\eqref{eq:calAhat} with the corresponding coercive bilinear form $\ahat_\varphibf$. Then for any $\vbf\in\Hsp$ with $\vbf\geq 0$, the $p$-frame $\calAhat_\varphibf^{-1}\vbf$ only has non-negative components, i.e., $\calAhat_\varphibf^{-1}\vbf\geq 0$.
\end{lemma}
\begin{proof}
    First, we observe that $\calAhat_\varphibf^{-1}\vbf$ is the unique minimiser of the convex cost functional 
    \[
    \calF(\ubf)=\frac{1}{2}\, \ahat_{\varphibf}(\ubf,\ubf)-(\vbf,\ubf)_{\Lsp},\qquad \ubf\in\Hsp.
    \]
    Further, due to $\vbf\geq 0$ and $(\vbf,\ubf)_{\Lsp}\leq (\vbf,|\ubf|)_{\Lsp}$ for all $\ubf\in\Hsp$, we obtain that $\calF(|\calAhat_\varphibf^{-1}\vbf|)\leq \calF(\calAhat_\varphibf^{-1}\vbf)$. This immediately implies that
    $\calAhat_\varphibf^{-1}\vbf=|\calAhat_\varphibf^{-1}\vbf|\geq 0$.
\end{proof}
The following theorem shows that the iterates of the eaRGD method \eqref{eq:eaRGD} are uniformly bounded and the energy functional $\,\calE$ decays for sufficiently small step sizes.
Thereby, we will consider two options for the shift constants $\widehat{c}_{\varphibf_k}$, namely 
\begin{align}
     \widehat{c}_{\varphibf_k} 
     &= c_{\varphibf_k} = C_c^{}\, \|\varphibf_k\|_{[L^4(\Omega)]^p}^8, \qquad k\geq 0,\tag{\bf{S1}} \label{ass:S1} \\
	\widehat{c}_{\varphibf_k} 
    &= C_c^{}C_4^8\big(1 + \tfrac{1}{2}C_K\, \|\varphibf_0\|_\Hsp^2\big)^4\|\varphibf_0\|_\Hsp^8, \qquad k\geq 0.\tag{\bf{S2}} \label{ass:S2}
\end{align}
Here, $C_4$ and $C_6$ are the constants from the Sobolev embedding inequality \eqref{eq:qSobolev} with $q = 4$ and $q = 6$, respectively, $C_{K}=C_4^4\,\|K\|_2$, and $C_c^{} = \tfrac{1}{4}\,C_6^6\,\|K\|_2^4$.
While~\eqref{ass:S1} guarantees coercivity trivially, we will show that the constant shift~\eqref{ass:S2} relying only on the initial guess also suffices.
\begin{theorem}[Energy decay]\label{thm:EnergyDecay}
    Let Assumptions~{\bf A1} and {\bf A2} be fulfilled, and let the shift $\widehat{c}_{\varphibf_k}$ be defined either via~\eqref{ass:S1} or~\eqref{ass:S2}. Further, let     $\{\varphibf_k\}_{k=0}^\infty\subset \OB$ be a~sequence generated by the eaRGD method~\eqref{eq:eaRGD}. Then there exist constants $C_0>0$ and $0<C_\tau\leq 1$ depending on $\Omega$, $d$, $K$, $N$, and $\|\varphibf_0\|_{\ahat_{\varphibf_0}}$ such that for any step size $0< \tau_{\min}\leq \tau_k\leq\tau_{\max}\leq C_\tau$ and any~$k\geq 0$, the following relations hold:
    \vspace{0.4em}
    \begin{itemize}[itemsep=0.2em]
	   \item[\rm (i)] $\|\varphibf_k\|_{\ahat_{\varphibf_k}}\leq C_0$, 
	   \item[\rm (ii)] $\calE(\varphibf_k)-\calE(\varphibf_{k+1})\geq \frac{1}{2}\, \tau_{\min}\, \|\grad_a\calE(\varphibf_k)\|_{\ahat_{\varphibf_k}}^2$.
    \end{itemize}
    Additionally, \eqref{ass:S2} is a valid shift in the sense that the bilinear form $\ahat_{\varphibf_k}$ is coercive for all $k$. 
\end{theorem}
\begin{proof} We prove the estimates by mathematical induction adapting the proof of \cite[Th.~3.1]{ChenLLZ24} to the multicomponent setting and the energy-adaptive norm. 
We start with the case~\eqref{ass:S2}. Set
\[
    C_0=\big(2+C_K\|\varphibf_0\|_\Hsp^2\big)^{1/2}\|\varphibf_0\|_{\ahat_{\varphibf_0}}.
\]
For $k=0$, we obviously have $\|\varphibf_0\|_{\ahat_{\varphibf_0}}\leq C_0$ and \mbox{$\widehat{c}_{\varphibf_0} \geq c_{\varphibf_0}$}. Hence, the bilinear form $\ahat_{\varphibf_0}$ is coercive. 
Now assume that bound~(i) holds for $0,1,\ldots,k$. Define \mbox{$\gbf_k=(\gbf_{k,1}, \dots, \gbf_{k,p})=\grad_a\calE(\varphibf_k)$} and \mbox{$\widetilde{\varphibf}_k=\varphibf_k-\tau_k\gbf_k$}.
As the Riemannian gradient $\gbf_k$ is the~\mbox{$\ahat_{\varphibf_k}$-orthogonal} projection of $\varphibf_k$, we can estimate $\|\gbf_k\|_{\Hsp}\leq 2\,\|\gbf_k\|_{\ahat_{\varphibf_k}} \leq 2\,\|\varphibf_k\|_{\ahat_{\varphibf_k}} \leq 2\,C_0$. 
Due to the triangle inequality,  this implies that
$\|\widetilde{\varphibf}_k\|_{\Hsp}\leq \|\varphibf_k\|_{\ahat_{\varphibf_k}}+2\,\|\gbf_k\|_{\ahat_{\varphibf_k}} \leq 3\,C_0$.
Furthermore, we consider the diagonal matrix \mbox{$\Sigma_{\widetilde{\varphibf}_k} = \diag(\sigma_{k,1}, \dots, \sigma_{k,p})$} with 
\[
0<\sigma_{k,j} = \big(1 + \tau_k^2\,\|\gbf_{k,j}\|_{L^2(\Omega)}^2N_j^{-1}\big)^{-1/2} \leq 1, \qquad j = 1, \dots, p,
\] 
such that we have $\varphibf_{k+1}= \calN(\widetilde{\varphibf}_k) =\widetilde{\varphibf}_k\Sigma_{\widetilde{\varphibf}_k}$ and
$\|I_p - \Sigma_{\widetilde{\varphibf}_k}^2\|_2^{} \leq \tau_k^2\, C_N\, \|\gbf_k\|_{\ahat_{\varphibf_k}}^2$
with \mbox{$C_N = 4\,C_2^2\,\|N^{-1}\|_2^{}$} and $C_2$ as in \eqref{eq:qSobolev} with $q=2$. Then Lemma~\ref{lm:estEnergy} implies that
\begin{align*}
\calE(\varphibf_k) - \calE(\widetilde{\varphibf}_k)
& \geq \tau_k\, a_{\varphibf_k}(\varphibf_k,\gbf_k) 
    - \tau_k^2\,\big(\tfrac{1}{2}a_{\varphibf_k}(\gbf_k,\gbf_k) 
    + 16\, C_K C_0^2\, \|\gbf_k\|_{\ahat_{\varphibf_k}}^2\big), \\
\calE(\widetilde{\varphibf}_k) - \calE(\widetilde{\varphibf}_k\Sigma_{\widetilde{\varphibf}_k})
& \geq -\tfrac{81}{2}\tau_k^2\, C_K C_N C_0^4\, \|\gbf_k\|_{\ahat_{\varphibf_k}}^2.
\end{align*}
    
Using these estimates and the relations 
\begin{align*}
    a_{\varphibf_k}(\varphibf_k,\gbf_k) 
        = \ahat_{\varphibf_k}(\varphibf_k,\gbf_k)
        = \|\gbf_k\|_{\ahat_{\varphibf_k}}^2 
        = \ahat_{\varphibf_k}(\gbf_k,\gbf_k)
        \geq a_{\varphibf_k}(\gbf_k,\gbf_k),    
\end{align*}
we obtain that
\begin{align*}
    \calE(\varphibf_k)-\calE(\varphibf_{k+1})
    & \geq \calE(\varphibf_k) - \calE(\widetilde{\varphibf}_k)
    + \calE(\widetilde{\varphibf}_k) 
    - \calE(\widetilde{\varphibf}_k \Sigma_{\widetilde{\varphibf}_k}) \\
    & \geq \tau_k\, \|\gbf_k\|_{\ahat_{\varphibf_k}}^2 - \tau_k^2\, \big(\tfrac{1}{2}+16\,C_K C_0^2 + 
\tfrac{81}{2}C_K C_N C_0^4
    \big)\, \|\gbf_k\|_{\ahat_{\varphibf_k}}^2
    \geq \frac{\tau_{\min}}{2}\, \|\gbf_k\|_{\ahat_{\varphibf_k}}^2,
\end{align*}
provided that~$\tau_{\max}\big(\frac{1}{2}+16\,C_K C_0^2 + \tfrac{81}{2}C_NC_K C_0^4
\big)\leq \frac{1}{2}$, which is guaranteed by $\tau_{\max}\leq C_\tau$ for some sufficiently small~$C_\tau$. This proves (ii) and, in particular, $\calE(\varphibf_{k+1}) \leq \calE(\varphibf_{k})$. 
In order to show (i) for $k+1$, we estimate
\begin{align*}
\widehat{a}_{\varphibf_{k+1}}(\varphibf_{k+1},\varphibf_{k+1})
& 
= 2\,\calE(\varphibf_{k+1})+\frac{1}{2} \int_\Omega (\varphibf_{k+1}\circ\varphibf_{k+1})K(\varphibf_{k+1}\circ\varphibf_{k+1})^T\dx + \widehat{c}_{\varphibf_{k+1}}\|\varphibf_{k+1}\|_\Lsp^2 \\
&\leq 4\,\calE(\varphibf_{k+1}) + \widehat{c}_{\varphibf_{k+1}}\trace{N} \\
&\leq 4\,\calE(\varphibf_0) + \widehat{c}_{\varphibf_0}\trace{N} \\
&\leq 2\,\|\varphibf_0\|_\Hsp^2 + \|K\|_2\|\varphibf_0\|_{[L^4(\Omega)]^p}^4 + \widehat{c}_{\varphibf_0}\|\varphibf_0\|_{\Lsp}^2 \\
&\leq \big(2+C_K\|\varphibf_0\|_\Hsp^2\big)\|\varphibf_0\|_{\ahat_{\varphibf_0}}^2.
\end{align*}

A very similar argument implies $\|\varphibf_{k+1}\|^2_\Hsp \leq \big(1 + \tfrac{1}{2}C_K\|\varphibf_0\|_\Hsp^2\big)\|\varphibf_0\|_\Hsp^2$, which allows us to estimate 
\begin{equation*}
    c_{\varphibf_{k+1}} = C_c^{} \|\varphibf_{k+1}\|_{[L^4(\Omega)]^p}^8 \leq C_c^{} C_4^8 \|\varphibf_{k+1}\|_H^8 \leq  C_c^{}C_4^8\big(1 + \tfrac{1}{2}C_K\|\varphibf_0\|_\Hsp^2\big)^4\|\varphibf_0\|_\Hsp^8 = \widehat{c}_{\varphibf_{k+1}}.
\end{equation*}
Therefore, the bilinear form $\ahat_{\varphibf_{k+1}}$ is coercive and~\eqref{ass:S2} defines a valid shift. Moreover, we have $\|\varphibf_{k+1}\|_{\ahat_{\varphibf_{k+1}}}\leq C_0$.

For the shift~\eqref{ass:S1}, the claim follows analogously with the choice
\[
    C_0=\Big(1 + C_c^{}\trace{N}\,\|K^{-1}\|_2^2\big(2 + C_K\|\varphibf_0\|_\Hsp^2\big)\|\varphibf_0\|_\Hsp^2\Big)^{1/2}\Big(2 + C_K\|\varphibf_0\|_\Hsp^2\Big)^{1/2}\|\varphibf_0\|_{\ahat_{\varphibf_0}}
\]
and using that $\|\varphibf_{k+1}\|_{[L^4(\Omega)]^p}^4 \leq 4\,\|K^{-1}\|_2\,\calE(\varphibf_{k+1})$.
\end{proof}
We are now ready to prove the global convergence result for the eaRGD iteration. 
\begin{theorem}[Global convergence to the ground state]\label{thm:globalConv}
Let Assumptions~\textup{\textbf{A1}} and \textup{\textbf{A2}} be fulfilled and let the shift $\widehat{c}_{\varphibf_k}$ be defined via~\eqref{ass:S1} or~\eqref{ass:S2}. Further, let 
$\{\varphibf_k\}_{k=0}^\infty\subset \OB$ be a~sequence generated by the eaRGD method \eqref{eq:eaRGD} with step sizes \mbox{$0<\tau_{\min}\leq\tau_k\leq\tau_{\max}\leq C_\tau \leq 1$} as in Theorem~\textup{\ref{thm:EnergyDecay}}, and a~starting guess $\varphibf_0\in\OB$ with $\varphibf_0\geq 0$. Then this  sequence converges strongly in $\Hsp$ to the unique ground state~\mbox{$\varphibf_*\!\geq 0$}.
\end{theorem}
\begin{proof}
Starting with $\varphibf_0\geq 0$, we apply Lemma~\ref{lm:positivity} recursively. Given~$\varphibf_k\geq 0$, we get 
\begin{align*}
  \widetilde{\varphibf}_k 
  & = \varphibf_k-\tau_k\big(\varphibf_k-\calAhat_{\varphibf_k}^{-1}\varphibf_k N\out{\varphibf_k}{\calAhat_{\varphibf_k}^{-1}\varphibf_k}^{-1}\big)\\
  & = (1-\tau_k)\varphibf_k + \tau_k\calAhat_{\varphibf_k}^{-1}\varphibf_k N\out{\varphibf_k}{\calAhat_{\varphibf_k}^{-1}\varphibf_k}^{-1}\geq 0.
\end{align*}
This then implies $\varphibf_{k+1}=\widetilde{\varphibf}_k\out{\widetilde{\varphibf}_k}{\widetilde{\varphibf}_k}^{-1/2}N^{1/2}\geq 0$. Hence, the eaRGD iteration~\eqref{eq:eaRGD} preserves the non-negativity of the components.

By using Proposition~\ref{prop:LagrEigVal} and the uniform boundedness of $\{\varphibf_k\}$ shown in Theorem~\ref{thm:EnergyDecay}, similarly to \cite[Th.~4.9 \& Th.~5.1]{HenP20}, we can prove that there exists a~subsequence $\{\varphibf_{k_l}\}$ converging strongly in~$\Lsp$ and $\Hsp$ to a~ground state $\varphibf_*\in\OB$ with $\varphibf_*\geq 0$ and $\calE(\varphibf_*)=\lim\limits_{k\to\infty}\calE(\varphibf_k)$. Let $\Lambda_*$ denote the associated Lagrange multiplier. Then for any $\varphibf\in\OB$, we have 
\[
    a_{\varphibf_*}(\varphibf_*,\varphibf_*) 
    = (\varphibf_*\Lambda_*,\varphibf_*)_\Lsp
    = \sum_{j=1}^p N_j\lambda_{*,j}
    = \sum_{j=1}^p \|\varphi_{j}\|_{L^2(\Omega)}^2\lambda_{*,j}
    = (\varphibf\Lambda_*,\varphibf)_\Lsp.
\]
Hence, due to the definition of the smallest eigenvalues of the operators $\calA_{\varphibf_*,j}$ for $j=1,\ldots,p$, and the matrix $K$, it holds that
\begin{align*}
\calE(\varphibf) - \calE(\varphibf_*)  
&= \frac{1}{2}\, \big(a_{\varphibf_*}(\varphibf,\varphibf)-a_{\varphibf_*}(\varphibf_*,\varphibf_*)\big)\\
&\qquad + \frac{1}{4} \int_\Omega (\varphibf\circ\varphibf)K(\varphibf\circ\varphibf)^T
-2\,(\varphibf_*\circ\varphibf_*)K(\varphibf\circ\varphibf)^T
+(\varphibf_*\circ\varphibf_*)K(\varphibf_*\circ\varphibf_*)^T \dx\\
&= \frac{1}{2}\, a_{\varphibf_*}(\varphibf-\varphibf_*,\varphibf-\varphibf_*)
-\frac{1}{2}\, \big((\varphibf-\varphibf_*)\Lambda_*,\varphibf-\varphibf_*\big)_\Lsp \\
&\qquad+\frac{1}{4}\int_\Omega (\varphibf\circ\varphibf-\varphibf_*\circ\varphibf_*)K(\varphibf\circ\varphibf-\varphibf_*\circ\varphibf_*)^T\dx\\
&\geq \frac{1}{4\,\|K^{-1}\|_2}\, \int_\Omega (\varphibf\circ\varphibf-\varphibf_*\circ\varphibf_*)(\varphibf\circ\varphibf-\varphibf_*\circ\varphibf_*)^T\dx \\
&= \frac{1}{4\,\|K^{-1}\|_2}\, \sum_{j=1}^p\int_\Omega \big(|\varphi_j|^2-|\varphi_{*,j}|^2\big)^2\dx.
\end{align*}
The above relation with $\varphibf=\varphibf_k=(\varphi_{k,1},\ldots,\varphi_{k,p})$ implies that 
\[
\sum_{j=1}^p\int_\Omega \big(|\varphi_{k,j}|^2-|\varphi_{*,j}|^2\big)^2\dx \leq 4\,\|K^{-1}\|_2
\big(\calE(\varphibf_k) -\calE(\varphibf_*) \big) \underset{k \to \infty}{\longrightarrow} 0.
\]
Finally, taking into account that $\varphibf_k\geq 0$ and $\varphibf_*\geq 0$, we conclude that the whole sequence~$\{\varphibf_k\}$ converges strongly to $\varphibf_*$ in $\Lsp$ and in $\Hsp$. 
\end{proof}
%
%
\subsection{Local convergence}
\label{ssec:LocConv}
Our goal is to characterise the local convergence rate of the eaRGD iteration \eqref{eq:eaRGD} with constant step size $\tau_k=\tau$ and element-wise non-negative interaction matrix $K$. Due to the latter assumption, we do not need shifting. Therefore, the operator $\calAhat_{\varphibf}$ in \eqref{eq:eaRGD} can be replaced by $\calA_\varphibf$. Then the eaRGD iteration \eqref{eq:eaRGD} can equivalently be written as the fixed point iteration $\varphibf_{k+1}=\psi_\tau(\varphibf_k)$ with the mapping $\psi_\tau\colon H\to H$,
\[
\psi_\tau(\varphibf) 
    = \widetilde{\psi}_\tau(\varphibf)\out{\widetilde{\psi}_\tau(\varphibf)}{\widetilde{\psi}_\tau(\varphibf)}^{-1/2}N^{1/2},
\]
where $\widetilde{\psi}_\tau(\varphibf)
    = (1-\tau)\varphibf+\tau \calA_\varphibf^{-1}\varphibf N\out{\varphibf}{\calA_\varphibf^{-1}\varphibf}^{-1}$.
For a~ground state $\varphibf_*\in\OB$, we have $\widetilde{\psi}_\tau(\varphibf_*)=\varphibf_*=\psi_\tau(\varphibf_*)$ showing that $\varphibf_*$ is a~fixed point of $\psi_\tau$. Proving that the spectral radius of the Fr\'echet derivative $\Drm\psi_\tau(\varphibf_*)$, denoted by~$\varrho_*:=\varrho(\Drm {\psi}_\tau(\varphibf_*))$, is smaller than $1$, local convergence with rate~$\varrho_*$ follows from Ostrowski's theorem, see, e.g., \cite[Prop.~1]{AltHP21}.

To proceed, we introduce the mappings $\zeta(\varphibf)=\calA_\varphibf^{-1}\varphibf$ and $\Xi(\varphibf)=\out{\varphibf}{\zeta(\varphibf)}^{-1}$. Their directional derivatives at any $\varphibf\in H\setminus\{\zerobf\}$ along $\vbf\in H$ are given by
\begin{align*}
\Drm \zeta(\varphibf)[\vbf] 
& = \calA_\varphibf^{-1} \vbf- \calA_\varphibf^{-1}\calB_\varphibf(\vbf,\calA_\varphibf^{-1}\varphibf), \\
\Drm \Xi(\varphibf)[\vbf] & = -\big(\out{\vbf}{\calA_\varphibf^{-1}\varphibf}
+\out{\varphibf}{\calA_\varphibf^{-1}\vbf-\calA_\varphibf^{-1}\calB_\varphibf(\vbf,\calA_\varphibf^{-1}\varphibf)}\big)\out{\varphibf}{\calA_\varphibf^{-1}\varphibf}^{-2}.
\end{align*}
Since the linear operators $\Drm \zeta(\varphibf)$ and $\Drm \Xi(\varphibf)$ are both continuous in a~neighbourhood of $\varphibf_*$, which excludes the zero element, $\zeta(\varphibf)$ and $\Xi(\varphibf)$ are Fr\'echet differentiable in this neighbourhood. This implies that $\widetilde{\psi}_\tau(\varphibf)=(1-\tau)\varphibf+\tau\zeta(\varphibf)N\Xi(\varphibf)$ and $\psi_{\tau}(\varphibf)$ are also Fr\'echet differentiable in a~neighbourhood of $\varphibf_*$.
In particular, evaluating $\zeta(\varphibf)$, $\Xi(\varphibf)$ and their derivatives at the ground state $\varphibf_*$ and taking into account that
$\calA_{\varphibf_*}^{-1}\varphibf_*^{}=\varphibf_*^{}\Lambda_*^{-1}$ with $\Lambda_*=\out{\varphibf_*}{\calA_{\varphibf_*}\varphibf_*}N^{-1}$, we obtain that 
\begin{align*}
\zeta(\varphibf_*)& =\varphibf_*\Lambda_*^{-1}, &\; 
\Drm\zeta(\varphibf_*)[\vbf] &=\calA_{\varphibf_*}^{-1}\vbf-\calA_{\varphibf_*}^{-1}\calB_{\varphibf_*}(\vbf,\varphibf_*\Lambda_*^{-1}), \\
\Xi(\varphibf_*)& =\Lambda_*N^{-1}, \;
& \Drm\Xi(\varphibf_*)[\vbf] &= -
\big(\out{\vbf}{\varphibf_*}\Lambda_{*}^{-1}
+\out{\varphibf_*}{\calA_{\varphibf_*}^{-1}\vbf-\calA_{\varphibf_*}^{-1}\calB_{\varphibf_*}(\vbf,\varphibf_*\Lambda_*^{-1})}\big)\Lambda_{*}^2N^{-2}.
\end{align*} 
Then the derivatives of $\widetilde{\psi}_\tau$ and $\psi_{\tau}$ at $\varphibf_*$ are computed as 
\begin{align*}
\Drm \widetilde{\psi}_\tau(\varphibf_*)[\vbf] 
& = (1-\tau)\,\vbf+\tau\, \big(\Drm\zeta(\varphibf_*)[\vbf]N\Xi(\varphibf_*)+\zeta(\varphibf_*)N\Drm\Xi(\varphibf_*)[\vbf]\big), \\
\Drm \psi_\tau(\varphibf_*)[\vbf] 
& = \Drm \widetilde{\psi}_\tau(\varphibf_*)[\vbf] \out{\widetilde{\psi}_\tau(\varphibf_*)}{\widetilde{\psi}_\tau(\varphibf_*)}^{-1/2}N^{1/2} \\
& \qquad -\widetilde{\psi}_\tau(\varphibf_*)\out{\widetilde{\psi}_\tau(\varphibf_*)}{\widetilde{\psi}_\tau(\varphibf_*)}^{-3/2}\out{\Drm\widetilde{\psi}_\tau(\varphibf_*)[\vbf]}{\widetilde{\psi}_\tau(\varphibf_*)}N^{1/2}\\
& = \Drm \widetilde{\psi}_\tau(\varphibf_*)[\vbf] 
-\varphibf_*\out{\Drm\widetilde{\psi}_\tau(\varphibf_*)[\vbf]}{\varphibf_*}N^{-1}.
\end{align*}

We now consider the following linear eigenvalue problem: find an~eigenfunction $\vbf_{\tau,i}\in H$ and an~eigenvalue $\mu_{\tau,i}\in\mathbb{R}$ such that 
\begin{equation}\label{eq:EVP4Dpsi}
\Drm \psi_\tau(\varphibf_*)[\vbf_{\tau,i}] = \mu_{\tau,i}\vbf_{\tau,i}. 
\end{equation}
The following lemma shows that all eigenfunctions of $\Drm \psi_\tau(\varphibf_*)$ corresponding to non-zero eigenvalues belong to the tangent space $T_{\varphibf_*}\OB$.
\begin{lemma}\label{lem:eigsDpsi}
Let Assumptions \textup{\textbf{A1}} and \textup{\textbf{A2}} be fulfilled and let $\varphibf_*\in\OB$ be a ground state. Then all eigenfunctions $\vbf_{\tau,i}$ of \eqref{eq:EVP4Dpsi} corresponding to~$\mu_{\tau,i}\neq 0$ satisfy \mbox{$\vbf_{\tau,i}\in T_{\varphibf_*}\OB$}.
\end{lemma}
\begin{proof}
For all $\vbf\in H$, we have
\begin{align*}
\out{\Drm \psi_\tau(\varphibf_*)[\vbf]}{\varphibf_*} 
& = \out{\Drm \widetilde{\psi}_\tau(\varphibf_*)[\vbf]}{\varphibf_*} 
- \out{\varphibf_*\out{\Drm\widetilde{\psi}_\tau(\varphibf_*)[\vbf]}{\varphibf_*}N^{-1}}{\varphibf_*} \\
& = \out{\Drm \widetilde{\psi}_\tau(\varphibf_*)[\vbf]}{\varphibf_*} 
- \out{\varphibf_*}{\varphibf_*}\out{\Drm\widetilde{\psi}_\tau(\varphibf_*)[\vbf]}{\varphibf_*}N^{-1} = 0_p.
\end{align*}
From~\eqref{eq:EVP4Dpsi}, we hence obtain for all eigenfunctions $\vbf_{\tau,i}\in H$ of $\Drm \psi_\tau(\varphibf_*)$ that
\[
0_p 
= \out{\Drm \psi_\tau(\varphibf_*)[\vbf_{\tau,i}]}{\varphibf_*} 
= \mu_{\tau,i}\, \out{\vbf_{\tau,i}}{\varphibf_*}. 
\]
This yields $\out{\vbf_{\tau,i}}{\varphibf_*}=0$ or, equivalently, $\vbf_{\tau,i}\in T_{\varphibf_*}\OB$ whenever $\mu_{\tau,i}\neq 0$.
\end{proof}
Since zero eigenvalues are not relevant for the spectral radius of $\Drm \psi_\tau(\varphibf_*)$, we restrict the eigenvalue problem \eqref{eq:EVP4Dpsi} to $T_{\varphibf_*}\OB$. In order to estimate~$\varrho_*$, we first investigate an auxiliary linear eigenvalue problem: find \mbox{$\vbf_i\in T_{\varphibf_*}\OB$} with $\out{\vbf_i}{\vbf_i}=I_p$ and $\mu_i\in\mathbb{R}$ such that 
\begin{equation}\label{eq:auxEVP}
    \calA_{\varphibf_*}^{-1}\big(\vbf_i-\calB_{\varphibf_*}(\vbf_i,\varphibf_*\Lambda_*^{-1})\big)\, \Lambda_* 
    = \mu_i\vbf_i. 
\end{equation}
The following lemma provides estimates for the eigenvalues $\mu_i$.
\begin{lemma}\label{lem:est4mu}
    Let Assumptions \textup{\textbf{A1}} and \textup{\textbf{A2}} be fulfilled and let $\varphibf_*\in\OB$ be a ground state. Then all eigenvalues $\mu_i$ of \eqref{eq:auxEVP} are real and satisfy 
\[
    \mu_i\leq \frac{\sum_{j=1}^p \lambda_1(\calA_{\varphibf_*,j})}
{\sum_{j=1}^p \lambda_2(\calA_{\varphibf_*,j}) } <1.
\]
    If, additionally, all entries of the interaction matrix~$K$ are non-negative, then $\mu_i>-2$
    for all $i\in\mathbb{N}$.
\end{lemma}
\begin{proof}
    It immediately follows from  \eqref{eq:auxEVP} that
    \[
     \mu_i\, \langle\calA_{\varphibf_*}\!\vbf_i,\vbf_i\rangle 
     = (\vbf_i\Lambda_*,\vbf_i)_{\Lsp} - \big\langle \calB_{\varphibf_*}(\vbf_i,\varphibf_*^{}\Lambda_*^{-1})\Lambda_*^{},\vbf_i\big\rangle
     = (\vbf_i\Lambda_*,\vbf_i)_{\Lsp} - \big\langle \calB_{\varphibf_*}(\vbf_i,\varphibf_*^{}),\vbf_i\big\rangle
    \]
    and, hence, the~$\mu_i$ are real. Furthermore, the Courant--Fischer theorem implies together with the orthogonality~$\out{\vbf_{i}}{\varphibf_*}=0$ that 
    \[
    \langle\calA_{\varphibf_*}\!\vbf_i,\vbf_i\rangle
    = \sum_{j=1}^p \langle\calA_{\varphibf_*,j}v_{i,j},v_{i,j}\rangle
    \geq \sum_{j=1}^p \lambda_2(\calA_{\varphibf_*,j})\, \|v_{i,j}\|_{L^2(\Omega)}^2
    =\sum_{j=1}^p \lambda_2(\calA_{\varphibf_*,j}),
    \]
    where $v_{i,j}$ denotes the $j$-th component of $\vbf_i$. Then Proposition~\ref{prop:LagrEigVal} together with \mbox{$\langle \calB_{\varphibf_*}\!(\vbf_i,\varphibf_*),\vbf_i\rangle\! \geq\! 0$} yields that
    \[
    \mu_i 
    = \frac{\big(\vbf_i\Lambda_*,\vbf_i\big)_{\Lsp}-\langle\calB_{\varphibf_*}(\vbf_i,\varphibf_*),\vbf_i\rangle}{\langle\calA_{\varphibf_*}\vbf_i,\vbf_i\rangle}
    \leq \frac{(\vbf_i,\vbf_i\Lambda_*)_{\Lsp}}{\langle\calA_{\varphibf_*}\!\vbf_i,\vbf_i\rangle}
    \leq \frac{\sum_{j=1}^p \lambda_1(\calA_{\varphibf_*,j})}
{\sum_{j=1}^p \lambda_2(\calA_{\varphibf_*,j}) } <1
    \]
due to $\lambda_1(\calA_{\varphibf_*,j})<\lambda_2(\calA_{\varphibf_*,j})$ for $j=1,\ldots,p$.
On the other hand, we can estimate
\begin{align*}
    -\mu_i 
    & \leq \frac{\langle\calB_{\varphibf_*}(\vbf_i,\varphibf_*),\vbf_i\rangle}{\langle\calA_{\varphibf_*}\!\vbf_i,\vbf_i\rangle}
    = \frac{2\int_\Omega(\varphibf_*\circ\vbf_i)K(\varphibf_*\circ\vbf_i)^T\dx}{\|\vbf_i\|_\Hsp^2+\int_\Omega(\varphibf_*\circ\varphibf_*)K(\vbf_i\circ\vbf_i)^T\dx}
    < 2.
\end{align*}
The last inequality follows from 
\begin{align*}
\int_\Omega(\varphibf_*\circ\vbf_i)K(\varphibf_*\circ\vbf_i)^T\dx
\le \int_\Omega(\varphibf_*\circ\varphibf_*)K(\vbf_i\circ\vbf_i)^T\dx,
\end{align*} 
provided that all entries of the interaction matrix~$K$ are non-negative.
\end{proof}

We are now ready to prove that the eaRGD method \eqref{eq:eaRGD} converges locally linear in $\Hsp$ to the ground state $\varphibf_*$ and to specify the convergence rate. 

\begin{theorem}[Local convergence rate]\label{th:LocConv}
Let Assumptions~\textup{\textbf{A1}} and~\textup{\textbf{A2}} be fulfilled and let $K$ be element-wise non-negative. Further, let $\varphibf_*\in\OB$ be a~ground state and the eigenvalues~$\mu_i$ of \eqref{eq:auxEVP} be ordered decreasingly in magnitude, i.e., \mbox{$|\mu_1|\geq|\mu_2|\geq\ldots\,$}. Then the spectral radius \mbox{$\varrho_*=\varrho(\Drm\psi_\tau(\varphibf_*))$} satisfies  
\begin{equation}\label{eq:specrad}
    \varrho_*<1\qquad \text{for all } 
    \tau\in\left\{\begin{array}{ll} 
    (0, \tau_*), & \text{if }\mu_1>0, \\ 
    \big(0, \frac{2}{1+|\mu_1|}\big), & \text{if }\mu_1<0,
    \end{array}\right.
\end{equation}
where 
    \[
    \tau_*
    = \frac{2\,\sum_{j=1}^p \lambda_2(\calA_{\varphibf_*,j})}{\sum_{j=1}^p \big(\lambda_1(\calA_{\varphibf_*,j})+\lambda_2(\calA_{\varphibf_*,j})\big)}
    > 1.
    \]  
    Furthermore, for every $\epsilon>0$, there exists a neighbourhood $\,\mathcal{U}_\epsilon$ of $\varphibf_*$ in $\OB$ and a positive constant $C_\epsilon$ such that for all starting functions $\varphibf_0\in\mathcal{U}_\epsilon$, the~sequence $\{\varphibf_k\}$ generated by the eaRGD method~\eqref{eq:eaRGD} with the constant step size $\tau_k=\tau$ fulfills
    \[
    \|\varphibf_k-\varphibf_*\|_{\Hsp}
    \leq C_\epsilon\, |\varrho_*+\epsilon|^k\, \|\varphibf_0-\varphibf_*\|_{\Hsp}\qquad \text{for all } k\geq 1,
    \]
    meaning that the eaRGD iteration~\eqref{eq:eaRGD} converges locally linear with rate~$\varrho_*$.    
\end{theorem}
\begin{proof}
    Let $\mu_{\tau,i}\neq 0$ be an eigenvalue of $\Drm\psi_\tau(\varphibf_*)$ and $\vbf_{\tau,i}\in T_{\varphibf_*}\OB$ be the corresponding eigenfunction. Then~for all $\wbf\in T_{\varphibf_*}\OB$, \eqref{eq:EVP4Dpsi} and the formulae of the directional derivatives imply 
    \begin{align*}
        \mu_{\tau,i}\, (\vbf_{\tau,i},\wbf)_{\Lsp} 
        &= \big(\Drm\psi_\tau(\varphibf_*)[\vbf_{\tau,i}],\wbf\big)_{\Lsp}
        = \big(\Drm\widetilde{\psi}_\tau(\varphibf_*)[\vbf_{\tau,i}],\wbf\big)_{\Lsp} \\
        &= (1-\tau)(\vbf_{\tau,i},\wbf)_{\Lsp} 
        +\tau\,\big(\calA_{\varphibf_*}^{-1}\big(\vbf_{\tau,i}-\calB_{\varphibf_*}(\vbf_{\tau,i},\varphibf_*\Lambda_*^{-1}))\Lambda_*,\wbf\big)_{\Lsp}.
    \end{align*} 
    This shows that $\mu_{\tau,i}\neq 0$ is an~eigenvalue of~$\Drm\psi_\tau(\varphibf_*)$ if and only if $\mu_i=(\mu_{\tau,i}-1+\tau)/\tau$ is an eigenvalue of~\eqref{eq:auxEVP}. 
    Lemma~\ref{lem:est4mu} yields that $-2<\mu_1 <1$. Then using the same arguments as in the proof of \cite[Lem.~5.8]{HenY25}, we obtain that the spectral radius of $\Drm\psi_\tau(\varphibf_*)$ defined as \mbox{$\varrho_*=\max_{i\in\mathbb{N}}|1-\tau+\tau\mu_i|$} satisfies~\eqref{eq:specrad}. Therefore, the local convergence estimate immediately follows from Ostrowski's theorem \cite[Prop.~1]{AltHP21}.~
\end{proof}
\begin{remark}
Since $\tau_*>1$ and $1+|\mu_1|<3$, the restriction for $\tau$ in \eqref{eq:specrad} can be replaced by a~stronger condition $\tau\in(0,\frac{2}{3}]$, which is much simpler to verify than~\eqref{eq:specrad}. 
\end{remark}
Despite the favorable global convergence and local linear convergence rate for the eaRGD method presented above, the convergence analysis of other Riemannian optimisation methods for multicomponent BECs remains largely unexplored, particularly concerning local quantitative convergence results. 
Previously, the preconditioned $L^2$-RGD method has been studied in \cite{SchRNB09} for Hartree--Fock and Kohn--Sham problems. More recently, the RGD method with a~canonical $H^1$-metric has been rigorously analysed in \cite{ChenLLZ24}. Additionally, the RN method has been numerically proven to exhibit local quadratic convergence for Gross--Pitaevskii and Kohn--Sham problems~\cite{AltPS24}; see also~\cite{ZhaoBJ15} for the convergence analysis of the RN scheme for the discretised simplified Kohn--Sham model. However, the generalisation of these convergence results to multicomponent BECs is non-trivial and far beyond the scope of this paper. 
%
%
\section{Finite element discretisation}\label{sec:FEM}
Due to the general infinite-dimensional formulation of our optimisation schemes, we may employ any appropriate spatial discretisation technique to discretise the minimisation problem~\eqref{eq:min} and the associated NLEVP~\eqref{eq:nlevp}. This includes (mixed) finite elements, multiscale, spectral, or pseudospectral methods~\cite{BaoC13,CanCM10,GalHLP24,HenMP14}. For illustration purposes, a~finite element discretisation with $n \in \mathbb{N}$ degrees of freedom is considered in the following, which is also used in numerical experiments in Section~\ref{sec:numerics}. Let $\Phi=[\phi_1,\ldots,\phi_p]$, $U=[u_1,\ldots,u_p]$, and $W=[w_1,\ldots,w_p]$ denote the discrete tuples with $\phi_j, u_j, w_j\in\R^n$, $j=1,\dots, p$, corresponding to the components of the \mbox{$p$-frames} $\varphibf$, $\ubf$ and~$\wbf$, respectively. Note that, here and in the following, $\phi_j$, $u_j$ and $w_j$ will always refer to the discrete vectors in $\mathbb{R}^n$ rather than their continuous counterparts in $H_0^1(\Omega)$ considered in the previous sections.

The discretised version of the minimisation problem (\ref{eq:min}) takes the form
\begin{align*}
    \min_{\Phi \in \OBn} E(\Phi),
\end{align*}
where the generalised oblique matrix manifold is given by 
\[
\OBn 
= \big\{ \Phi \in \R^{n\times p}\enskip:\enskip \ddiag(\Phi^TM\Phi) = N \big\},
\]
and the discretised energy functional reads
\[
E(\Phi) =  \sum_{j=1}^p \phi_j^T\, \Big(\tfrac{1}{2}\, S + \tfrac{1}{2}\, M_{\V_j} + \tfrac{1}{4}\, M_{\rho_j(\Phi)}\Big)\, \phi_j^{}.
\]
Here, $M$ is the $L^2$-mass matrix, $S$ is the stiffness matrix, $M_{\rho_j(\Phi)} = \sum_{i=1}^p \kappa_{ij}M_{\phi_i\phi_i}$ with the discretised density functions $\rho_j(\Phi)$, and $M_{\V_j}$ and $M_{\phi_i\phi_i}$ are weighted mass matrices, where $\phi_i\phi_i$ should be understood as the element-wise product.

The discrete version of the bilinear form $a_\varphibf$ in \eqref{eq:aphi} reads 
\[
a_\Phi(U,W) = \langle A_\Phi(U),W\rangle_M = \trace\big(W^TMA_\Phi(U)\big)
\]
with the linear operator $A_\Phi\colon\R^{n\times p}\to\R^{n\times p}$ given by
\[
A_\Phi(U) = \big[ M^{-1} A_{\Phi,1} u_1,\ldots,M^{-1}A_{\Phi,p} u_p\big],
\]
where $A_{\Phi,j} = S + M_{\V_j} + M_{\rho_j(\Phi)}$ for $j=1,\ldots,p$. Furthermore, the discretisation of the operator~$\calB_\varphibf$ in~\eqref{eq:calB} is given by
\[
B_\Phi(U,\Phi)=\bigg[M^{-1}\sum_{j=1}^p B_{\Phi,1j}\, u_j,\ \ldots,\ M^{-1}\sum_{j=1}^p B_{\Phi,pj}\, u_j\bigg]
\]
with $B_{\Phi,ij} = 2\kappa_{ij}\, M_{\phi_i\phi_j}$ for $i,j=1,\ldots,p$. 
We equip the tangent space 
\[
T_\Phi\OBn 
= \big\{ Z \in \R^{n\times p}\enskip:\enskip \ddiag(\Phi^TM Z)=0_p\big\}
\]
with a Riemannian metric which follows the product structure 
\[
g_\Phi^\times(Z,Y)=\trace(Z^TMG_\Phi^\times(Y)), \qquad
Z,Y\in T_\Phi\OBn,
\]
where the operator $G_{\Phi}^\times(Y)=[M^{-1}G_{\Phi,1}y_1,\ldots, M^{-1}G_{\Phi,p}y_p]$ acts column-wise and the matrices \mbox{$G_{\Phi,j}\in \R^{n\times n}$} are symmetric positive definite. As discussed in Section~\ref{sec:oblique}, this restriction to metrics, that act on all components independently, allows expressions that are both simple and computationally efficient. The $g_\Phi^\times$-orthogonal projection onto the tangent space $T_\Phi\OBn$ is given~by
\begin{align*}
P_{\Phi}^\times(U) 
&= U-(G_\Phi^\times)^{-1}(\Phi) \ddiag\big(\Phi^TMU\big)\big(\ddiag\big(\Phi^TM(G_\Phi^\times)^{-1}(\Phi)\big)\big)^{-1} \\ 
&= \left[u_1 - \frac{\phi_1^TMu_1}{\phi_1^TMG_{\Phi, 1}^{-1}M\phi_1}G_{\Phi,1}^{-1}M\phi_1,\ \ldots,\ 
u_p - \frac{\phi_p^TMu_p}{\phi_p^TMG_{\Phi, p}^{-1}M\phi_p}G_{\Phi,p}^{-1}M\phi_p \right].
\end{align*}
Further, the discretisation of $\calR_\varphibf$ is given by
\begin{equation*}
    R_\Phi(U) = \bigg[M^{-1}R_{\Phi,1}u_1, \ldots, M^{-1}R_{\Phi,p}u_p \bigg]
\end{equation*}
with $R_{\Phi,j} = A_{\Phi,j} - \sigma_jM$ and $\sigma_j = \frac{\phi_j^TA_{\Phi,j}\phi_j^{}}{N_j}$ for $j=1,\ldots,p$.
Using the above definitions, the corresponding Riemannian gradient of $E$ at $\Phi$ then takes the form
\begin{align*}
\grad E(\Phi)
 &= P_\Phi^\times\big((G_\Phi^\times)^{-1}(R_\Phi(\Phi))\big) 
 = \big[\,G_{\Phi,1}^{-1}(R_{\Phi,1}^{}-\theta_1\, M)\phi_1, \,\ldots\,,\, 
 G_{\Phi,p}^{-1}(R_{\Phi,p}^{} - \theta_p\,M)\phi_p\,\big],
\end{align*}
where
\[
\theta_j = \frac{\phi_j^TMG_{\Phi,j}^{-1}R_{\Phi,j}\phi_j}{\phi_j^TMG_{\Phi,j}^{-1}M\phi_j}, \qquad j=1,\ldots,p.
\]
Using $G_{\Phi,j} = M$, $G_{\Phi,j} = \widehat{A}_{\Phi,j}$ or $G_{\Phi,j}=G_{\Phi,\omega,j} = A_{\Phi,j} + B_{\Phi,jj} - \omega\sigma_jM$, we obtain the Riemannian gradients with respect to the discrete $L^2$-, energy-adaptive or Lagrangian-based metric, respectively. 
Here, $\widehat{A}_{\Phi,j} = A_{\Phi,j} + c_\Phi M$, where $c_\Phi \geq 0$ is an appropriately chosen shift. If $K$ has only non-negative entries, we can use $c_\Phi = 0$, i.e., $\widehat{A}_{\Phi,j} = A_{\Phi,j}$.
The discretisation of the Riemannian Hessian with respect to the $L^2$-metric is given by
\[
    \Hess_M E(\Phi)[Z] = M^{-1} \Big[P_{\Phi,M,1}\big(R_{\Phi,1}z_1 + \sum\limits_{j=1}^p B_{\Phi,1j}\, z_j\big), \ldots, P_{\Phi,M,p}\big(R_{\Phi,p}z_p + \sum\limits_{j=1}^p B_{\Phi,pj}\, z_j\big) \Big]
\]
with the projectors $P_{\Phi,M,j}=I - \frac{1}{N_j}M\phi_j^{}\phi_j^T$.
These expressions can be used to construct the discrete counterparts of the Riemannian optimisation schemes presented in Section~\ref{sec:methods}.

It should be noted that the global convergence results for the eaRGD method no longer hold in the finite-dimensional case because our particular choice of the finite element discretisation does not satisfy a~discrete maximum principle. Achieving this would require a~slightly modified, stabilised scheme as proposed in \cite{hauck2024positivitypreservingfiniteelement}. However, in practice, unless specifically triggered, a~violation of the maximum principle and potential convergence to an excited state from a~positive initial guess will typically not be observed. We also refer to \cite{ChenLLZ24a}, for recent convergence results of the RGD method based on the $H^1$-metric for a~discretised single-component BEC.

%
\section{Numerical experiments}\label{sec:numerics}
We report on numerical experiments for two test examples in 1D and 2D domains, aiming at comparing the performance of the different Riemannian optimisation schemes presented in Section~\ref{sec:methods}. For the spatial discretisation, we use the finite element method with bi-quadratic elements on a quadrilateral mesh of width $h$. We impose Neumann boundary conditions instead of the Dirichlet boundary conditions discussed above, but due to the strong trapping potentials in the models, all values on the boundary are sufficiently close to zero in all experiments.
The implementation was done in the Julia programming language using the package {\tt Ferrite.jl} for the finite element discretisation. The source code is available in the GitHub repository
\url{https://github.com/MaHermann/Riemannian-coupledGPE}
and relies on the following parameter and design choices:
\begin{itemize}
    \item Stopping criterion: We terminate the iterations once the norm of the residual \linebreak
    \mbox{${\rm res}_k=(\trace(R_{\Phi_k}(\Phi_k)^TMR_{\Phi_k}(\Phi_k)))^{1/2}$}
    falls below the tolerance $\tol=10^{-8}$.

    \item Initial guess: For all schemes, we use a consistent initial guess $\Phi_0$ computed by the alternating eaRGD method, where each component is initialised by $1$. Iterations continue until the residual norm, as defined by the stopping criterion, falls below $10^{-2}$ in the 1D case and $10^{-4}$ in the 2D case.

    \item Step size: To maintain transparency in the comparison between methods, we use a constant step size of $\tau_k = 1$ for all RGD methods presented here. This approach allows a~clear evaluation of the inherent convergence properties of each method, without introducing variability due to the interplay between method choice and adaptive step size strategies. In particular, all methods show stable convergence without the need for smaller step sizes, except for the 2D problem with random potential (see below). The use of an adaptive step size strategy, such as a non-monotone line search combined with the alternating Barzilai--Borwein step size strategy \cite{WenY13, ZhaH04}, could further improve the optimisation performance, but we leave this method-specific refinement for future work.
    
    \item Linear system solvers: For solving linear systems in Riemannian gradient and Newton direction calculations, we employ the preconditioned conjugate gradient method with the preconditioner determined from the ILU decomposition of $A_{0,j}=S+M_{V_j}$ and an~adaptive residual tolerance $\tol_{\rm CG}$ for the 1D system and $10 \tol_{CG}$ for the 2D system, where the latter has a weaker nonlinearity, allowing for less accurate solutions. Here, $\tol_{\rm CG}={\rm res}_k$ as defined above for the Newton-type methods, and the residual norm for the individual components for the RGD methods. In the computation of the initial guess $\Phi_0$ by using the alternating eaRGD method, we take a~tolerance of $1.5 \cdot 10^{-8} \tol_{\rm CG}$ in the linear systems.
    
    \item Regularisation parameters: In the LgrRGD  method, we take $\omega_k=1$, as discussed in Section~\ref{ssec:LgrRGD}. In the regRN method, we use $\omega_k=0.99$, which proves to be enough to provide convergence in the cases where the RN method (which corresponds to $\omega_k = 1$) does not converge.
\end{itemize}

Although the most canonical performance measure is probably the number of outer iterations to a~given stopping criterion, in practice, other measures may play a more important role. For example, a Newton step takes on average much longer than a gradient step, so it makes sense to also compare the total CPU time required. Since the computational time depends to a large extent on the implementation and the hardware used, we also report the average number of matrix--vector multiplications with $n \times n$ matrices per optimisation step, which, together with the assembly of the finite element matrices discussed below, are the most expensive operations in all optimisation schemes. Together, these performance measures provide sufficient insight into the individual strengths of the different optimisation methods. 

Each optimisation step requires the assembly of the finite element matrices $M_{\phi_i\phi_j}$. In the eaRGD and LgrRGD methods, only the $p$ matrices $M_{\phi_i\phi_i}$ are needed, while the RN and regRN methods require all $p(p + 1) / 2$ combinations to compute the (regularised) Riemannian Hessian. In the cases below with~$p = 2, 3$, this amounts to a~factor of at most two, but it should be taken into account when applying the optimisation schemes to models with a large number of components.

Note that in both examples below, the interaction matrix $K$ has non-negative entries, so by Remark~\ref{re:non-negative} there is no need to shift $A_\Phi$ to ensure its invertibility. 
%
%
\subsection{Two-component BEC in 1D}
First, we consider a well-known benchmark example from the literature, see, e.g., \cite{BaoC11,HuaY24}. It describes a two-component BEC on $\Omega = [-16, 16]$ with the potentials 
\[
    V_1(x) 
    = V_2(x)
    = 2\, \left(\frac{1}{2}\,x^2 + 24\cos^2(x)\right)
\]
and the interaction parameters $\kappa_{11} = 2.08\, \beta$, $\kappa_{22} = 1.94\, \beta$ and $\kappa_{12} = \kappa_{21} = 2\, \beta$, where different values for $\beta > 0$ can be chosen. By renormalisation, we can assume, without loss of gene\-ra\-li\-ty, that \mbox{$N_1+N_2=1$}, so we take \mbox{$N_1 = \alpha$} and $N_2 = 1 - \alpha$ with $\alpha \in (0, 1)$. The differences in parameters compared to the earlier works result from the choice of constants in the energy and the fact that the interaction matrix~$K$ used there is not positive definite and therefore does not fit in the setting of this paper, so we use a~slightly larger $\kappa_{11}$. 
In the following, we fix $\alpha = 0.8$ and perform experiments for $\beta = 10, 100, 1000$ to examine how the strength of the nonlinearity affects the individual optimisation methods. 

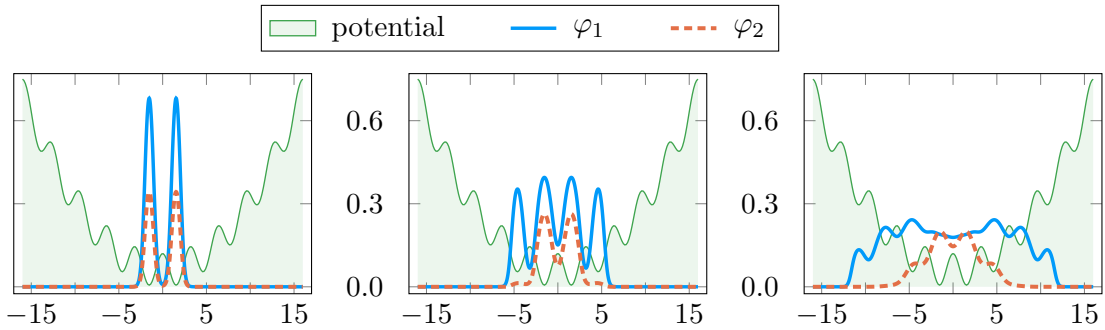
\begin{figure}[t]
\hspace*{-5mm}
\input{img/results/1D/ground_state/beta=10.tex}
\input{img/results/1D/ground_state/beta=100.tex}
\input{img/results/1D/ground_state/beta=1000.tex}
\vspace{-5mm}
\caption{Two-component BEC: potential and components of the ground state for $\beta = 10, 100, 1000$ (from left to right). The potential is rescaled by a~factor of $0.0025$ for plotting purposes.}
\label{fig:groundstates1D}
\end{figure}

We choose a finite element mesh width of $h = 32 \cdot 2^{-10}$, resulting in $n=2049$ degrees of freedom. The initialisation with the alternating eaRGD method takes $5$, $9$, and $17$ iterations for $\beta = 10, 100$, and $1000$, respectively. Figure~\ref{fig:groundstates1D} presents the potential and the resulting ground state components. The convergence history of the residual norms for different optimisation methods and different values of $\beta$ is shown in Figure~\ref{fig:convergence1D}. To illustrate the advantage of the alternating variant of the LgrRGD scheme, in addition to the convergence history for the alternating versions, we also include the residual plots for the non-alternating schemes for the case $\beta = 100$. Further details on the number of outer iterations and the average number of matrix--vector multiplications in an optimisation step are given in Table~\ref{tab:twoBECresults}. We do not provide time measurements for these one-dimensional experiments, as they are all in the order of seconds, where precompilation of the Julia code and other factors such as implementation details and processor load dominate. 

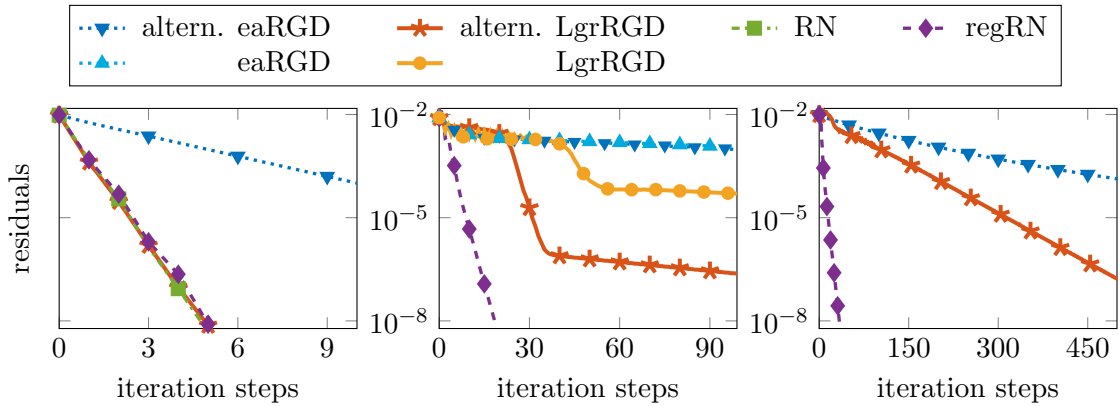
\begin{figure}[t]
\hspace*{-2mm}
\input{img/results/1D/convergence/beta=10.tex}
\input{img/results/1D/convergence/beta=100.tex}
\input{img/results/1D/convergence/beta=1000.tex}
\caption{Two-component BEC: convergence history of the residuals for \mbox{$\beta = 10, 100, 1000$} (from left to right). For $\beta = 100$, the non-alternating versions of the eaRGD and LgrRGD methods are shown as well.}
\label{fig:convergence1D}
\end{figure}

\begin{table}[t]
\caption{Two-component BEC: performance measures for different optimisation methods
\label{tab:twoBECresults}}
\footnotesize
\begin{tabular}{l || r | c || r | c || r | c }
& \begin{tabular}{c}outer\\[-1mm]iter.\end{tabular} & \begin{tabular}{c}aver.~matr.--vec.\\[-1mm]mult.~per iter.\end{tabular} & \begin{tabular}{c}outer\\[-1mm] iter.\end{tabular} & \begin{tabular}{c}matr.--vec.~mult.\\[-1mm]
per iter.\end{tabular} & \begin{tabular}{c}outer\\[-1mm] iter.\end{tabular} & \begin{tabular}{c}matr.--vec.~mult.\\[-1mm]per iter.\end{tabular} \\
\hline
       & \multicolumn{2}{ c ||}{ $\beta = 10$} & \multicolumn{2}{ c ||}{ $\beta = 100$} & \multicolumn{2}{ c }{ $\beta = 1000$} \\
\hline
altern. eaRGD  &  31\hspace*{2mm} & 12.1  &  1147\hspace*{1mm} & \hspace*{-1.8mm}14.0 & 2225\hspace*{1mm} & 20.3 \\ 
altern. LgrRGD &   5\hspace*{2mm} & 30.4  &   259\hspace*{1mm} & \hspace*{-1.8mm}60.0 &  629\hspace*{1mm} & \hspace*{-1.8mm}106.6 \\
\hspace*{9.5mm} RN & 5\hspace*{2mm} & 20.0  &  - \hspace*{1mm} &  -    & -\hspace*{2mm}  &     - \\
\hspace*{9.5mm} regRN &5\hspace*{2mm}& 20.0  &   19\hspace*{1mm} & \hspace*{-1.8mm}53.3 & 34\hspace*{1mm}   & \hspace*{-1.8mm}131.9 \\
\end{tabular}
\end{table}

Figure~\ref{fig:convergence1D} and Table~\ref{tab:twoBECresults} show that the regRN method converges in significantly fewer iterations than the alternating eaRGD method for all three parameter values of $\beta$ and that the alternating LgrRGD method is competitive to both the RN and regRN methods for $\beta=10$. For strong interactions ($\beta = 100, 1000$), the RN method does not converge to the ground state, as the initial value is too far away. This could be fixed by more initial iterations with the reliable eaRGD method or by regularisation with the regRN method. 
%
%
\subsection{Three-component BEC in 2D}
As a second example, we consider a~three-compo\-nent BEC model on the unit square $\Omega = [0,1]^2$, uniformly partitioned with mesh width $h = 2^{-10}$ in each direction, resulting in $n = 4\,198\,401$ degrees of freedom, a~much larger problem size than in the previous 1D example. The interaction parameters are $\kappa_{11} = 0.5$, $\kappa_{22} = 5$, $\kappa_{33} = 10$, and $\kappa_{ij} = 1$ otherwise, while the number of particles is $N_1 = N_2 = N_3 = 1$. 

\begin{figure}[t]
	\centering
	\begin{subfigure}[b]{0.24\textwidth}
		\centering
		\includegraphics[width=\textwidth]{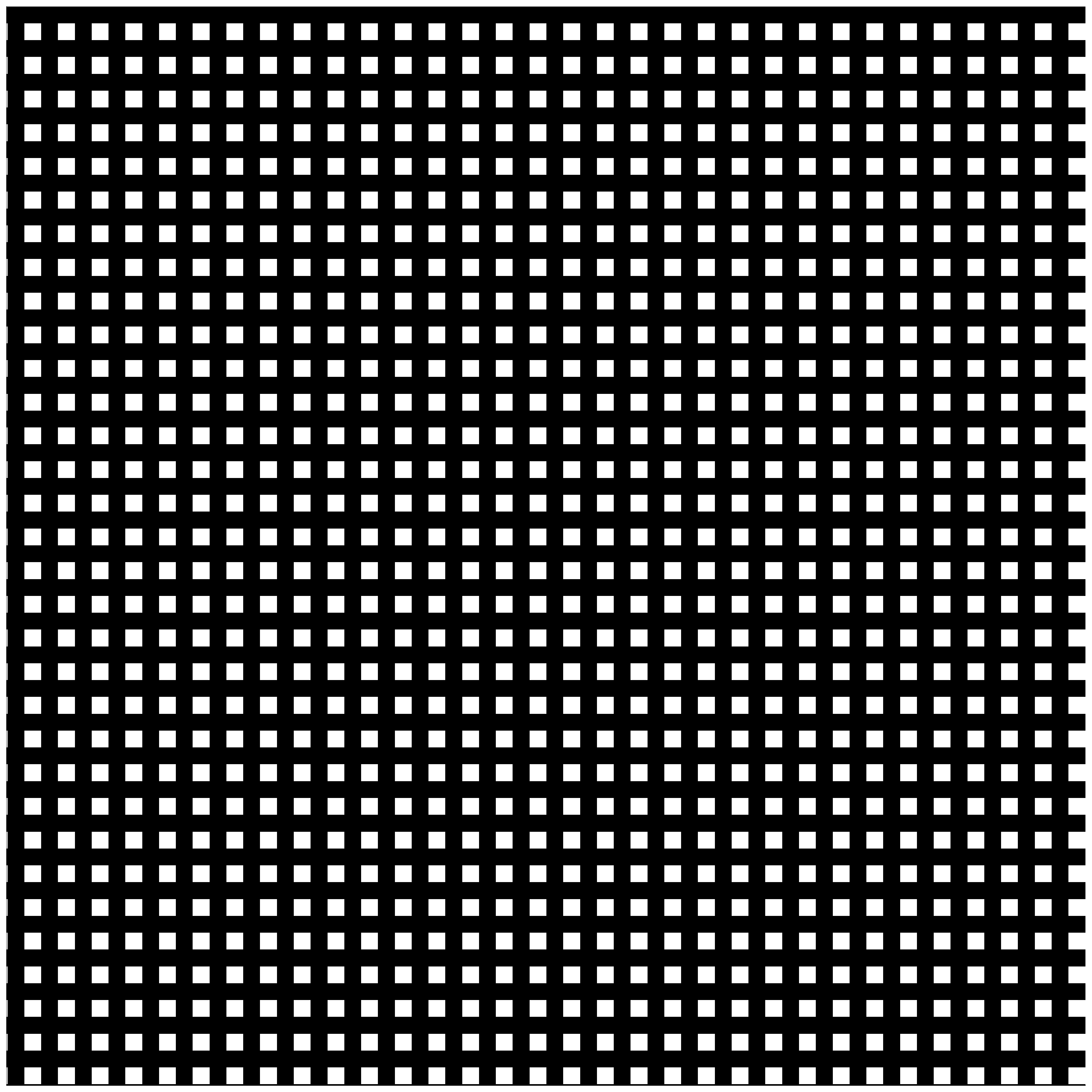}
		\caption{periodic potential}
	\end{subfigure}
	\begin{subfigure}[b]{0.24\textwidth}
		\centering
		\includegraphics[width=\textwidth]{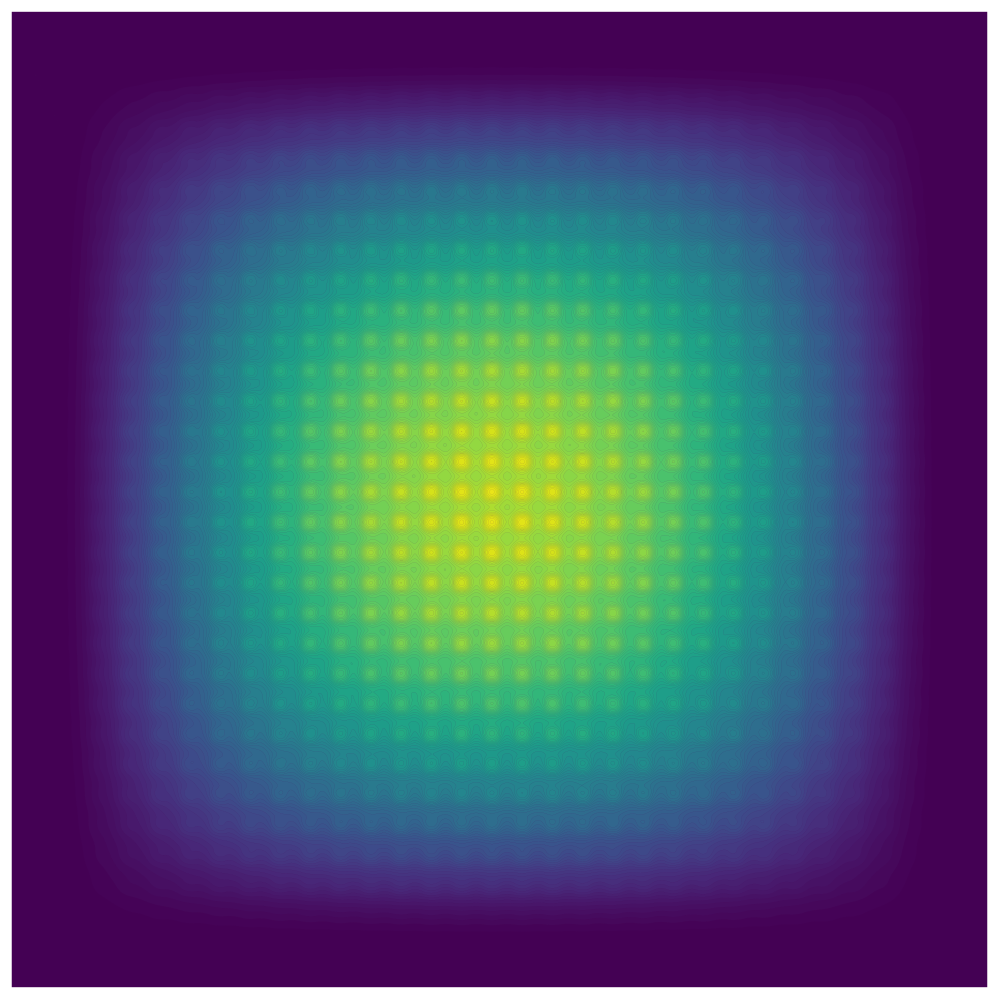}
		\caption{$\varphibf_1$}
	\end{subfigure}
	\begin{subfigure}[b]{0.24\textwidth}
		\centering
		\includegraphics[width=\textwidth]{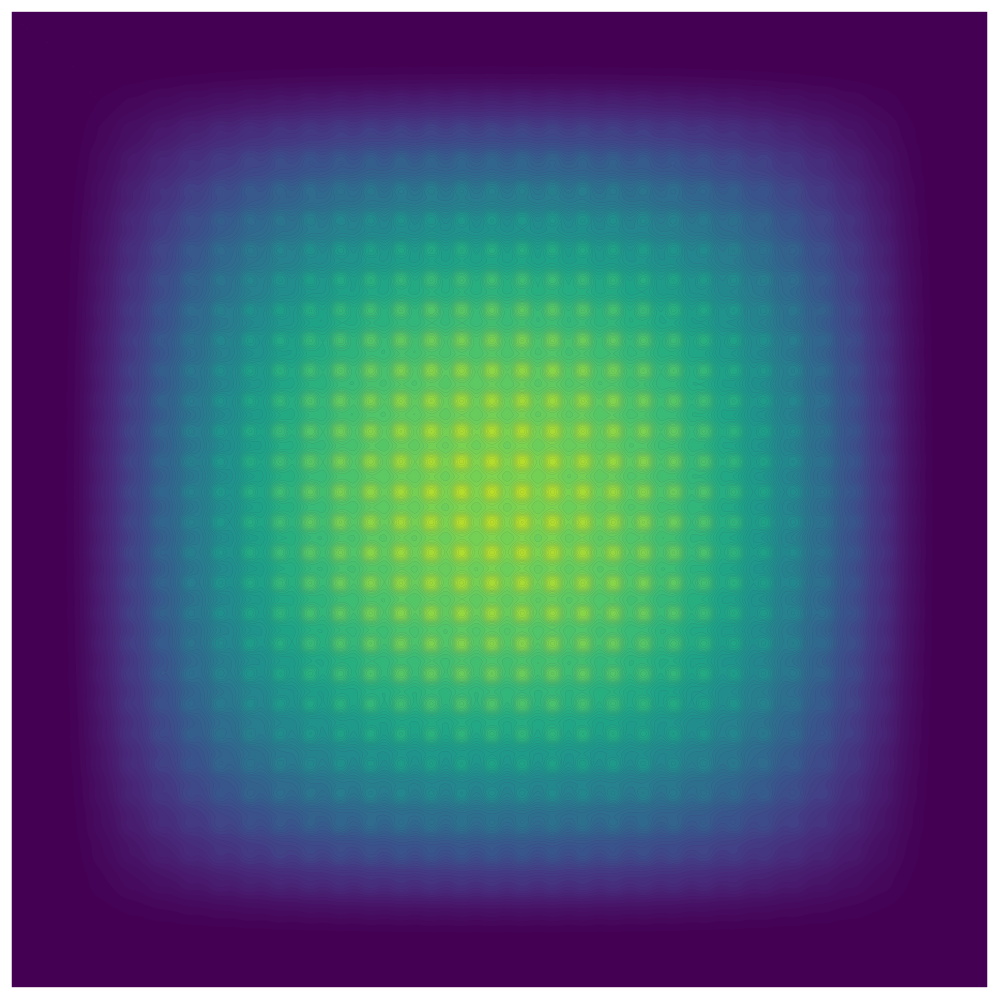}
		\caption{$\varphibf_2$}
	\end{subfigure}
	\begin{subfigure}[b]{0.24\textwidth}
		\centering
		\includegraphics[width=\textwidth]{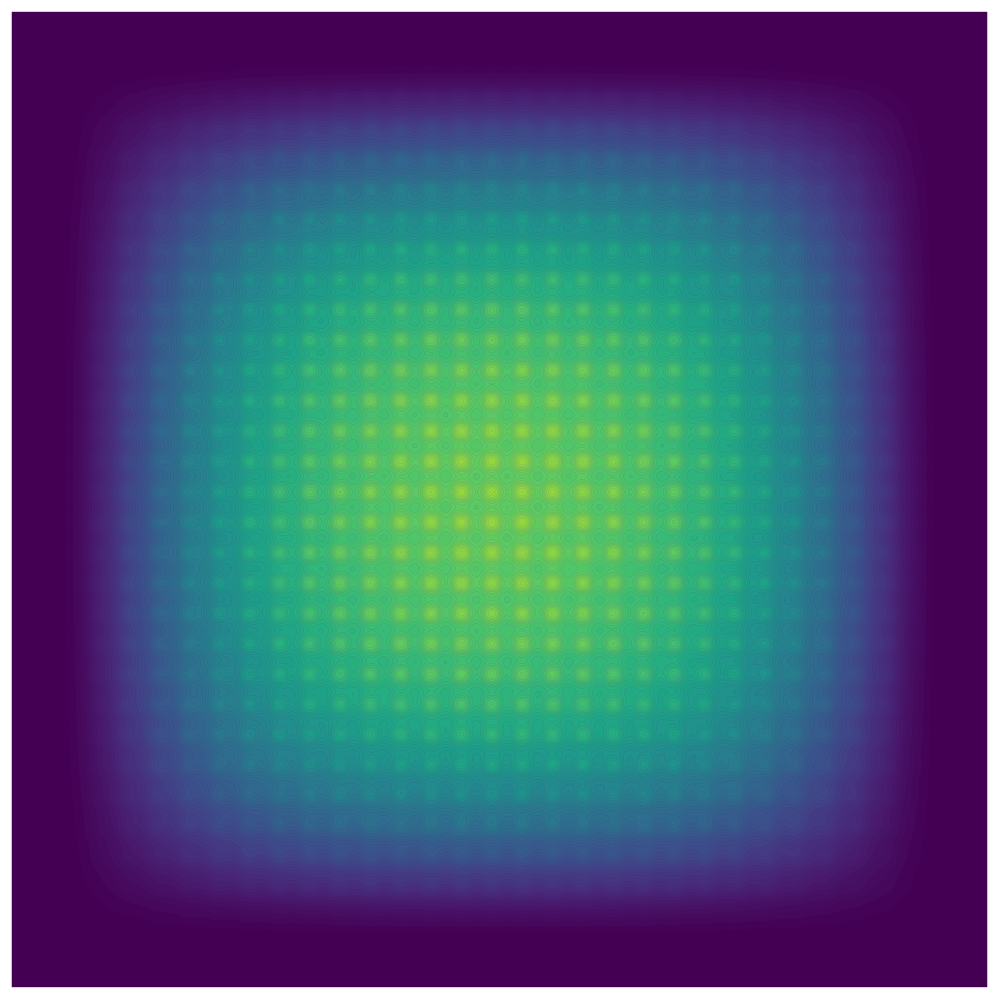}
		\caption{$\varphibf_3$}
	\end{subfigure} \\[0.5em]
	\begin{subfigure}[b]{0.24\textwidth}
		\centering
		\includegraphics[width=\textwidth]{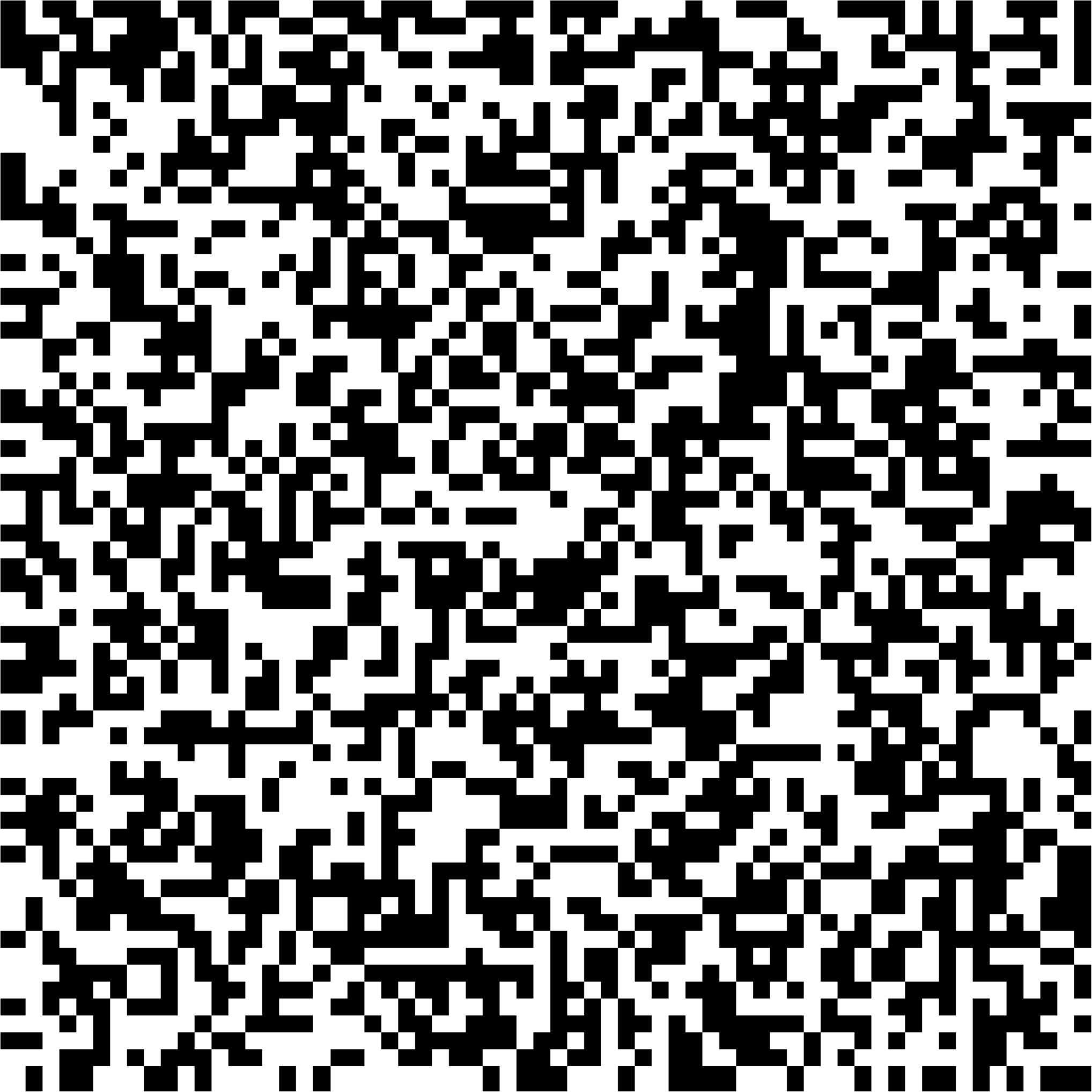}
		\caption{random potential}
	\end{subfigure}
	\begin{subfigure}[b]{0.24\textwidth}
		\centering
		\includegraphics[width=\textwidth]{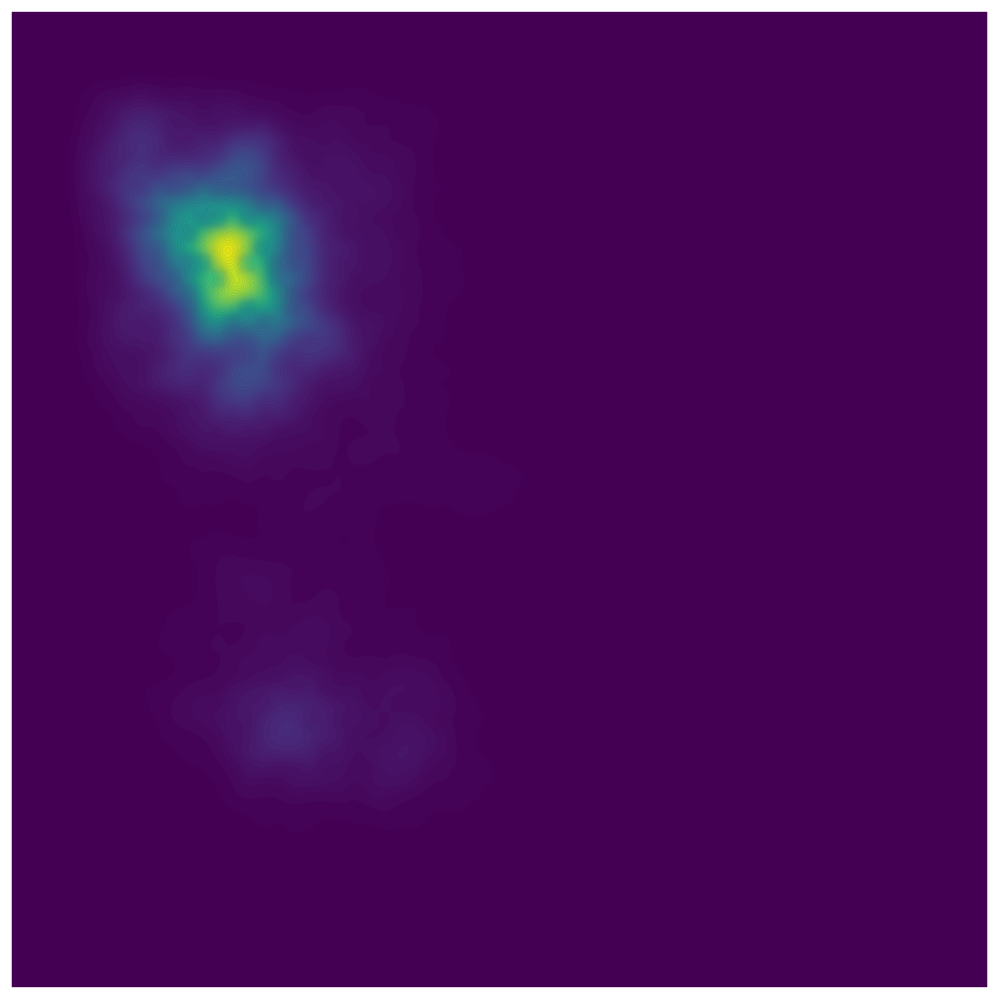}
		\caption{$\varphibf_1$}
	\end{subfigure}
	\begin{subfigure}[b]{0.24\textwidth}
		\centering
		\includegraphics[width=\textwidth]{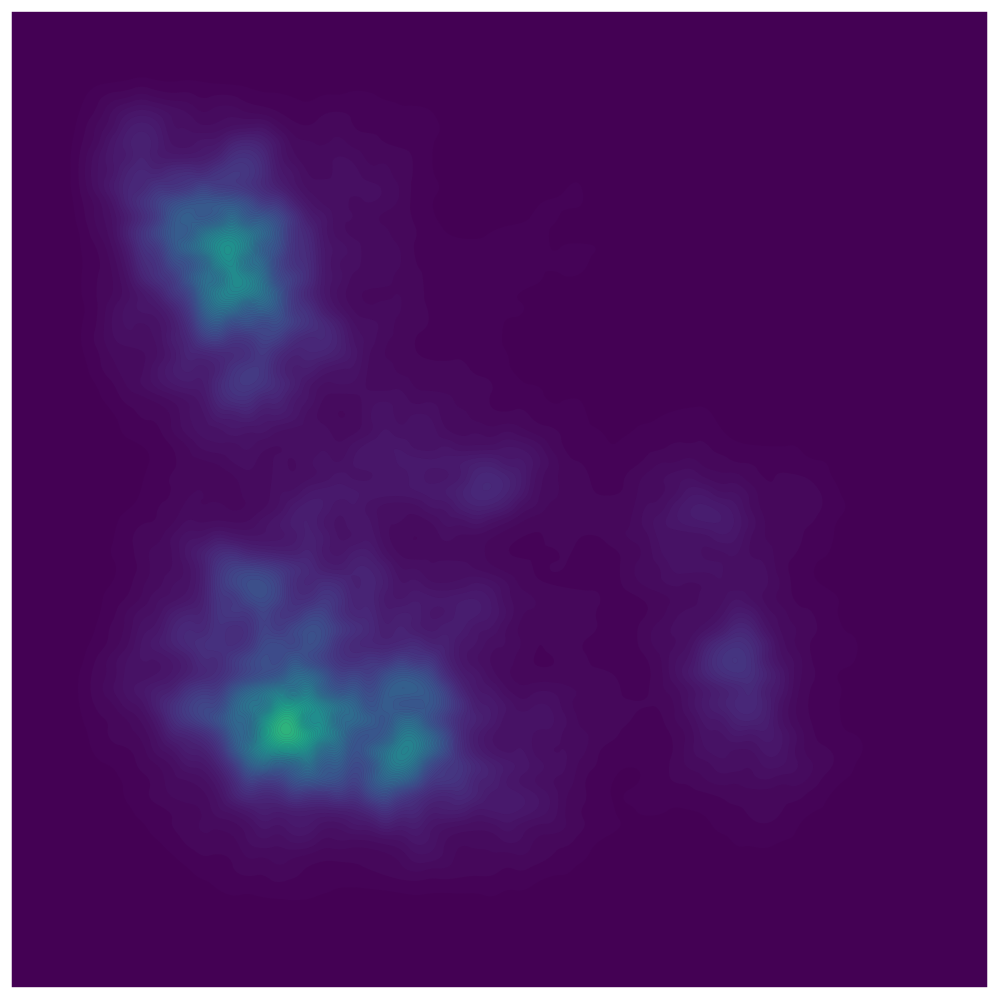}
		\caption{$\varphibf_2$}
	\end{subfigure}
	\begin{subfigure}[b]{0.24\textwidth}
		\centering
		\includegraphics[width=\textwidth]{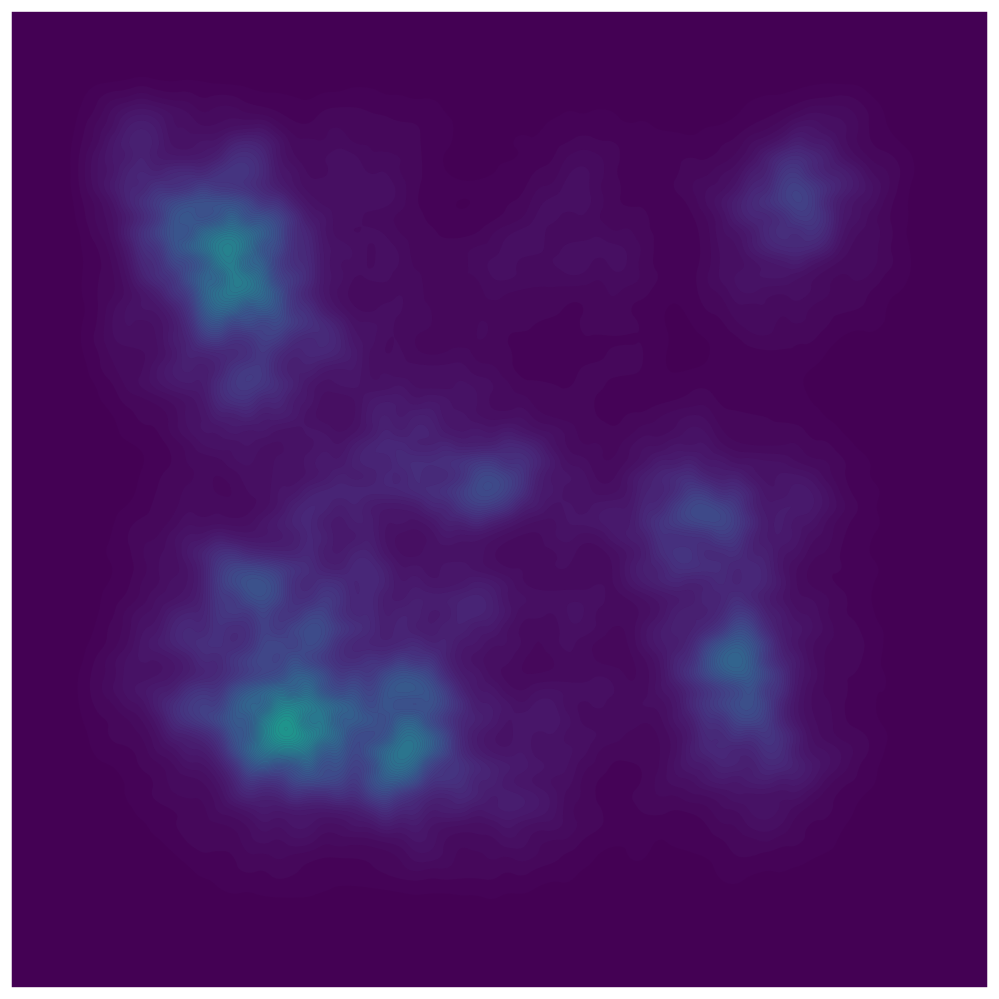}
		\caption{$\varphibf_3$}
	\end{subfigure}
	\vspace*{-3mm}
	\caption{Three-component BEC: random and periodic potentials and components of the ground states computed with the alternating eaRGD method. For the potentials, we do not depict the additional trapping potential that is used to enforce the homogeneous Dirichlet boundary conditions.}
	\label{fig:2D_states}
\end{figure}

We consider two cases: a~periodic potential and a~piecewise random potential, both taking values in $\{0, 2^{12}\}$ and varying on a length scale of $\epsilon = 2^{-6}$. In both cases, to enforce small values on the boundary, we additionally add a trapping potential 
\[
V_\text{trap}(x_1, x_2) = 10^6 \max\{(2x_1 - 1)^{40}, (2x_2 - 1)^{40}\}.
\]
The initialisation takes $3$ and $4$~eaRGD steps for the periodic and random potentials, respectively.

In our experiments, we observed that for the periodic potential, all optimisation schemes converge to the same ground state. For the random potential, however, only the eaRGD method reliably provides the ground state, while the other schemes find excited states with higher energy, even with the relatively small initialisation residual of $10^{-4}$. 

Comparing the computed ground states for the two models shown in Figure~\ref{fig:2D_states}, we see that for the periodic potential, the ground state remains spread over the entire domain, while for the random potential, the condensate undergoes localisation, see \cite{AltHP20,AltHP22,AltP19} for related results for single-component BECs. These works also justify that both cases are extremely computationally challenging, as the spectral gaps in the lower energy spectrum are very small and scale like $\epsilon$, which explains the need for very accurate initial estimates to avoid approaching excited states of similar energy levels. 

The performance measures including the computational time together with the final energy values and residual norms for the convergent methods are collected in Table~\ref{tab:threeBECresults}.  It can be seen that for the periodic potential, the alternating LgrRGD, RN, and regRN methods, which are based on the second-order information, significantly outperform the alternating eaRGD method both in terms of the number of outer iterations and the computational time. For the random potential, as reported above, only the alternating eaRGD method provides the ground state, but  requires a~prohibitively large number of iterations to achieve small target residuals and is computationally very expensive. Therefore, a~combination of the eaRGD and RN methods, exploiting both the global convergence guarantees of the former and the fast local convergence speed of the latter, seems to be a promising approach to compute ground states in an efficient way. In addition, the development of better preconditioners to reduce the large number of inner iterations and suitable step size control strategies to speed up convergence could further improve the computational performance of the optimisation algorithms.

\begin{table}[t]
\footnotesize
\caption{Three-component BEC: performance measures for different optimisation methods
\label{tab:threeBECresults}}
\begin{tabular}{l || r | r | r | c | c}
    &\begin{tabular}{c}outer\\[-1mm]iter.\end{tabular}   
    &\begin{tabular}{c}aver.~matr.--vec.\\[-1mm]mult.~per iter.\end{tabular}
    &\begin{tabular}{c}CPU time\\[-1mm] (minutes)\end{tabular} & energy & residual \\
\hline
    & \multicolumn{5}{ c }{ periodic potential }\\
\hline
altern. eaRGD  & 206\hspace*{2mm}   &   42.8\hspace*{6mm} & 87\hspace*{4mm}
& \hspace*{3mm}4582.2\hspace*{3mm} & \hspace*{3mm}9.9017e-09\hspace*{3mm}\\
altern. LgrRGD &   4\hspace*{2mm}   &  665.2\hspace*{6mm} & 24\hspace*{4mm} & \hspace*{3mm}4582.2\hspace*{3mm} & \hspace*{3mm}4.3065e-10\hspace*{3mm}\\
\hspace*{7.3mm} RN     &   4\hspace*{2mm}   & 1249.5\hspace*{6mm} & 25\hspace*{4mm} &\hspace*{3mm}4582.2\hspace*{3mm} & \hspace*{3mm}6.3026e-11\hspace*{3mm}\\
\hspace*{7.3mm} regRN  &   5\hspace*{2mm}   & 954.6\hspace*{6mm} & 25\hspace*{4mm}  & \hspace*{3mm}4582.2\hspace*{3mm}& \hspace*{3mm}7.5963e-09\hspace*{3mm}\\
\hline
 & \multicolumn{5}{ c }{ random potential } \\\hline
 altern. eaRGD  & 1503\hspace*{1mm}  &   63.5\hspace*{6mm} & 773\hspace*{4mm} & \hspace*{3mm}2332.1\hspace*{3mm}& \hspace*{3mm}9.9942e-09\hspace*{3mm}
\end{tabular}
\end{table}
%
%
%
\section{Conclusion}
In this paper, we have introduced a general framework for computing the ground state of multi-component Bose--Einstein condensates using Riemannian optimisation methods on the infinite-dimensional generalised oblique manifold. This problem is formulated as an energy minimisation problem with mass conservation constraints for each condensate component. We have established existence and uniqueness results for the ground state on a bounded spatial domain and linked it to the coupled Gross--Pitaevskii eigenvalue problem. In addition, we have explored the Riemannian structure of the generalised oblique manifold by defining several Riemannian metrics and computing essential geometric tools such as Riemannian gradients and Hessians.

By incorporating first- and second-order information of the energy functional, we have constructed appropriate metrics that allow preconditioning within Riemannian optimisation, which, combined with improvements through alternating iterations, significantly enhances optimisation performance. Our qualitative global and quantitative local convergence results for the energy-adaptive Riemannian gradient descent method provide a robust basis for further exploration of complex multicomponent Bose--Einstein condensation phenomena in quantum systems.
\section*{Acknowledgments}
The authors would like to thank Jonas P\"uschel for fruitful and constructive discussions on important implementation aspects. 
%
%
\bibliographystyle{siamplain}
\bibliography{refBEC} 
\end{document}

%% file: def.tex
\definecolor{myBlue}{RGB}{113,104,238} 
\definecolor{myGreen}{RGB}{154,205,50} 
\definecolor{myGreen2}{RGB}{114,175,30} 
\definecolor{myRed}{RGB}{180,50,50}  
\definecolor{myOrange}{RGB}{225,92,22} 
\definecolor{lgray}{RGB}{200,200,200} 
\definecolor{llgray}{RGB}{155,155,155} 
\definecolor{alphaCol}{RGB}{230,190,50} 
\definecolor{betaCol}{RGB}{90,80,180} 
\definecolor{mycolor1}{rgb}{0.00000,0.44700,0.74100}
\definecolor{mycolor1_}{rgb}{0.00000,0.67400,0.87050}
\definecolor{mycolor2}{rgb}{0.85000,0.32500,0.09800}
\definecolor{mycolor2_}{rgb}{0.95000,0.64000,0.12900}
\definecolor{mycolor3}{rgb}{0.92900,0.69400,0.12500}
\definecolor{mycolor4}{rgb}{0.49400,0.18400,0.55600}
\definecolor{mycolor5}{rgb}{0.46600,0.67400,0.18800}
\definecolor{mycolor6}{rgb}{0.30100,0.74500,0.93300}
\definecolor{mycolor7}{rgb}{0.63500,0.07800,0.18400}%

\definecolor{corrRed}{RGB}{18,124,175} 

\newcommand\R{\mathbb R}
\newcommand\Sphere{\mathcal{S}}

\DeclareMathAlphabet{\mathpzc}{OT1}{pzc}{m}{it}

\newcommand\OB{\mathcal{O}\mathcal{B}_\Nbf(p,\Hsp)}
\newcommand\OBn{\mathcal{O}\mathcal{B}_{\Nbf,M}(p,n)}
\newcommand\Dp{\mathcal{D}(p)}

\DeclareMathOperator{\diag}{diag}

\DeclareMathOperator{\ddiag}{ddiag}

\DeclareMathOperator{\grad}{grad}
\DeclareMathOperator{\Hess}{Hess}
\DeclareMathOperator{\Kern}{ker}

\DeclareMathOperator{\tol}{tol}

\DeclareMathOperator{\trace}{trace}

\DeclareMathOperator{\Drm}{D}

\newcommand\calA{\mathcal A}
\newcommand\calB{\mathcal B}
\newcommand\calC{\mathcal C}
\newcommand\calE{\mathcal E}
\newcommand\calF{\mathcal F}
\newcommand\calG{\mathcal G}
\newcommand\calH{\mathcal H}

\newcommand\calL{\mathcal L}

\newcommand\calN{\mathcal N}

\newcommand\calP{\mathcal P}
\newcommand\calR{\mathcal R}

\newcommand\calAhat{\widehat{\calA}}
\newcommand\ahat{\widehat{a}}
\newcommand{\etabf}{{\bm \eta}}

\newcommand{\rhobf}{{\bm \rho}}

\newcommand{\varphibf}{{\bm \varphi}}

\newcommand{\gbf}{\bm{g}}

\newcommand{\ybf}{\bm{y}}

\newcommand{\ubf}{\bm{u}}
\newcommand{\vbf}{\bm{v}}
\newcommand{\wbf}{\bm{w}}
\newcommand{\zbf}{\bm{z}}

\newcommand{\zerobf}{\bm{0}}
\newcommand{\V}{V}
\newcommand{\Hsp}{H}
\newcommand{\Lsp}{L}
\newcommand{\Nbf}{N}

\def\dx{\,\text{d}x}

\makeatletter
\newsavebox{\@brx}
\newcommand{\llangle}[1][]{\savebox{\@brx}{\(\m@th{#1\langle}\)}%
  \mathopen{\copy\@brx\mkern2mu\kern-0.9\wd\@brx\usebox{\@brx}}}
\newcommand{\rrangle}[1][]{\savebox{\@brx}{\(\m@th{#1\rangle}\)}%
  \mathclose{\copy\@brx\mkern2mu\kern-0.9\wd\@brx\usebox{\@brx}}}
\makeatother

\newcommand{\out}[2]{\llangle{#1},{#2}\rrangle}

%% file: img/results/1D/convergence/beta=10.tex
\begin{tikzpicture}

\begin{axis}[
width=55.0mm, height=45.0mm, 
at={(0.0mm,0.0mm)},
xlabel={iteration steps}, 
xmin=0, xmax=10, xticklabels={$0$,$3$,$6$,$9$}, xtick={0.0,3.0,6.0,9.0}, 
ylabel={residuals}, 
ymode=log, log basis y=10, 
ymin= 6e-9, ymax=0.015, 
yticklabels=\empty, ytick={1e-8, 1e-5, 1e-2}, 
legend columns = 4,
legend style={legend cell align=right, align=left, at={(1.7,1.1)}, anchor=south}
]
\addplot[color=mycolor1, line width=1.25, dotted, mark=triangle*, mark size=2.2pt, mark repeat=3, mark options={solid, mycolor1,rotate=180}]
    table[row sep={\\}]{
            0.0  0.009570379644680443  \\
            1.0  0.005763851757907287  \\
            2.0  0.003670515808875858  \\
            3.0  0.002354401262862156  \\
            4.0  0.0015035034249445212  \\
            5.0  0.0009586935305191378  \\
            6.0  0.0006111898612796092  \\
            7.0  0.0003896681042389785  \\
            8.0  0.0002484626307142471  \\
            9.0  0.00015844637156441776  \\
            10.0  0.00010105560618660883  \\
    };
\addlegendentry {\ altern. eaRGD\qquad}

\addplot[color=mycolor2, line width=1.5, solid, mark=star, mark size=3.5pt, mark repeat=1.5, mark options={solid, mycolor2}]
        table[row sep={\\}]{
            0.0  0.009570379644680443  \\
            1.0  0.00041216949470411924  \\
            2.0  2.9647795011433045e-5  \\
            3.0  1.5696791841558317e-6  \\
            4.0  9.40905430439428e-8  \\
            5.0  7.466262240413793e-9  \\
        };
\addlegendentry {\ altern. LgrRGD\qquad}

\addplot[color=mycolor5, line width=1.25, dashdotted, mark=square*, mark size=2.2pt, mark repeat=2, mark options={solid, mycolor5}]
    table[row sep={\\}]{
            0.0  0.009570379644680443  \\
            1.0  0.00047033295729495333  \\
            2.0  3.512926935766896e-5  \\
            3.0  1.40858588529358e-6  \\
            4.0  8.618384583768708e-8  \\
            5.0  4.797352179367817e-9  \\
    };
\addlegendentry {\ RN\qquad}

\addplot[color=mycolor4, line width=1.25, dashed, mark=diamond*, mark size=2.7pt, mark repeat=1.2, mark options={solid, mycolor4}]
        table[row sep={\\}]{
            0.0  0.009570379644680443  \\
            1.0  0.0004806866135227447  \\
            2.0  4.810225176500962e-5  \\
            3.0  1.991117609142943e-6  \\
            4.0  2.2701323161646077e-7  \\
            5.0  8.026587675744231e-9  \\
        };
\addlegendentry {\ regRN}
\addplot[color=mycolor1_, line width=1.25, dotted, mark=triangle*, mark size=2.2pt, mark repeat=8, mark options={solid, mycolor1_}]
        table[row sep={\\}]{
            -1.0  1e-10 \\
        };
\addlegendentry {\ eaRGD\qquad}
        \addplot[color=mycolor2_, line width=1.5, solid, mark=*, mark size=1.8pt, mark repeat=8, mark options={solid, mycolor2_}]
        table[row sep={\\}]{
            -1.0  1e-10 \\
        };
\addlegendentry {\ LgrRGD\qquad}
\end{axis}

%% file: img/results/1D/convergence/beta=100.tex

\begin{axis}[
width=55.0mm, height=45.0mm, 
at={(50.0mm,0.0mm)},
xlabel={iteration steps}, 
xmin=0, xmax=99, xticklabels={$0$, $30$, $60$, $90$}, xtick={0.0, 30.0, 60.0, 90.0}, 
ylabel={}, 
ymode=log, log basis y=10, 
ymin= 6e-9, ymax=0.015, 
yticklabels={$10^{-8}$, $10^{-5}$, $10^{-2}$}, ytick={1e-8, 1e-5, 1e-2}, 
]
\addplot[color=mycolor1, line width=1.25, dotted, mark=triangle*, mark size=1.8pt, mark repeat=10, mark options={solid, mycolor1,rotate=180}]
        table[row sep={\\}]{
            5.0  0.003576285867490417  \\
            6.0  0.003225677374647564  \\
            7.0  0.002954390570500354  \\
            8.0  0.0027429211828795238  \\
            9.0  0.0025770498961666726  \\
            10.0  0.002446183547431001  \\
            11.0  0.002342302894618506  \\
            12.0  0.0022592644787341066  \\
            13.0  0.00219232706286442  \\
            14.0  0.002137818457696913  \\
            15.0  0.0020928919481288536  \\
            16.0  0.0020553423703936185  \\
            17.0  0.0020234641623161982  \\
            18.0  0.001995940530272529  \\
            19.0  0.0019717565660289853  \\
            20.0  0.0019501311813345161  \\
            21.0  0.0019304639409010462  \\
            22.0  0.001912293688022265  \\
            23.0  0.0018952664667479435  \\
            24.0  0.0018791107332543594  \\
            25.0  0.0018636182517874622  \\
            26.0  0.0018486294022867361  \\
            27.0  0.00183402189822513  \\
            28.0  0.0018197021318803372  \\
            29.0  0.0018055985390546356  \\
            30.0  0.0017916565129430491  \\
            31.0  0.0017778345048582779  \\
            32.0  0.0017641010334351393  \\
            33.0  0.0017504323888865784  \\
            34.0  0.0017368108689315312  \\
            35.0  0.0017232234215033555  \\
            36.0  0.0017096605987780321  \\
            37.0  0.0016961157496838514  \\
            38.0  0.001682584395182944  \\
            39.0  0.0016690637439152763  \\
            40.0  0.001655552315668284  \\
            41.0  0.0016420496479141377  \\
            42.0  0.0016285560664215434  \\
            43.0  0.001615072505399222  \\
            44.0  0.0016016003660814153  \\
            45.0  0.0015881414051161613  \\
            46.0  0.0015746976462811984  \\
            47.0  0.0015612713103731741  \\
            48.0  0.0015478647594083044  \\
            49.0  0.001534480452081381  \\
            50.0  0.001521120908131039  \\
            51.0  0.0015538535075413642  \\
            52.0  0.0015150634223195685  \\
            53.0  0.001492715379646586  \\
            54.0  0.0014750981199344265  \\
            55.0  0.0014595355231623865  \\
            56.0  0.0014449739552962393  \\
            57.0  0.0014757825713615346  \\
            58.0  0.0014377264493082977  \\
            59.0  0.0014156409167069349  \\
            60.0  0.0013981846154511508  \\
            61.0  0.0013827781919944309  \\
            62.0  0.0014118735055672236  \\
            63.0  0.0013745728288156636  \\
            64.0  0.0013528220124072534  \\
            65.0  0.0013355955323189619  \\
            66.0  0.0013203994310234286  \\
            67.0  0.001349163159864791  \\
            68.0  0.001312721121455353  \\
            69.0  0.0012913964349240696  \\
            70.0  0.0012744756497273927  \\
            71.0  0.0013001285957588476  \\
            72.0  0.0012644838853008015  \\
            73.0  0.0012435590223233684  \\
            74.0  0.0012269308784383422  \\
            75.0  0.0012521039845053508  \\
            76.0  0.0012172883125631795  \\
            77.0  0.0011968043528618505  \\
            78.0  0.0011805048262828277  \\
            79.0  0.00120515200319961  \\
            80.0  0.0011712099416117739  \\
            81.0  0.0011512026147056475  \\
            82.0  0.0011352623264102963  \\
            83.0  0.001159332238563156  \\
            84.0  0.001126303450526244  \\
            85.0  0.0011068025985782202  \\
            86.0  0.0010912475796359597  \\
            87.0  0.0011146951153908327  \\
            88.0  0.001082611247833857  \\
            89.0  0.0010636410734944065  \\
            90.0  0.001048492910622803  \\
            91.0  0.0010712804536234456  \\
            92.0  0.001040165001325215  \\
            93.0  0.0010217443854609106  \\
            94.0  0.0010070204317294783  \\
            95.0  0.001029117644169592  \\
            96.0  0.0009989865219664565  \\
            97.0  0.0009811294640547587  \\
            98.0  0.0009668431451133601  \\
            99.0  0.0009882264809811798  \\
            100.0  0.0009590887752061874  \\
        };
\addplot[color=mycolor2, line width=1.5, solid, mark=star, mark size=3.5pt, mark repeat=10, mark options={solid, mycolor2}]
        table[row sep={\\}]{
            0.0  0.008045383885828997  \\
            1.0  0.008347404484215087  \\
            2.0  0.0057022311964384596  \\
            3.0  0.004995611126145978  \\
            4.0  0.0046079411914758325  \\
            5.0  0.00456169544016185  \\
            6.0  0.004469399444822962  \\
            7.0  0.004403998171241897  \\
            8.0  0.004308415699287617  \\
            9.0  0.0042317167954396065  \\
            10.0  0.004151836905996246  \\
            11.0  0.004065333033249291  \\
            12.0  0.003972449822535045  \\
            13.0  0.003904435582372348  \\
            14.0  0.0037107620941570227  \\
            15.0  0.0036879908082921888  \\
            16.0  0.0035757360656974765  \\
            17.0  0.003426806217882893  \\
            18.0  0.0032668519912623815  \\
            19.0  0.003059590087286205  \\
            20.0  0.0028166217844235106  \\
            21.0  0.0023700891026054685  \\
            22.0  0.0020230961012791255  \\
            23.0  0.0014723771220831604  \\
            24.0  0.00100478882663966  \\
            25.0  0.000522376961036092  \\
            26.0  0.00026342544973405644  \\
            27.0  0.0001275180594412962  \\
            28.0  6.207768439345548e-5  \\
            29.0  3.7652619896627266e-5  \\
            30.0  1.942462919460222e-5  \\
            31.0  9.560998036714853e-6  \\
            32.0  5.585248758513261e-6  \\
            33.0  2.9066612262935793e-6  \\
            34.0  1.8512939786678977e-6  \\
            35.0  1.19440815105565e-6  \\
            36.0  9.692211796013667e-7  \\
            37.0  8.620788209654124e-7  \\
            38.0  8.24727369288222e-7  \\
            39.0  7.859546315168097e-7  \\
            40.0  7.782564454745432e-7  \\
            41.0  7.484738218897739e-7  \\
            42.0  7.443528520304317e-7  \\
            43.0  7.170172299172506e-7  \\
            44.0  7.138339552295184e-7  \\
            45.0  6.878619518129779e-7  \\
            46.0  6.851188525941839e-7  \\
            47.0  6.60318688460866e-7  \\
            48.0  6.578403160302803e-7  \\
            49.0  6.341209269213708e-7  \\
            50.0  6.318289987676989e-7  \\
            51.0  6.091131556699255e-7  \\
            52.0  6.069652234476511e-7  \\
            53.0  5.851870614198238e-7  \\
            54.0  5.831569815682549e-7  \\
            55.0  5.622611845098009e-7  \\
            56.0  5.603316171521194e-7  \\
            57.0  5.402712302870032e-7  \\
            58.0  5.384301144174628e-7  \\
            59.0  5.191643581612902e-7  \\
            60.0  5.174029961231639e-7  \\
            61.0  4.98895621974922e-7  \\
            62.0  4.972075151685357e-7  \\
            63.0  4.794256403933134e-7  \\
            64.0  4.778057805220437e-7  \\
            65.0  4.607190947951814e-7  \\
            66.0  4.591634758039263e-7  \\
            67.0  4.427436885276233e-7  \\
            68.0  4.412489905139397e-7  \\
            69.0  4.2546945178293953e-7  \\
            70.0  4.2403282919274003e-7  \\
            71.0  4.088682582936312e-7  \\
            72.0  4.0748720125579613e-7  \\
            73.0  3.929135134314745e-7  \\
            74.0  3.915857578905061e-7  \\
            75.0  3.7757992298464976e-7  \\
            76.0  3.763033598615273e-7  \\
            77.0  3.628433169588575e-7  \\
            78.0  3.6161598009177084e-7  \\
            79.0  3.486805563937508e-7  \\
            80.0  3.475005693308381e-7  \\
            81.0  3.3506941964651627e-7  \\
            82.0  3.339350101975278e-7  \\
            83.0  3.2198857558352323e-7  \\
            84.0  3.2089802634633913e-7  \\
            85.0  3.0941751098676855e-7  \\
            86.0  3.083691771836331e-7  \\
            87.0  2.973364937727603e-7  \\
            88.0  2.9632878523378024e-7  \\
            89.0  2.8572654522978015e-7  \\
            90.0  2.8475792685102344e-7  \\
            91.0  2.74569398494615e-7  \\
            92.0  2.7363839070710193e-7  \\
            93.0  2.638474905655019e-7  \\
            94.0  2.629526633980525e-7  \\
            95.0  2.5354392237921614e-7  \\
            96.0  2.5268389060632613e-7  \\
            97.0  2.436424281513732e-7  \\
            98.0  2.428158622964698e-7  \\
            99.0  2.3412737844417935e-7  \\
            100.0  2.3333299623533617e-7  \\
        };
\addplot[color=mycolor5, line width=1.25, dashdotted]
        table[row sep={\\}]{
            -1.0  1e-10  \\
        };
\addplot[color=mycolor4, line width=1.25, dashed, mark=diamond*, mark size=2.7pt, mark repeat=5, mark options={solid, mycolor4}]
        table[row sep={\\}]{
            0.0  0.008045383885828997  \\
            1.0  0.01412708161274717  \\
            2.0  0.002036714647915366  \\
            3.0  0.0022725323472872734  \\
            4.0  0.0006985389076402384  \\
            5.0  0.0003273527147187127  \\
            6.0  0.00011283656673244289  \\
            7.0  4.443950812502576e-5  \\
            8.0  1.961253179586072e-5  \\
            9.0  9.327911718705009e-6  \\
            10.0  4.679303525931583e-6  \\
            11.0  2.2533541677220733e-6  \\
            12.0  1.1046469023864055e-6  \\
            13.0  5.097281795402868e-7  \\
            14.0  2.4550121088839925e-7  \\
            15.0  1.1903386416376346e-7  \\
            16.0  5.7809283527150074e-8  \\
            17.0  2.82391851768005e-8  \\
            18.0  1.3744548038767197e-8  \\
            19.0  6.764469588581895e-9  \\
        };
\addplot[color=mycolor1_, line width=1.25, dotted, mark=triangle*, mark size=2.2pt, mark repeat=10, mark options={solid, mycolor1_}]
        table[row sep={\\}]{
            0.0  0.008045383885828997  \\
            1.0  0.006570283763879971  \\
            2.0  0.0054992680180472936  \\
            3.0  0.004719062068677781  \\
            4.0  0.004158772534307144  \\
            5.0  0.0037160021213735192  \\
            6.0  0.0033718517253839182  \\
            7.0  0.0031009020444539813  \\
            8.0  0.0028855047114012826  \\
            9.0  0.0027127293095905365  \\
            10.0  0.002573250760460021  \\
            11.0  0.002460020894354392  \\
            12.0  0.0023676325867381786  \\
            13.0  0.0022918489318351667  \\
            14.0  0.0022293152505623924  \\
            15.0  0.002177355585814609  \\
            16.0  0.0021338292533078004  \\
            17.0  0.002097022479281108  \\
            18.0  0.0020655643223712744  \\
            19.0  0.0020383595889590577  \\
            20.0  0.0020145347186904254  \\
            21.0  0.001993393911494944  \\
            22.0  0.001974383580022195  \\
            23.0  0.0019570636282942775  \\
            24.0  0.001941084335241926  \\
            25.0  0.0019261678222473363  \\
            26.0  0.001912093249038317  \\
            27.0  0.0018986850239444983  \\
            28.0  0.0018858034380481578  \\
            29.0  0.0018733372390568681  \\
            30.0  0.0018611977513687734  \\
            31.0  0.0018493142246989377  \\
            32.0  0.0018376301563337754  \\
            33.0  0.001826100383595712  \\
            34.0  0.0018146887846425337  \\
            35.0  0.0018033664592974877  \\
            36.0  0.0017921102884323845  \\
            37.0  0.0017809017918207148  \\
            38.0  0.0017697262212752047  \\
            39.0  0.0017585718394033689  \\
            40.0  0.0017474293448243327  \\
            41.0  0.00173629141308075  \\
            42.0  0.0017251523290550187  \\
            43.0  0.0017140076918429632  \\
            44.0  0.0017028541770840492  \\
            45.0  0.0016916893449943214  \\
            46.0  0.0016805114847365812  \\
            47.0  0.001669319487863643  \\
            48.0  0.0016581127449531484  \\
            49.0  0.0016468910609269596  \\
            50.0  0.001635654585341635  \\
            51.0  0.0016244037548227494  \\
            52.0  0.0016131392452948484  \\
            53.0  0.0016018619321924623  \\
            54.0  0.0015905728571803464  \\
            55.0  0.0015792732002160397  \\
            56.0  0.0015679642559641596  \\
            57.0  0.0015566474138614544  \\
            58.0  0.0015897822844064202  \\
            59.0  0.0015534706756049211  \\
            60.0  0.0015332087704875634  \\
            61.0  0.0015176812102886661  \\
            62.0  0.0015041025988159346  \\
            63.0  0.001491476444364475  \\
            64.0  0.0015225559965248501  \\
            65.0  0.0014867733607644954  \\
            66.0  0.0014665868557521729  \\
            67.0  0.0014510061505383842  \\
            68.0  0.0014373647536864156  \\
            69.0  0.0014666914979727197  \\
            70.0  0.0014314347954232406  \\
            71.0  0.0014113865918935933  \\
            72.0  0.0013958305420097936  \\
            73.0  0.0013821979690926908  \\
            74.0  0.0014110864511754633  \\
            75.0  0.0013764311516488142  \\
            76.0  0.0013565876485437764  \\
            77.0  0.001341116986568164  \\
            78.0  0.0013275476365890702  \\
            79.0  0.0013559346060378855  \\
            80.0  0.0013219578823156863  \\
            81.0  0.0013023825935500834  \\
            82.0  0.0012870553331554829  \\
            83.0  0.0013123895802545436  \\
            84.0  0.0012790162243498714  \\
            85.0  0.0012596983239459303  \\
            86.0  0.00124452568777115  \\
            87.0  0.0012693377353168115  \\
            88.0  0.0012366038917001128  \\
            89.0  0.001217578000294097  \\
            90.0  0.0012025919719619211  \\
            91.0  0.0012268646319125275  \\
            92.0  0.0011948041743050086  \\
            93.0  0.0011761013970145555  \\
            94.0  0.0011613314405895434  \\
            95.0  0.0011850405902201118  \\
            96.0  0.0011536857344256986  \\
            97.0  0.0011353344885408883  \\
            98.0  0.0011208077358817931  \\
            99.0  0.0011439301902589423  \\
            100.0  0.0011133095397962498  \\
    };
\addplot[color=mycolor2_, line width=1.5, solid, mark=*, mark size=1.8pt, mark repeat=8, mark options={solid, mycolor2_}]
        table[row sep={\\}]{
            0.0  0.008045383885828997  \\
            1.0  0.006855840138860518  \\
            2.0  0.003274057646719527  \\
            3.0  0.004936933469692815  \\
            4.0  0.002621706336930459  \\
            5.0  0.00429334537034308  \\
            6.0  0.002375805837518175  \\
            7.0  0.004052419683777918  \\
            8.0  0.0022415699989482936  \\
            9.0  0.003783020348634075  \\
            10.0  0.0021376581456944207  \\
            11.0  0.003599616339013082  \\
            12.0  0.0020715218471236718  \\
            13.0  0.003435296313129548  \\
            14.0  0.0020234931545086545  \\
            15.0  0.003265875244690944  \\
            16.0  0.00197190761254614  \\
            17.0  0.0031493279482587154  \\
            18.0  0.0019063924241360826  \\
            19.0  0.003041051209821407  \\
            20.0  0.001866985631077078  \\
            21.0  0.0029288877207970204  \\
            22.0  0.0018282243519621825  \\
            23.0  0.002813536869360219  \\
            24.0  0.0019014460937649366  \\
            25.0  0.0025856287465108913  \\
            26.0  0.001930159068199619  \\
            27.0  0.0024243839258838427  \\
            28.0  0.0019258073581962735  \\
            29.0  0.0022690681844866607  \\
            30.0  0.0019250044093857626  \\
            31.0  0.002190475387026035  \\
            32.0  0.0018782296488634147  \\
            33.0  0.0020722039370562133  \\
            34.0  0.0018063956014208722  \\
            35.0  0.0019330048028378373  \\
            36.0  0.0017144261379499809  \\
            37.0  0.0017660493513101565  \\
            38.0  0.0015899341310186776  \\
            39.0  0.0015656364226049179  \\
            40.0  0.0014006910795511872  \\
            41.0  0.0012959653738005502  \\
            42.0  0.0011003282919922735  \\
            43.0  0.0009124372537906658  \\
            44.0  0.000697505085181033  \\
            45.0  0.0004799549340908742  \\
            46.0  0.00034818960731081307  \\
            47.0  0.00024128745733676096  \\
            48.0  0.00019126407983688509  \\
            49.0  0.00014784284468706722  \\
            50.0  0.00012014164989462958  \\
            51.0  0.00010499787324915403  \\
            52.0  9.335363660597742e-5  \\
            53.0  8.312257816385815e-5  \\
            54.0  7.698987823656406e-5  \\
            55.0  7.173937959523179e-5  \\
            56.0  6.893917644000594e-5  \\
            57.0  6.58738100652779e-5  \\
            58.0  6.610053046281561e-5  \\
            59.0  6.582723122686383e-5  \\
            60.0  6.632794586418888e-5  \\
            61.0  6.631100096097332e-5  \\
            62.0  6.485544885500907e-5  \\
            63.0  6.622390567067015e-5  \\
            64.0  6.511198282350248e-5  \\
            65.0  6.629493314737755e-5  \\
            66.0  6.547014854365003e-5  \\
            67.0  6.426518489206702e-5  \\
            68.0  6.65109429742087e-5  \\
            69.0  6.473758067230442e-5  \\
            70.0  6.448474427303399e-5  \\
            71.0  6.353298304513855e-5  \\
            72.0  6.538970772077537e-5  \\
            73.0  6.378969789907043e-5  \\
            74.0  6.343950970020084e-5  \\
            75.0  6.253616131992339e-5  \\
            76.0  6.2294875643396e-5  \\
            77.0  6.139795620178412e-5  \\
            78.0  6.118541444548451e-5  \\
            79.0  6.0277275773008086e-5  \\
            80.0  6.006902591672991e-5  \\
            81.0  5.916010123801156e-5  \\
            82.0  5.8953961218886634e-5  \\
            83.0  5.805176727012129e-5  \\
            84.0  5.7847643262537444e-5  \\
            85.0  5.695610683412532e-5  \\
            86.0  5.675401132782463e-5  \\
            87.0  5.587516115476285e-5  \\
            88.0  5.5675163055604214e-5  \\
            89.0  5.481009313202384e-5  \\
            90.0  5.461232102902569e-5  \\
            91.0  5.376163481459022e-5  \\
            92.0  5.3566247404651204e-5  \\
            93.0  5.27302668456666e-5  \\
            94.0  5.253742413828479e-5  \\
            95.0  5.171629649299184e-5  \\
            96.0  5.152614214310457e-5  \\
            97.0  5.07198997693686e-5  \\
            98.0  5.053255303673242e-5  \\
            99.0  4.9741149581469603e-5  \\
            100.0  4.955670339670468e-5  \\
        };
\end{axis}

%% file: img/results/1D/convergence/beta=1000.tex

\begin{axis}[
width=55.0mm, height=45.0mm, 
at={(100.0mm,0.0mm)},
xlabel={iteration steps}, 
xmin=0, xmax=499, xticklabels={$0$, $150$, $300$, $450$}, xtick={0.0, 150.0, 300.0, 450.0}, 
ylabel={}, 
ymode=log, log basis y=10, 
ymin= 6e-9, ymax=0.015, 
yticklabels={$10^{-8}$, $10^{-5}$, $10^{-2}$}, ytick={1e-8, 1e-5, 1e-2}, 
]
\addplot[color=mycolor1, line width=1.25, dotted,mark=triangle*, mark size=2.2pt, mark repeat=50, mark options={solid, mycolor1,rotate=180}]
        table[row sep={\\}]{
            0.0  0.009902534395925953 \\
            1.0  0.009422180654373363  \\
            2.0  0.009131128545789555  \\
            3.0  0.00896969541543336  \\
            4.0  0.008833493932856323  \\
            5.0  0.00864026480446621  \\
            6.0  0.00854654580771786  \\
            7.0  0.008393651314159629  \\
            8.0  0.008211514492737134  \\
            9.0  0.00812226100468239  \\
            10.0  0.00799788492132622  \\
            11.0  0.007901782546105197  \\
            12.0  0.0078078263153386965  \\
            13.0  0.00775586266776169  \\
            14.0  0.007644263388518556  \\
            15.0  0.007617866632774783  \\
            16.0  0.007474946071914801  \\
            17.0  0.007439754642291112  \\
            18.0  0.007299735700253019  \\
            19.0  0.007264195403302539  \\
            20.0  0.007128901831244147  \\
            21.0  0.007091882667706243  \\
            22.0  0.006961992004915707  \\
            23.0  0.006923521752023753  \\
            24.0  0.006843775702131765  \\
            25.0  0.006714599484451191  \\
            26.0  0.006680523367167693  \\
            27.0  0.006602544920591725  \\
            28.0  0.006524654883917358  \\
            29.0  0.006446222680439339  \\
            30.0  0.006368990229706433  \\
            31.0  0.006173543631973959  \\
            32.0  0.006118254144029005  \\
            33.0  0.0061458967931291965  \\
            34.0  0.006050704832223252  \\
            35.0  0.005988182413332543  \\
            36.0  0.005916353185134829  \\
            37.0  0.005854758404677369  \\
            38.0  0.0056826889337760695  \\
            39.0  0.005614348866713898  \\
            40.0  0.005666144334873955  \\
            41.0  0.005500109557632471  \\
            42.0  0.005431285427401502  \\
            43.0  0.005426377195674727  \\
            44.0  0.005372235297290527  \\
            45.0  0.005230769159390484  \\
            46.0  0.005192474221086682  \\
            47.0  0.005107264031549487  \\
            48.0  0.005049370351515192  \\
            49.0  0.005048914167292168  \\
            50.0  0.00493321106811736  \\
            51.0  0.004874964529484835  \\
            52.0  0.0048717462539565335  \\
            53.0  0.004763786906617022  \\
            54.0  0.004704687897928194  \\
            55.0  0.004706004572234113  \\
            56.0  0.004601880274657768  \\
            57.0  0.004545120966591362  \\
            58.0  0.004549223524461336  \\
            59.0  0.004446718707659217  \\
            60.0  0.0043922969942761046  \\
            61.0  0.004339186032164113  \\
            62.0  0.004294151754856431  \\
            63.0  0.004246337648862762  \\
            64.0  0.004201896472880349  \\
            65.0  0.004155514807522395  \\
            66.0  0.004111360817154484  \\
            67.0  0.0040661485462694935  \\
            68.0  0.00402264609142711  \\
            69.0  0.003978702927511761  \\
            70.0  0.0039361386307890794  \\
            71.0  0.003893548520448479  \\
            72.0  0.0038520907407785404  \\
            73.0  0.003810870008463214  \\
            74.0  0.0037705866150321912  \\
            75.0  0.0037306860681969927  \\
            76.0  0.0036915634414189566  \\
            77.0  0.003652880286593573  \\
            78.0  0.003614837343220211  \\
            79.0  0.003577225242880959  \\
            80.0  0.0035401285384068578  \\
            81.0  0.003503414755010062  \\
            82.0  0.0034671087474049595  \\
            83.0  0.0034311305887182717  \\
            84.0  0.0033954859681446263  \\
            85.0  0.003360136992912248  \\
            86.0  0.0033250923676190775  \\
            87.0  0.0032903448708437124  \\
            88.0  0.0032559096995715487  \\
            89.0  0.0032217950142813404  \\
            90.0  0.0031880167113994815  \\
            91.0  0.0031545858399446287  \\
            92.0  0.00312151420430283  \\
            93.0  0.0030888098895643108  \\
            94.0  0.0030564795359543455  \\
            95.0  0.0030245271616235097  \\
            96.0  0.0029929551752312785  \\
            97.0  0.0029617642733551527  \\
            98.0  0.0029309539373432846  \\
            99.0  0.002900522579703799  \\
            100.0  0.0028704678129162706  \\
            101.0  0.0028407866021114914  \\
            102.0  0.0028114754219780947  \\
            103.0  0.0027825303679744557  \\
            104.0  0.0027539472519589703  \\
            105.0  0.002725721675390617  \\
            106.0  0.002697849088721814  \\
            107.0  0.0026703248379965874  \\
            108.0  0.002643144201981362  \\
            109.0  0.0026163024215142875  \\
            110.0  0.0025897947227803003  \\
            111.0  0.002563616335728445  \\
            112.0  0.0025377625086288203  \\
            113.0  0.0025122285195619525  \\
            114.0  0.002487009685473201  \\
            115.0  0.0024621013692412473  \\
            116.0  0.002437498985219959  \\
            117.0  0.002413198003514704  \\
            118.0  0.002389193953260365  \\
            119.0  0.0023654824250947957  \\
            120.0  0.002342059072969108  \\
            121.0  0.0023189196154318388  \\
            122.0  0.002296059836492967  \\
            123.0  0.0022734755861180712  \\
            124.0  0.0022511627804767727  \\
            125.0  0.002229117401932853  \\
            126.0  0.002207335498867732  \\
            127.0  0.0021858131853309403  \\
            128.0  0.0021645466405720553  \\
            129.0  0.002143532108470065  \\
            130.0  0.0021227658968734414  \\
            131.0  0.0021022443768791864  \\
            132.0  0.0020819639820560632  \\
            133.0  0.002061921207610043  \\
            134.0  0.0020421126095288237  \\
            135.0  0.0020225348036862823  \\
            136.0  0.0020031844649240015  \\
            137.0  0.001984058326116551  \\
            138.0  0.0019651531772297095  \\
            139.0  0.0019464658643649924  \\
            140.0  0.0019279932887934109  \\
            141.0  0.0019097324060016493  \\
            142.0  0.0018916802247096924  \\
            143.0  0.0018738338059208347  \\
            144.0  0.001856190261966223  \\
            145.0  0.0018387467555370251  \\
            146.0  0.001821500498749422  \\
            147.0  0.0018044487521966489  \\
            148.0  0.00178758882403133  \\
            149.0  0.0017709180690333418  \\
            150.0  0.0017544338877090386  \\
            151.0  0.001738133725389685  \\
            152.0  0.001722015071355707  \\
            153.0  0.0017060754579523636  \\
            154.0  0.001690312459742981  \\
            155.0  0.0016747236926593035  \\
            156.0  0.0016593068131669623  \\
            157.0  0.0016440595174584164  \\
            158.0  0.0016289795406432823  \\
            159.0  0.0016140646559614138  \\
            160.0  0.0015993126740117  \\
            161.0  0.0015847214419952588  \\
            162.0  0.0015702888429666468  \\
            163.0  0.0015560127951092326  \\
            164.0  0.0015418912510149802  \\
            165.0  0.0015279221969948223  \\
            166.0  0.0015141036523776156  \\
            167.0  0.001500433668851432  \\
            168.0  0.0014869103298045456  \\
            169.0  0.0014735317496701492  \\
            170.0  0.0014602960733163697  \\
            171.0  0.001447201475411001  \\
            172.0  0.0014342461598323534  \\
            173.0  0.001421428359075141  \\
            174.0  0.0014087463336725525  \\
            175.0  0.0013961983716361563  \\
            176.0  0.0013837827879032408  \\
            177.0  0.0013714979238025305  \\
            178.0  0.0013593421465182838  \\
            179.0  0.0013473138485866152  \\
            180.0  0.0013354114473896654  \\
            181.0  0.0013236333846613855  \\
            182.0  0.001311978126008888  \\
            183.0  0.0013004441604481348  \\
            184.0  0.0012890299999406374  \\
            185.0  0.0012777341789452252  \\
            186.0  0.001266555253987828  \\
            187.0  0.001255491803230219  \\
            188.0  0.0012445424260490768  \\
            189.0  0.0012337057426325426  \\
            190.0  0.0012229803935866112  \\
            191.0  0.001212365039535415  \\
            192.0  0.0012018583607472943  \\
            193.0  0.0011914590567706786  \\
            194.0  0.0011811658460550056  \\
            195.0  0.0011709774656141452  \\
            196.0  0.0011608926706665125  \\
            197.0  0.0011509102343060928  \\
            198.0  0.001141028947165167  \\
            199.0  0.0011312476171008714  \\
            200.0  0.001121565068870695  \\
            201.0  0.0011119801438294775  \\
            202.0  0.0011024916996254558  \\
            203.0  0.0010930986099092744  \\
            204.0  0.0010837997640479418  \\
            205.0  0.0010745940668370913  \\
            206.0  0.0010654804382301448  \\
            207.0  0.0010564578130769125  \\
            208.0  0.0010475251408448624  \\
            209.0  0.0010386813853837207  \\
            210.0  0.0010299255246560743  \\
            211.0  0.0010212565505031649  \\
            212.0  0.0010126734684046583  \\
            213.0  0.0010041752972388298  \\
            214.0  0.0009957610690580839  \\
            215.0  0.0009874298288615566  \\
            216.0  0.0009791806343866709  \\
            217.0  0.0009710125558786773  \\
            218.0  0.0009629246758885827  \\
            219.0  0.0009549160890743691  \\
            220.0  0.000946985901994559  \\
            221.0  0.0009391332329054099  \\
            222.0  0.0009313572115831355  \\
            223.0  0.0009236569791229874  \\
            224.0  0.0009160316877645  \\
            225.0  0.0009084805007101267  \\
            226.0  0.0009010025919397705  \\
            227.0  0.0008935971460518074  \\
            228.0  0.0008862633580865253  \\
            229.0  0.000879000433359323  \\
            230.0  0.000871807587305339  \\
            231.0  0.0008646840453163486  \\
            232.0  0.0008576290425843289  \\
            233.0  0.0008506418239536007  \\
            234.0  0.00084372164376931  \\
            235.0  0.0008368677657339245  \\
            236.0  0.0008300794627541106  \\
            237.0  0.0008233560168180817  \\
            238.0  0.000816696718840613  \\
            239.0  0.0008101008685430704  \\
            240.0  0.0008035677743067485  \\
            241.0  0.0007970967530569402  \\
            242.0  0.000790687130125216  \\
            243.0  0.000784338239134815  \\
            244.0  0.0007780494218639757  \\
            245.0  0.0007718200281426533  \\
            246.0  0.0007656494157236543  \\
            247.0  0.0007595369501640048  \\
            248.0  0.0007534820047228436  \\
            249.0  0.0007474839602402268  \\
            250.0  0.0007415422050233084  \\
            251.0  0.0007356561347498993  \\
            252.0  0.0007298251523455293  \\
            253.0  0.0007240486678883074  \\
            254.0  0.0007183260984929263  \\
            255.0  0.0007126568682116643  \\
            256.0  0.0007070404079237134  \\
            257.0  0.0007014761552335696  \\
            258.0  0.0006959635543491685  \\
            259.0  0.0006905020559857523  \\
            260.0  0.0006850911172375062  \\
            261.0  0.0006797302014648883  \\
            262.0  0.0006744187781538286  \\
            263.0  0.0006691563227916713  \\
            264.0  0.000663942316692291  \\
            265.0  0.0006587762468345547  \\
            266.0  0.0006536576056552779  \\
            267.0  0.0006485858908210851  \\
            268.0  0.0006435606049424267  \\
            269.0  0.0006385812552424811  \\
            270.0  0.0006336473531461095  \\
            271.0  0.0006287584137618637  \\
            272.0  0.0006239139552511557  \\
            273.0  0.0006191134980215127  \\
            274.0  0.0006143565636915811  \\
            275.0  0.0006096426737964698  \\
            276.0  0.0006049713480880128  \\
            277.0  0.0006003421023890707  \\
            278.0  0.0005957544457888815  \\
            279.0  0.0005912078770283417  \\
            280.0  0.000586701879817719  \\
            281.0  0.0005822359167556624  \\
            282.0  0.0005778094214329044  \\
            283.0  0.0005734217881747516  \\
            284.0  0.000569072358706328  \\
            285.0  0.000564760404809823  \\
            286.0  0.0005604851058023835  \\
            287.0  0.0005562455192799774  \\
            288.0  0.0005520405431477455  \\
            289.0  0.0005478688664942436  \\
            290.0  0.0005437289060935572  \\
            291.0  0.0005396187247210391  \\
            292.0  0.0005355359265825752  \\
            293.0  0.0005314775244722691  \\
            294.0  0.0005274397729142517  \\
            295.0  0.0005234179620256388  \\
            296.0  0.0005194061692813866  \\
            297.0  0.0005153969726276685  \\
            298.0  0.0005113811413299502  \\
            299.0  0.0005087534268193255  \\
            300.0  0.0005041684008523843  \\
            301.0  0.0005002151901937006  \\
            302.0  0.0004963016084710628  \\
            303.0  0.0004937591892326687  \\
            304.0  0.0004893713146876406  \\
            305.0  0.0004855407141407523  \\
            306.0  0.00048176085634862284  \\
            307.0  0.000479308067665715  \\
            308.0  0.0004751035150025617  \\
            309.0  0.0004713926963968244  \\
            310.0  0.00046773612628277797  \\
            311.0  0.0004653689411383326  \\
            312.0  0.0004613342376045504  \\
            313.0  0.00045773949594168987  \\
            314.0  0.0004553661420070214  \\
            315.0  0.00045139160362048724  \\
            316.0  0.0004478986484164538  \\
            317.0  0.0004444771329177418  \\
            318.0  0.00044222410287699106  \\
            319.0  0.0004384217865175175  \\
            320.0  0.0004350340044027253  \\
            321.0  0.00043171037460841755  \\
            322.0  0.0004295402678290765  \\
            323.0  0.00042589280382542385  \\
            324.0  0.00042260653661810554  \\
            325.0  0.0004204362465786252  \\
            326.0  0.0004168423949741109  \\
            327.0  0.00041364747529418905  \\
            328.0  0.00041052256092347024  \\
            329.0  0.0004084626282583373  \\
            330.0  0.00040502693641979425  \\
            331.0  0.0004019248588520869  \\
            332.0  0.0003988836319537794  \\
            333.0  0.00039690110747316325  \\
            334.0  0.0003936075873695105  \\
            335.0  0.00039059572352596585  \\
            336.0  0.0003886128461757114  \\
            337.0  0.0003853622021774592  \\
            338.0  0.00038243301003712836  \\
            339.0  0.00037956828471785167  \\
            340.0  0.0003776870154829765  \\
            341.0  0.00037458267313492317  \\
            342.0  0.00037173641480828587  \\
            343.0  0.00036985940737936825  \\
            344.0  0.0003667992291896516  \\
            345.0  0.00036402955063486586  \\
            346.0  0.0003613239099289455  \\
            347.0  0.0003595432365352832  \\
            348.0  0.0003566224534944746  \\
            349.0  0.0003539289408325928  \\
            350.0  0.0003521535116641808  \\
            351.0  0.00034927399933781746  \\
            352.0  0.0003466519039814738  \\
            353.0  0.0003440916546230087  \\
            354.0  0.000342407477575622  \\
            355.0  0.0003396608711522978  \\
            356.0  0.0003371092862489619  \\
            357.0  0.00033543019368573483  \\
            358.0  0.0003327207278730835  \\
            359.0  0.0003302360684492169  \\
            360.0  0.0003278103205286302  \\
            361.0  0.0003262177387851685  \\
            362.0  0.0003236351194820147  \\
            363.0  0.0003212160755134947  \\
            364.0  0.0003196279099944063  \\
            365.0  0.000317077846622424  \\
            366.0  0.00031472170563107165  \\
            367.0  0.00031242126605398945  \\
            368.0  0.0003109151420193891  \\
            369.0  0.00030848619614854794  \\
            370.0  0.00030619134054109964  \\
            371.0  0.0003046887543224847  \\
            372.0  0.00030228773006493894  \\
            373.0  0.00030005210320250577  \\
            374.0  0.00029857662959045876  \\
            375.0  0.00029622046462567114  \\
            376.0  0.00029404159593819504  \\
            377.0  0.00029191949033177466  \\
            378.0  0.00029052087057983766  \\
            379.0  0.00028827877223754094  \\
            380.0  0.0002861535665039725  \\
            381.0  0.0002847605752725937  \\
            382.0  0.0002825466676671718  \\
            383.0  0.00028047469778204044  \\
            384.0  0.0002784538746437319  \\
            385.0  0.00027713272341715564  \\
            386.0  0.0002750270758021091  \\
            387.0  0.00027300616899830695  \\
            388.0  0.00027168823851381495  \\
            389.0  0.000269604109066373  \\
            390.0  0.0002676339565509876  \\
            391.0  0.0002663391664280514  \\
            392.0  0.00026429211290419837  \\
            393.0  0.0002623707905743142  \\
            394.0  0.00026049997293202057  \\
            395.0  0.00025927281816553587  \\
            396.0  0.0002573282351736878  \\
            397.0  0.00025545225338692134  \\
            398.0  0.000254229293812599  \\
            399.0  0.00025230543759363145  \\
            400.0  0.00025047545498221337  \\
            401.0  0.00024927408886854285  \\
            402.0  0.00024738443143146305  \\
            403.0  0.0002455988707764593  \\
            404.0  0.00024386121500873724  \\
            405.0  0.00024272242279953064  \\
            406.0  0.00024092883812834077  \\
            407.0  0.00023918410902019349  \\
            408.0  0.00023804919681031874  \\
            409.0  0.00023627360097954275  \\
            410.0  0.00023457097005282637  \\
            411.0  0.00023345575647612654  \\
            412.0  0.0002317109086956992  \\
            413.0  0.00023004903159510334  \\
            414.0  0.00022843192259360196  \\
            415.0  0.00022737478920791492  \\
            416.0  0.00022572027223779259  \\
            417.0  0.00022409551137360075  \\
            418.0  0.00022304163449675765  \\
            419.0  0.0002214020018737171  \\
            420.0  0.0002198159603368996  \\
            421.0  0.00021877990274088994  \\
            422.0  0.0002171675925991543  \\
            423.0  0.0002156190973511455  \\
            424.0  0.00021460295715119992  \\
            425.0  0.0002130213852241339  \\
            426.0  0.00021150863985486837  \\
            427.0  0.00021003906495111808  \\
            428.0  0.00020907539605274342  \\
            429.0  0.0002075768411556433  \\
            430.0  0.00020609600161310688  \\
            431.0  0.0002051360058606727  \\
            432.0  0.000203651467256753  \\
            433.0  0.0002022048990643554  \\
            434.0  0.0002012611127798418  \\
            435.0  0.00019980126050327495  \\
            436.0  0.00019838809442171944  \\
            437.0  0.00019746227478182357  \\
            438.0  0.0001960300071631507  \\
            439.0  0.0001946487523433085  \\
            440.0  0.00019330742770589816  \\
            441.0  0.00019242921275568943  \\
            442.0  0.00019107359903926165  \\
            443.0  0.00018972046312675672  \\
            444.0  0.0001888454511794916  \\
            445.0  0.00018750138348411007  \\
            446.0  0.00018617900177470082  \\
            447.0  0.00018531838574244752  \\
            448.0  0.00018399586498137318  \\
            449.0  0.00018270353718838553  \\
            450.0  0.00018185896787526007  \\
            451.0  0.00018056079270061822  \\
            452.0  0.00017929721117715693  \\
            453.0  0.0001784690512205341  \\
            454.0  0.00017719618524064535  \\
            455.0  0.00017596009160638813  \\
            456.0  0.00017476106272705402  \\
            457.0  0.0001739749889600137  \\
            458.0  0.00017277136403683903  \\
            459.0  0.0001715590337689678  \\
            460.0  0.0001707760587992479  \\
            461.0  0.00016958245348560653  \\
            462.0  0.00016839689795629105  \\
            463.0  0.00016762659762355544  \\
            464.0  0.00016645174796195368  \\
            465.0  0.00016529248423845802  \\
            466.0  0.00016453628431075168  \\
            467.0  0.00016338265348336787  \\
            468.0  0.00016224861466360973  \\
            469.0  0.0001615068322054054  \\
            470.0  0.00016037528563493848  \\
            471.0  0.00015926541189302594  \\
            472.0  0.00015818888127227388  \\
            473.0  0.00015748463762113677  \\
            474.0  0.00015641609941396714  \\
            475.0  0.00015532682771865125  \\
            476.0  0.00015462513599550777  \\
            477.0  0.0001535642706907179  \\
            478.0  0.00015249863030292472  \\
            479.0  0.00015180790526267928  \\
            480.0  0.0001507628676230668  \\
            481.0  0.0001497204849968996  \\
            482.0  0.0001490420669841173  \\
            483.0  0.0001480152515440253  \\
            484.0  0.0001469952066986663  \\
            485.0  0.00014632942213487466  \\
            486.0  0.00014532172164387808  \\
            487.0  0.00014432308451585609  \\
            488.0  0.00014366992967451895  \\
            489.0  0.00014268162610814634  \\
            490.0  0.000141703565222421  \\
            491.0  0.00014075540916132822  \\
            492.0  0.0001401349223655075  \\
            493.0  0.0001392027470728683  \\
            494.0  0.0001382419057586248  \\
            495.0  0.0001376236925982714  \\
            496.0  0.00013669764086564402  \\
            497.0  0.00013575711213241625  \\
            498.0  0.00013514830553909985  \\
            499.0  0.00013423555355127325  \\
            500.0  0.00013331509224160455  \\
        };
\addplot[color=mycolor2, line width=1.5, solid, mark=star, mark size=3.5pt, mark repeat=50, mark options={solid, mycolor2, fill=mycolor2}]
        table[row sep={\\}]{
            0.0  0.009902534395925953 \\
            1.0  0.019469983247971504  \\
            2.0  0.019128214193040484  \\
            3.0  0.017545766125753513  \\
            4.0  0.015950651260531905  \\
            5.0  0.014181224241145018  \\
            6.0  0.013107990527067  \\
            7.0  0.012745948807585348  \\
            8.0  0.012304486759507365  \\
            9.0  0.01176965442511158  \\
            10.0  0.011485412334763356  \\
            11.0  0.011018292917095296  \\
            12.0  0.010729670420445129  \\
            13.0  0.010530381434288435  \\
            14.0  0.010196961868540329  \\
            15.0  0.00984901790090117  \\
            16.0  0.009528598817448478  \\
            17.0  0.009188117417250919  \\
            18.0  0.008798629189075577  \\
            19.0  0.00831003201634411  \\
            20.0  0.007888786566529568  \\
            21.0  0.00734057063867918  \\
            22.0  0.006713523010212428  \\
            23.0  0.006043801007049303  \\
            24.0  0.005407403268151358  \\
            25.0  0.004883966264606791  \\
            26.0  0.0045552979726833035  \\
            27.0  0.004306667303448922  \\
            28.0  0.004088207992355016  \\
            29.0  0.0039562387147686695  \\
            30.0  0.003827475022359877  \\
            31.0  0.0036969854769078408  \\
            32.0  0.0035517076327650577  \\
            33.0  0.00359606098910867  \\
            34.0  0.0032704555724104322  \\
            35.0  0.0034058542546884377  \\
            36.0  0.003274479147885468  \\
            37.0  0.0032982042634560273  \\
            38.0  0.0031634532681975584  \\
            39.0  0.0032448007519426717  \\
            40.0  0.0030989557211254922  \\
            41.0  0.0031006226129766945  \\
            42.0  0.0029885975124606715  \\
            43.0  0.0029031661754611803  \\
            44.0  0.002852620509559944  \\
            45.0  0.0028338835320103123  \\
            46.0  0.0028542295730474873  \\
            47.0  0.002749342874555552  \\
            48.0  0.002747523468069302  \\
            49.0  0.0026150264987154174  \\
            50.0  0.0025917119236569253  \\
            51.0  0.002558019087844755  \\
            52.0  0.0025103579671406618  \\
            53.0  0.002473516623496447  \\
            54.0  0.002425379043514342  \\
            55.0  0.002388514413038815  \\
            56.0  0.0023347209424029584  \\
            57.0  0.0023037614848421383  \\
            58.0  0.002253950097050567  \\
            59.0  0.002258684640453309  \\
            60.0  0.002233032225893043  \\
            61.0  0.0022568471498708087  \\
            62.0  0.002090270838761083  \\
            63.0  0.002065504410463293  \\
            64.0  0.0019977586202466397  \\
            65.0  0.0021008191549219887  \\
            66.0  0.0019270779172446601  \\
            67.0  0.0019162985982393696  \\
            68.0  0.0019889727585840717  \\
            69.0  0.001837004857696635  \\
            70.0  0.001898561738663574  \\
            71.0  0.0017610866774099804  \\
            72.0  0.001859365848438405  \\
            73.0  0.0017022210409591532  \\
            74.0  0.0016997337386432556  \\
            75.0  0.0016298278596862386  \\
            76.0  0.0016926334136966438  \\
            77.0  0.0015680030482943534  \\
            78.0  0.0016592246276349455  \\
            79.0  0.001514417459327244  \\
            80.0  0.0015240556285224045  \\
            81.0  0.0014531328049109575  \\
            82.0  0.0015075472822781806  \\
            83.0  0.0014141749835153556  \\
            84.0  0.0013728112343817402  \\
            85.0  0.0014252108756818045  \\
            86.0  0.001328233511672905  \\
            87.0  0.0013906532224718514  \\
            88.0  0.0012726570894077735  \\
            89.0  0.0013409349941403581  \\
            90.0  0.0012235753942680238  \\
            91.0  0.001200010623119039  \\
            92.0  0.0012522975907752708  \\
            93.0  0.0011530703923345329  \\
            94.0  0.001139626391174278  \\
            95.0  0.0011181966034310418  \\
            96.0  0.001104729578300254  \\
            97.0  0.0010829027867320857  \\
            98.0  0.0010433334455345782  \\
            99.0  0.001025334150794647  \\
            100.0  0.00101346740777403  \\
            101.0  0.0010204866007748393  \\
            102.0  0.0010187791905794407  \\
            103.0  0.0009562159157886785  \\
            104.0  0.0009348200163443616  \\
            105.0  0.0009149864857014215  \\
            106.0  0.0008987520737472476  \\
            107.0  0.0008817683894108251  \\
            108.0  0.0008647056008639375  \\
            109.0  0.0008836138386815963  \\
            110.0  0.000825116023457892  \\
            111.0  0.0008111823635018697  \\
            112.0  0.0007960539934071367  \\
            113.0  0.000781299636522792  \\
            114.0  0.00076545754349954  \\
            115.0  0.0007541532552938312  \\
            116.0  0.0007394065802028325  \\
            117.0  0.0007189798584684166  \\
            118.0  0.0007115778897843049  \\
            119.0  0.0006960926206438657  \\
            120.0  0.0006768014956265259  \\
            121.0  0.0006691391489763673  \\
            122.0  0.0006551002572139884  \\
            123.0  0.0006479328217433354  \\
            124.0  0.0006270722572127902  \\
            125.0  0.0006150985368723277  \\
            126.0  0.0006063432063890497  \\
            127.0  0.0005942418960247787  \\
            128.0  0.000575327759989819  \\
            129.0  0.000570468112309738  \\
            130.0  0.0005575695111451455  \\
            131.0  0.0005513481334664677  \\
            132.0  0.0005379797072341757  \\
            133.0  0.0005192494396564962  \\
            134.0  0.000517474713664902  \\
            135.0  0.0005040521111616583  \\
            136.0  0.000498399332427931  \\
            137.0  0.00047960126879820734  \\
            138.0  0.00046647253568822664  \\
            139.0  0.00045854490778093866  \\
            140.0  0.0004532447436379814  \\
            141.0  0.00044835177547845125  \\
            142.0  0.0004443399693343706  \\
            143.0  0.0004282891010420853  \\
            144.0  0.0004089196952191586  \\
            145.0  0.00041286981358158893  \\
            146.0  0.0004032636861916439  \\
            147.0  0.0003858217007854321  \\
            148.0  0.00037671056052470075  \\
            149.0  0.00037397894159913874  \\
            150.0  0.00037147762989975204  \\
            151.0  0.0003520942184583679  \\
            152.0  0.00034832692999711755  \\
            153.0  0.00034719199385556285  \\
            154.0  0.0003360852922449588  \\
            155.0  0.00032372406893816833  \\
            156.0  0.00031956971470291714  \\
            157.0  0.00031808148429942414  \\
            158.0  0.0003091243400791852  \\
            159.0  0.0002965608204971598  \\
            160.0  0.0002919418877318892  \\
            161.0  0.00028740187098125035  \\
            162.0  0.00028716582842467205  \\
            163.0  0.0002722964853359369  \\
            164.0  0.0002681228138165913  \\
            165.0  0.00026700374718461033  \\
            166.0  0.0002617980280715132  \\
            167.0  0.00025256222012611846  \\
            168.0  0.0002503124885684774  \\
            169.0  0.00024020000306849805  \\
            170.0  0.00023871625543115857  \\
            171.0  0.0002286345632853201  \\
            172.0  0.00022507863621963443  \\
            173.0  0.00022105952766951094  \\
            174.0  0.00021608054792229627  \\
            175.0  0.00021441903896784475  \\
            176.0  0.00020888269681998867  \\
            177.0  0.00020118980521890992  \\
            178.0  0.000198068442882521  \\
            179.0  0.00019645272021890858  \\
            180.0  0.00019126240054288254  \\
            181.0  0.0001843489222624284  \\
            182.0  0.00018054836913053727  \\
            183.0  0.0001779716695961831  \\
            184.0  0.0001771519493994485  \\
            185.0  0.0001743672978362025  \\
            186.0  0.00016570627387348067  \\
            187.0  0.00016426449264459042  \\
            188.0  0.00015801100967254296  \\
            189.0  0.0001554233960186781  \\
            190.0  0.0001549563586417711  \\
            191.0  0.00015081539082006785  \\
            192.0  0.000145199275421531  \\
            193.0  0.00014212926360409264  \\
            194.0  0.00014100983590567843  \\
            195.0  0.0001381138478944035  \\
            196.0  0.00013358225804242966  \\
            197.0  0.00013058006200451628  \\
            198.0  0.00012939173609333  \\
            199.0  0.00012598985959185198  \\
            200.0  0.0001216440055774612  \\
            201.0  0.00011963664522679299  \\
            202.0  0.00011911594234882332  \\
            203.0  0.00011584753841145227  \\
            204.0  0.0001116522439190497  \\
            205.0  0.0001092599295970755  \\
            206.0  0.00010838765126260933  \\
            207.0  0.00010604808367593063  \\
            208.0  0.00010273959520728496  \\
            209.0  0.000100350065281748  \\
            210.0  9.93919005509375e-5  \\
            211.0  9.670848331562172e-5  \\
            212.0  9.350817109669462e-5  \\
            213.0  9.195120571956543e-5  \\
            214.0  9.152267561781119e-5  \\
            215.0  8.892030897340725e-5  \\
            216.0  8.581277560597601e-5  \\
            217.0  8.394710620878866e-5  \\
            218.0  8.325877672611489e-5  \\
            219.0  8.104183234022071e-5  \\
            220.0  7.870126922272459e-5  \\
            221.0  7.729908462405381e-5  \\
            222.0  7.65362722656279e-5  \\
            223.0  7.427432689366046e-5  \\
            224.0  7.190910184321975e-5  \\
            225.0  7.033643215983128e-5  \\
            226.0  6.975038717605793e-5  \\
            227.0  6.822929603069947e-5  \\
            228.0  6.62308526599826e-5  \\
            229.0  6.464143660120177e-5  \\
            230.0  6.398181772259516e-5  \\
            231.0  6.221795514791882e-5  \\
            232.0  6.025139147779011e-5  \\
            233.0  5.8940057622605805e-5  \\
            234.0  5.806432990278255e-5  \\
            235.0  5.779587373743701e-5  \\
            236.0  5.5982950026927705e-5  \\
            237.0  5.409105391255051e-5  \\
            238.0  5.2883548516531764e-5  \\
            239.0  5.1905997968482685e-5  \\
            240.0  5.068553858240912e-5  \\
            241.0  4.963590252939088e-5  \\
            242.0  4.870271254323646e-5  \\
            243.0  4.8429697594809006e-5  \\
            244.0  4.686109701338028e-5  \\
            245.0  4.536496010789012e-5  \\
            246.0  4.435215044565355e-5  \\
            247.0  4.344481070786516e-5  \\
            248.0  4.301151262918664e-5  \\
            249.0  4.180111779417964e-5  \\
            250.0  4.0533228332139755e-5  \\
            251.0  3.9643141998346376e-5  \\
            252.0  3.88356197502278e-5  \\
            253.0  3.875251355192423e-5  \\
            254.0  3.7658651554250885e-5  \\
            255.0  3.6478822634938656e-5  \\
            256.0  3.557957736899993e-5  \\
            257.0  3.4858344864636804e-5  \\
            258.0  3.450860271627908e-5  \\
            259.0  3.353003015894403e-5  \\
            260.0  3.253185619028815e-5  \\
            261.0  3.1819262848792565e-5  \\
            262.0  3.1166762907089213e-5  \\
            263.0  3.093466377244793e-5  \\
            264.0  2.9962754445448394e-5  \\
            265.0  2.9104375316396686e-5  \\
            266.0  2.8472991038187446e-5  \\
            267.0  2.790595985891504e-5  \\
            268.0  2.7632986775690137e-5  \\
            269.0  2.6865121226773206e-5  \\
            270.0  2.6063921501615964e-5  \\
            271.0  2.5499232157987154e-5  \\
            272.0  2.4982983607483097e-5  \\
            273.0  2.4798615793382657e-5  \\
            274.0  2.402766009132552e-5  \\
            275.0  2.334138385537778e-5  \\
            276.0  2.2838089417102704e-5  \\
            277.0  2.2385078579721767e-5  \\
            278.0  2.2166686327467736e-5  \\
            279.0  2.155164840774366e-5  \\
            280.0  2.091451479754452e-5  \\
            281.0  2.0462394281160887e-5  \\
            282.0  2.004786386884907e-5  \\
            283.0  1.989947937755765e-5  \\
            284.0  1.9278263459693664e-5  \\
            285.0  1.8735430888765e-5  \\
            286.0  1.8331015267969592e-5  \\
            287.0  1.7965898446205225e-5  \\
            288.0  1.7788892064229627e-5  \\
            289.0  1.7290234299122163e-5  \\
            290.0  1.6788017671904714e-5  \\
            291.0  1.6424688493845287e-5  \\
            292.0  1.608900255891341e-5  \\
            293.0  1.596910952187533e-5  \\
            294.0  1.5464635100313346e-5  \\
            295.0  1.5038904602023933e-5  \\
            296.0  1.4712628252990083e-5  \\
            297.0  1.441764788840174e-5  \\
            298.0  1.4273172572626027e-5  \\
            299.0  1.3865754559572406e-5  \\
            300.0  1.3473175859667157e-5  \\
            301.0  1.3181490117481434e-5  \\
            302.0  1.2907866295474854e-5  \\
            303.0  1.2669196415877146e-5  \\
            304.0  1.2535161942580869e-5  \\
            305.0  1.2029614061829588e-5  \\
            306.0  1.1802299956894111e-5  \\
            307.0  1.1706780116870018e-5  \\
            308.0  1.1358550707134224e-5  \\
            309.0  1.1059514174348284e-5  \\
            310.0  1.0820322224161929e-5  \\
            311.0  1.0592152213034284e-5  \\
            312.0  1.0372528301430702e-5  \\
            313.0  1.0287318034093337e-5  \\
            314.0  9.914003264117968e-6  \\
            315.0  9.721019164767537e-6  \\
            316.0  9.475302744564747e-6  \\
            317.0  9.276505593033457e-6  \\
            318.0  9.192959231389812e-6  \\
            319.0  8.910570921643824e-6  \\
            320.0  8.682556586317882e-6  \\
            321.0  8.498251204358158e-6  \\
            322.0  8.303168383946751e-6  \\
            323.0  8.112563504646666e-6  \\
            324.0  8.05402896742062e-6  \\
            325.0  7.76240016614072e-6  \\
            326.0  7.603788976302833e-6  \\
            327.0  7.457730290666724e-6  \\
            328.0  7.297678641163056e-6  \\
            329.0  7.106728794312145e-6  \\
            330.0  7.034304097459607e-6  \\
            331.0  6.82594579772932e-6  \\
            332.0  6.666568266821561e-6  \\
            333.0  6.520394958029822e-6  \\
            334.0  6.358152340451205e-6  \\
            335.0  6.242932323361453e-6  \\
            336.0  6.313435450332148e-6  \\
            337.0  5.968465741269726e-6  \\
            338.0  5.811856117270325e-6  \\
            339.0  5.678996249051615e-6  \\
            340.0  5.595429031046026e-6  \\
            341.0  5.457189438285404e-6  \\
            342.0  5.330950213713247e-6  \\
            343.0  5.2312511067310094e-6  \\
            344.0  5.109910925324675e-6  \\
            345.0  4.981930199260932e-6  \\
            346.0  4.887070713642032e-6  \\
            347.0  4.780138236080661e-6  \\
            348.0  4.664095059643792e-6  \\
            349.0  4.580688761947561e-6  \\
            350.0  4.474829448442771e-6  \\
            351.0  4.370226148748276e-6  \\
            352.0  4.283135362429444e-6  \\
            353.0  4.189128978306384e-6  \\
            354.0  4.0866648848134145e-6  \\
            355.0  4.0123520223381065e-6  \\
            356.0  3.9186427183415815e-6  \\
            357.0  3.829655588727185e-6  \\
            358.0  3.7528428673736767e-6  \\
            359.0  3.672004500551018e-6  \\
            360.0  3.5846414171030163e-6  \\
            361.0  3.5014105547810296e-6  \\
            362.0  3.4262132538883267e-6  \\
            363.0  3.4105351644142206e-6  \\
            364.0  3.3493843592851687e-6  \\
            365.0  3.1980893012203272e-6  \\
            366.0  3.1415601113623175e-6  \\
            367.0  3.0745118234940122e-6  \\
            368.0  3.0148793942814806e-6  \\
            369.0  2.945902044066623e-6  \\
            370.0  2.875610183375651e-6  \\
            371.0  2.811390508342856e-6  \\
            372.0  2.7845433803150162e-6  \\
            373.0  2.684738587949476e-6  \\
            374.0  2.6297580631661813e-6  \\
            375.0  2.5799191663499186e-6  \\
            376.0  2.5252219147178664e-6  \\
            377.0  2.4649238309193875e-6  \\
            378.0  2.409628882926773e-6  \\
            379.0  2.3617190542497287e-6  \\
            380.0  2.3110289741153695e-6  \\
            381.0  2.260478681729184e-6  \\
            382.0  2.205771661731585e-6  \\
            383.0  2.163649764225848e-6  \\
            384.0  2.110814379604347e-6  \\
            385.0  2.093815504621882e-6  \\
            386.0  2.0176576056112486e-6  \\
            387.0  1.9780106626854537e-6  \\
            388.0  1.9324500388166066e-6  \\
            389.0  1.916878594087141e-6  \\
            390.0  1.8471683072284636e-6  \\
            391.0  1.8118494922118289e-6  \\
            392.0  1.7692754371487675e-6  \\
            393.0  1.7537201001608203e-6  \\
            394.0  1.691347486994747e-6  \\
            395.0  1.6578176519620492e-6  \\
            396.0  1.6208525782506008e-6  \\
            397.0  1.5931009809470315e-6  \\
            398.0  1.550158507557852e-6  \\
            399.0  1.520350002148284e-6  \\
            400.0  1.4860843823962602e-6  \\
            401.0  1.450892849921482e-6  \\
            402.0  1.4373661207867781e-6  \\
            403.0  1.3876910559732704e-6  \\
            404.0  1.3545651650318134e-6  \\
            405.0  1.3462291182348957e-6  \\
            406.0  1.2991764271443053e-6  \\
            407.0  1.2723384297485413e-6  \\
            408.0  1.2410238400447432e-6  \\
            409.0  1.2171782500401911e-6  \\
            410.0  1.1904469353470683e-6  \\
            411.0  1.1694265033089854e-6  \\
            412.0  1.1403419979078928e-6  \\
            413.0  1.114607968920136e-6  \\
            414.0  1.0891674921879483e-6  \\
            415.0  1.0835223515140132e-6  \\
            416.0  1.037691664107126e-6  \\
            417.0  1.0128348312253509e-6  \\
            418.0  9.961760464161486e-7  \\
            419.0  9.765047149025331e-7  \\
            420.0  9.498045008404931e-7  \\
            421.0  9.491051045714413e-7  \\
            422.0  9.106043794032616e-7  \\
            423.0  8.860856235118576e-7  \\
            424.0  8.727915676215584e-7  \\
            425.0  8.559925472750655e-7  \\
            426.0  8.320285743640934e-7  \\
            427.0  8.318121741646298e-7  \\
            428.0  7.985053992002126e-7  \\
            429.0  7.766005222790702e-7  \\
            430.0  7.653156198060274e-7  \\
            431.0  7.505429533414278e-7  \\
            432.0  7.294990944013894e-7  \\
            433.0  7.291787738741761e-7  \\
            434.0  7.001111631291303e-7  \\
            435.0  6.810198064121766e-7  \\
            436.0  6.709833977700907e-7  \\
            437.0  6.579867107413294e-7  \\
            438.0  6.395545315786966e-7  \\
            439.0  6.391564944258488e-7  \\
            440.0  6.120824569096073e-7  \\
            441.0  5.956046242927508e-7  \\
            442.0  5.874786562684361e-7  \\
            443.0  5.753798904141856e-7  \\
            444.0  5.667585121310657e-7  \\
            445.0  5.511575971113066e-7  \\
            446.0  5.402691816592291e-7  \\
            447.0  5.270961798306031e-7  \\
            448.0  5.170288732270187e-7  \\
            449.0  5.040479734405513e-7  \\
            450.0  4.940566999354882e-7  \\
            451.0  4.82971576854133e-7  \\
            452.0  4.723322468604487e-7  \\
            453.0  4.653958107769356e-7  \\
            454.0  4.519346042614123e-7  \\
            455.0  4.428736458774396e-7  \\
            456.0  4.3568163408594626e-7  \\
            457.0  4.234102688254448e-7  \\
            458.0  4.143700164907735e-7  \\
            459.0  4.0775483366342473e-7  \\
            460.0  3.9611328817358e-7  \\
            461.0  3.879935303665345e-7  \\
            462.0  3.8069441130265484e-7  \\
            463.0  3.6923430467973867e-7  \\
            464.0  3.624567266113642e-7  \\
            465.0  3.565916782253177e-7  \\
            466.0  3.468508525884943e-7  \\
            467.0  3.394572697454928e-7  \\
            468.0  3.3223939448936147e-7  \\
            469.0  3.2486278272169033e-7  \\
            470.0  3.1809534501932815e-7  \\
            471.0  3.1008462837887814e-7  \\
            472.0  3.085526764659599e-7  \\
            473.0  2.96188855326818e-7  \\
            474.0  2.893037565732197e-7  \\
            475.0  2.840819078218285e-7  \\
            476.0  2.785740750437278e-7  \\
            477.0  2.7105511101782834e-7  \\
            478.0  2.706038664185705e-7  \\
            479.0  2.590800103063191e-7  \\
            480.0  2.52552632826995e-7  \\
            481.0  2.488300958187266e-7  \\
            482.0  2.4366462461490774e-7  \\
            483.0  2.4006639564664444e-7  \\
            484.0  2.3339912103229252e-7  \\
            485.0  2.289166456964119e-7  \\
            486.0  2.2364040310291685e-7  \\
            487.0  2.1773483518132195e-7  \\
            488.0  2.1374828972932806e-7  \\
            489.0  2.0962874531580226e-7  \\
            490.0  2.0507748890087838e-7  \\
            491.0  2.0321869217801177e-7  \\
            492.0  1.9456808567671205e-7  \\
            493.0  1.902607505362328e-7  \\
            494.0  1.8717362755006243e-7  \\
            495.0  1.82068220062104e-7  \\
            496.0  1.7932860643900723e-7  \\
            497.0  1.74507990137284e-7  \\
            498.0  1.7172502948376252e-7  \\
            499.0  1.670945636374776e-7  \\
            500.0  1.6441326900426796e-7  \\
        };
\addplot[color=mycolor5, line width=1.25, dashdotted]
        table[row sep={\\}]{
            -1.0  1e-10 \\
        };
\addplot[color=mycolor4, line width=1.25, dashed, mark=diamond*, mark size=2.7pt, mark repeat=6, mark options={solid, mycolor4, fill=mycolor4}]
        table[row sep={\\}]{
            0.0  0.009902534395925953 \\
            1.0  0.06548667549974482  \\
            2.0  0.005139196215510177  \\
            3.0  0.008193269068324115  \\
            4.0  0.002079776174401103  \\
            5.0  0.0012071823866565638  \\
            6.0  0.0005399664641508578  \\
            7.0  0.000278329325984315  \\
            8.0  0.00016144755790313504  \\
            9.0  0.00010105113525421758  \\
            10.0  6.591061406383995e-5  \\
            11.0  4.654431330671754e-5  \\
            12.0  3.058633219582261e-5  \\
            13.0  2.107426693109457e-5  \\
            14.0  1.4529705848790543e-5  \\
            15.0  9.935970450331198e-6  \\
            16.0  6.822178871150205e-6  \\
            17.0  4.877143167203613e-6  \\
            18.0  3.2564243146388388e-6  \\
            19.0  2.2525012799099375e-6  \\
            20.0  1.556097037797235e-6  \\
            21.0  1.079195441805897e-6  \\
            22.0  7.496527944743798e-7  \\
            23.0  5.205481707081661e-7  \\
            24.0  3.6147369339953366e-7  \\
            25.0  2.5058101547232377e-7  \\
            26.0  1.7387522018823731e-7  \\
            27.0  1.197503434166841e-7  \\
            28.0  8.314804937007513e-8  \\
            29.0  5.746110222254236e-8  \\
            30.0  3.988139513317876e-8  \\
            31.0  2.7502674159578615e-8  \\
            32.0  1.9034554440754522e-8  \\
            33.0  1.3096684098991038e-8  \\
            34.0  9.099167458335326e-9  \\
        };
\end{axis}
\end{tikzpicture}

%% file: main_v2.bbl
\def\cprime{$'$}
\begin{thebibliography}{10}

\bibitem{AbsG06}
{\sc P.-A. Absil and K.~Gallivan}, {\em Joint diagonalization on the oblique
  manifold for independent component analysis}, in Proceedings of the IEEE
  International Conference on Acoustics Speech and Signal Processing, vol.~5,
  2006, pp.~V945--V948, \url{https://doi.org/10.1109/ICASSP.2006.1661433}.

\bibitem{AbsiMS08}
{\sc P.-A. Absil, R.~Mahony, and R.~Sepulchre}, {\em Optimization Algorithms on
  Matrix Manifolds}, Princeton University Press, Princeton, NJ, 2008.

\bibitem{Alt16}
{\sc H.~Alt}, {\em Linear Functional Analysis: An Application-Oriented
  Introduction}, Springer, London, 2016,
  \url{https://doi.org/10.1007/978-1-4471-7280-2}.

\bibitem{AltM95}
{\sc W.~Alt and K.~Malanowski}, {\em The {L}agrange-{N}ewton method for state
  constrained optimal control problems}, Comput. Optim. Appl., 4 (1995),
  pp.~217--239, \url{https://doi.org/10.1007/BF01300872}.

\bibitem{AltHP20}
{\sc R.~Altmann, P.~Henning, and D.~Peterseim}, {\em Quantitative {A}nderson
  localization of {S}chr\"odinger eigenstates under disorder potentials}, Math.
  Models Methods Appl. Sci., 30 (2020), pp.~917--955,
  \url{https://doi.org/10.1142/S0218202520500190}.

\bibitem{AltHP21}
{\sc R.~Altmann, P.~Henning, and D.~Peterseim}, {\em The {$J$}-method for the
  {G}ross--{P}itaevskii eigenvalue problem}, Numer. Math., 148 (2021),
  pp.~575--610, \url{https://doi.org/10.1007/s00211-021-01216-5}.

\bibitem{AltHP22}
{\sc R.~Altmann, P.~Henning, and D.~Peterseim}, {\em Localization and
  delocalization of ground states of {B}ose--{E}instein condensates under
  disorder}, SIAM J. Appl. Math., 82 (2022), pp.~330--358,
  \url{https://doi.org/10.1137/20M1342434}.

\bibitem{AltP19}
{\sc R.~Altmann and D.~Peterseim}, {\em Localized computation of eigenstates of
  random {S}chr\"odinger operators}, SIAM J. Sci. Comput., 41 (2019),
  pp.~B1211--B1227, \url{https://doi.org/10.1137/19M1252594}.

\bibitem{AltPS22}
{\sc R.~Altmann, D.~Peterseim, and T.~Stykel}, {\em Energy-adaptive
  {R}iemannian optimization on the {S}tiefel manifold}, ESAIM Math. Model.
  Numer. Anal., 56 (2022), pp.~1629--1653,
  \url{https://doi.org/10.1051/m2an/2022036}.

\bibitem{AltPS24}
{\sc R.~Altmann, D.~Peterseim, and T.~Stykel}, {\em Riemannian {N}ewton methods
  for energy minimization problems of {K}ohn--{S}ham type}, J. Sci. Comput.,
  101 (2024), article number 6,
  \url{https://doi.org/10.1007/s10915-024-02612-3}.

\bibitem{AntD14}
{\sc X.~Antoine and R.~Duboscq}, {\em {GPEL}ab, a {M}atlab toolbox to solve
  {G}ross–{P}itaevskii equations~{I}: {C}omputation of stationary solutions},
  Comput. Phys. Commun., 185 (2014), pp.~2969--2991,
  \url{https://doi.org/10.1016/j.cpc.2014.06.026}.

\bibitem{AntD14a}
{\sc X.~Antoine and R.~Duboscq}, {\em Robust and efficient preconditioned
  {K}rylov spectral solvers for computing the ground states of fast rotating
  and strongly interacting {B}ose--{E}instein condensates}, J. Comput. Phys.,
  258 (2014), pp.~509--523, \url{https://doi.org/10.1016/j.jcp.2013.10.045}.

\bibitem{AntLT17}
{\sc X.~Antoine, A.~Levitt, and Q.~Tang}, {\em Efficient spectral computation
  of the stationary states of rotating {B}ose–{E}instein condensates by
  preconditioned nonlinear conjugate gradient methods}, J. Comput. Phys., 343
  (2017), pp.~92--109, \url{https://doi.org/10.1016/j.jcp.2017.04.040}.

\bibitem{Bao04}
{\sc W.~Bao}, {\em Ground states and dynamics of multicomponent
  {B}ose--{E}instein condensates}, Multiscale Model. Simul., 2 (2004),
  pp.~210--236, \url{https://doi.org/10.1137/030600209}.

\bibitem{BaoC11}
{\sc W.~Bao and Y.~Cai}, {\em Ground states of two-component {B}ose--{E}instein
  condensates with an internal atomic {J}osephson junction}, East Asia J. Appl.
  Math., 1 (2011), pp.~49--81,
  \url{https://doi.org/10.4208/eajam.190310.170510a}.

\bibitem{BaoC13}
{\sc W.~Bao and Y.~Cai}, {\em Mathematical theory and numerical methods for
  {B}ose--{E}instein condensation}, Kinet. Relat. Mod., 6 (2013), pp.~1--135,
  \url{https://doi.org/10.3934/krm.2013.6.1}.

\bibitem{BaoC18}
{\sc W.~Bao and Y.~Cai}, {\em Mathematical models and numerical methods for
  spinor {B}ose--{E}instein condensates}, Commun. Comput. Phys., 24 (2018),
  pp.~899--965, \url{https://doi.org/10.4208/cicp.2018.hh80.14}.

\bibitem{BaoCL06}
{\sc W.~Bao, I.-L. Chern, and F.~Y. Lim}, {\em Efficient and spectrally
  accurate numerical methods for computing ground and first excited states in
  {B}ose--{E}instein condensates}, J. Comput. Phys., 219 (2006), pp.~836--854,
  \url{https://doi.org/10.1016/j.jcp.2006.04.019}.

\bibitem{BaoCZ13}
{\sc W.~Bao, I.-L. Chern, and Y.~Zhang}, {\em Efficient numerical methods for
  computing ground states of spin-1 {B}ose–{E}instein condensates based on
  their characterizations}, J. Comput. Phys., 253 (2013), pp.~189--208,
  \url{https://doi.org/10.1016/j.jcp.2013.06.036}.

\bibitem{BaoD04}
{\sc W.~Bao and Q.~Du}, {\em Computing the ground state solution of
  {B}ose--{E}instein condensates by a~normalized gradient flow}, SIAM J. Sci.
  Comput., 25 (2004), pp.~1674--1697,
  \url{https://doi.org/10.1137/S1064827503422956}.

\bibitem{BaoT03}
{\sc W.~Bao and W.~Tang}, {\em Ground-state solution of {B}ose--{E}instein
  condensate by directly minimizing the energy functional}, J. Comput. Phys.,
  187 (2003), pp.~230--254,
  \url{https://doi.org/10.1016/S0021-9991(03)00097-4}.

\bibitem{BouAMC20}
{\sc F.~Bouchard, B.~Afsari, J.~Malick, and M.~Congedo}, {\em Approximate joint
  diagonalization with {R}iemannian optimization on the general linear group},
  SIAM J. Matrix Anal. Appl., 41 (2020), pp.~152--170,
  \url{https://doi.org/10.1137/18M1232838}.

\bibitem{CalORT09}
{\sc M.~Caliari, A.~Ostermann, S.~Rainer, and M.~Thalhammer}, {\em A
  minimisation approach for computing the ground state of {G}ross--{P}itaevskii
  systems}, J. Comput. Phys., 228 (2009), pp.~349--360,
  \url{https://doi.org/10.1016/j.jcp.2008.09.018}.

\bibitem{CanCM10}
{\sc E.~Canc{\`e}s, R.~Chakir, and Y.~Maday}, {\em Numerical analysis of
  nonlinear eigenvalue problems}, J. Sci. Comput., 45 (2010), pp.~90--117,
  \url{https://doi.org/10.1007/s10915-010-9358-1}.

\bibitem{CanKL21}
{\sc E.~Canc\`{e}s, G.~Kemlin, and A.~Levitt}, {\em Convergence analysis of
  direct minimization and self-consistent iterations}, SIAM J. Matrix Anal.
  Appl., 42 (2021), pp.~243--274, \url{https://doi.org/10.1137/20M1332864}.

\bibitem{ChaKLS05}
{\sc S.-M. Chang, Y.-C. Kuo, W.-W. Lin, and S.-F. Shieh}, {\em A continuation
  {BSOR}-{L}anczos–{G}alerkin method for positive bound states of a
  multi-component {B}ose–{E}instein condensate}, J. Comput. Phys., 210
  (2005), pp.~439--458, \url{https://doi.org/10.1016/j.jcp.2005.04.019}.

\bibitem{CheW03}
{\sc G.-H. Chen and Y.-S. Wu}, {\em Quantum phase transition in a
  multicomponent {B}ose--{E}instein condensate in optical lattices}, Phys. Rev.
  A, 67 (2003), p.~013606, \url{https://doi.org/10.1103/PhysRevA.67.013606}.

\bibitem{ChenLLZ24}
{\sc Z.~Chen, J.~Lu, Y.~Lu, and X.~Zhang}, {\em On the convergence of {S}obolev
  gradient flow for the {G}ross--{P}itaevskii eigenvalue problem}, SIAM J.
  Numer. Anal., 62 (2024), pp.~667--691,
  \url{https://doi.org/10.1137/23M1552553}.

\bibitem{ChenLLZ24a}
{\sc Z.~Chen, J.~Lu, Y.~Lu, and X.~Zhang}, {\em Fully discretized {S}obolev
  gradient flow for the {G}ross--{P}itaevskii eigenvalue problem}, Math. Comp.,
   (2024, to appear), \url{https://doi.org/10.1090/mcom/4032}.

\bibitem{DanK10}
{\sc I.~Danaila and P.~Kazemi}, {\em A new {S}obolev gradient method for direct
  minimization of the {G}ross–{P}itaevskii energy with rotation}, SIAM J.
  Sci. Comput., 32 (2010), pp.~2447--2467,
  \url{https://doi.org/10.1137/100782115}.

\bibitem{DanP17}
{\sc I.~Danaila and B.~Protas}, {\em Computation of ground states of the
  {G}ross--{P}itaevskii functional via {R}iemannian optimization}, SIAM J. Sci.
  Comput., 39 (2017), pp.~B1102--B1129,
  \url{https://doi.org/10.1137/17M1121974}.

\bibitem{DioC07}
{\sc C.~M. Dion and E.~Canc{\`e}s}, {\em Ground state of the time-independent
  {G}ross--{P}itaevskii equation}, Comput. Phys. Commun., 177 (2007),
  pp.~787--798, \url{https://doi.org/10.1016/j.cpc.2007.04.007}.

\bibitem{DuL22}
{\sc C.-E. Du and C.-S. Liu}, {\em Newton--{N}oda iteration for computing the
  ground states of nonlinear {S}chr{\"o}dinger equations}, SIAM J. Sci.
  Comput., 44 (2022), pp.~A2370--A2385,
  \url{https://doi.org/10.1137/21M1435793}.

\bibitem{Eva10}
{\sc L.~Evans}, {\em Partial Differential Equations}, AMS, Providence, R.I.,
  2~ed., 2010.

\bibitem{GalHLP24}
{\sc D.~Gallistl, M.~Hauck, Y.~Liang, and D.~Peterseim}, {\em Mixed finite
  elements for the {G}ross-{P}itaevskii eigenvalue problem: a priori error
  analysis and guaranteed lower energy bound}, IMA J. Numer. Anal.,  (drae048,
  2024), \url{https://doi.org/10.1093/imanum/drae048}.

\bibitem{GaoPY23}
{\sc B.~Gao, R.~Peng, and Y.~Yuan}, {\em Optimization on product manifolds
  under a preconditioned metric}, {ArXiv e-print 2306.08873},  (2023),
  \url{https://doi.org/10.48550/arXiv.2306.08873}.

\bibitem{GarP01}
{\sc J.~J. Garc{\'{\i}}a-Ripoll and V.~M. P{\'e}rez-Garc{\'{\i}}a}, {\em
  Optimizing {S}chr\"odinger functionals using {S}obolev gradients:
  applications to quantum mechanics and nonlinear optics}, SIAM J. Sci.
  Comput., 23 (2001), pp.~1316--1334,
  \url{https://doi.org/10.1137/S1064827500377721}.

\bibitem{hauck2024positivitypreservingfiniteelement}
{\sc M.~Hauck, Y.~Liang, and D.~Peterseim}, {\em Positivity preserving finite
  element method for the {G}ross--{P}itaevskii ground state: discrete
  uniqueness and global convergence}, ArXiv e-print 2405.17090,  (2024),
  \url{https://doi.org/10.48550/arXiv.2405.17090}.

\bibitem{Hen23}
{\sc P.~Henning}, {\em The dependency of spectral gaps on the convergence of
  the inverse iteration for a~nonlinear eigenvector problem}, Math. Models
  Methods Appl. Sci., 33 (2023), pp.~1517--1544,
  \url{https://doi.org/10.1142/S0218202523500343}.

\bibitem{HenJ23}
{\sc P.~Henning and E.~Jarlebring}, {\em The {G}ross--{P}itaevskii equation and
  eigenvector nonlinearities: Numerical methods and algorithms}, SIAM Review,
  (2025, to appear).

\bibitem{HenMP14}
{\sc P.~Henning, A.~M\r{a}lqvist, and D.~Peterseim}, {\em Two-level
  discretization techniques for ground state computations of {B}ose--{E}instein
  condensates}, SIAM J. Numer. Anal., 52 (2014), pp.~1525--1550,
  \url{https://doi.org/10.1137/130921520}.

\bibitem{HenP20}
{\sc P.~Henning and D.~Peterseim}, {\em Sobolev gradient flow for the
  {G}ross--{P}itaevskii eigenvalue problem: global convergence and
  computational efficiency}, SIAM J. Numer. Anal., 58 (2020), pp.~1744--1772,
  \url{https://doi.org/doi: 10.1137/18M1230463}.

\bibitem{HenY25}
{\sc P.~Henning and M.~Yadav}, {\em Convergence of a {R}iemannian gradient
  method for the {G}ross--{P}itaevskii energy functional in a rotating frame},
  ESAIM Math. Model. Numer. Anal.,  (2025, to appear),
  \url{https://doi.org/10.1051/m2an/2025018}.

\bibitem{HuaY24}
{\sc P.~Huang and Q.~Yang}, {\em Newton-based alternating methods for the
  ground state of a class of multicomponent {B}ose-–{E}instein condensates},
  SIAM J. Optim., 34 (2024), pp.~3136--3162,
  \url{https://doi.org/10.1137/23M1580346}.

\bibitem{JarKM14}
{\sc E.~Jarlebring, S.~Kvaal, and W.~Michiels}, {\em An inverse iteration
  method for eigenvalue problems with eigenvector nonlinearities}, SIAM J. Sci.
  Comput., 36 (2014), pp.~A1978--A2001,
  \url{https://doi.org/10.1137/S1064827500366124}.

\bibitem{KaE10}
{\sc P.~Kazemi and M.~Eckart}, {\em Minimizing the {G}ross--{P}itaevskii energy
  functional with the {S}obolev gradient -- analytical and numerical results},
  Int. J. Comput. Methods, 7 (2010), pp.~453--475,
  \url{https://doi.org/10.1142/S0219876210002301}.

\bibitem{KreSV16}
{\sc D.~Kressner, M.~Steinlechner, and B.~Vandereycken}, {\em Preconditioned
  low-rank {R}iemannian optimization for linear systems with tensor product
  structure}, SIAM J. Sci. Comput., 38 (2016), pp.~A2018--A2044,
  \url{https://doi.org/10.1137/15M1032909}.

\bibitem{LieSY00}
{\sc E.~H. Lieb, R.~Seiringer, and J.~Yngvason}, {\em Bosons in a trap: {A}
  rigorous derivation of the {G}ross-{P}itaevskii energy functional}, Phys.
  Rev. A, 61 (2000), 043602, \url{https://doi.org/10.1103/PhysRevA.61.043602}.

\bibitem{LinW06}
{\sc T.-C. Lin and J.~Wei}, {\em Spikes in two-component systems of nonlinear
  {S}chr\"odinger equations with trapping potentials}, J. Differential
  Equations, 229 (2006), pp.~538--569,
  \url{https://doi.org/10.1016/j.jde.2005.12.011}.

\bibitem{MisS16}
{\sc B.~Mishra and R.~Sepulchre}, {\em Riemannian preconditioning}, SIAM J.
  Optim., 26 (2016), pp.~635--660, \url{https://doi.org/10.1137/140970860}.

\bibitem{Ngu23}
{\sc D.~Nguyen}, {\em Operator-valued formulas for {R}iemannian gradient and
  {H}essian and families of tractable metrics in {R}iemannian optimization}, J.
  Optim. Theory Appl., 198 (2023), pp.~135--164,
  \url{https://doi.org/10.1007/s10957-023-02242-z}.

\bibitem{PayTAAJ92}
{\sc M.~Payne, M.~Teter, D.~Ailan, T.~Arias, and J.~Joannopouios}, {\em
  Iterative minimization techniques for ab initio total-energy calculations:
  molecular dynamics and conjugate gradients}, Rev. Mod. Phys., 64 (1992),
  pp.~1045--1097, \url{https://doi.org/10.1103/RevModPhys.64.1045}.

\bibitem{SchRNB09}
{\sc R.~Schneider, T.~Rohwedder, A.~Neelov, and J.~Blauert}, {\em Direct
  minimization for calculating invariant subspaces in density functional
  computations of the electronic structure}, J. Comput. Math., 27 (2009),
  pp.~360--387, \url{http://global-sci.org/intro/article_detail/jcm/8577.html}.

\bibitem{SelAGQCLA12}
{\sc S.~E. Selvan, U.~Amato, K.~A. Gallivan, C.~Qi, M.~F. Carfora, M.~Larobina,
  and B.~Alfano}, {\em Descent algorithms on oblique manifold for
  source-adaptive {ICA} contrast}, IEEE Trans. Neural Netw. Learn. Syst., 23
  (2012), pp.~1930--1947, \url{https://doi.org/10.1109/TNNLS.2012.2218060}.

\bibitem{ShuA23}
{\sc B.~Shustin and H.~Avron}, {\em Riemannian optimization with a
  preconditioning scheme on the generalized {S}tiefel manifold}, J. Comput.
  Appl. Math., 423 (2023), 114953,
  \url{https://doi.org/10.1016/j.cam.2022.114953}.

\bibitem{TiaCWW20}
{\sc T.~Tian, Y.~Cai, X.~Wu, and Z.~Wen}, {\em Ground states of spin-{$F$}
  {B}ose--{E}instein condensates}, SIAM J. Sci. Comput., 42 (2020),
  pp.~B983--B1013, \url{https://doi.org/10.1137/19M1271117}.

\bibitem{Tri00}
{\sc M.~Trippenbach, K.~G{\'o}ral, K.~Rz\c{a}\.{z}ewski, B.~Malomed, and
  Y.~Band}, {\em Structure of binary {B}ose--{E}instein condensates}, J. Phys.
  B: At. Mol. Opt. Phys., 33 (2000), 4017,
  \url{https://doi.org/10.1088/0953-4075/33/19/314}.

\bibitem{WenY13}
{\sc Z.~Wen and W.~Yin}, {\em A feasible method for optimization with
  orthogonality constraints}, Math. Program., 142 (2013), pp.~397--434,
  \url{https://doi.org/10.1007/s10107-012-0584-1}.

\bibitem{WuWB17}
{\sc X.~Wu, Z.~Wen, and W.~Bao}, {\em A regularized {N}ewton method for
  computing ground states of {B}ose–{E}instein condensates}, J. Sci. Comput.,
  73 (2017), pp.~303--329, \url{https://doi.org/10.1007/s10915-017-0412-0}.

\bibitem{Zei88}
{\sc E.~Zeidler}, {\em Nonlinear Functional Analysis and its Applications IV:
  Applications to Mathematical Physics}, Springer, New York, 1988,
  \url{https://doi.org/10.1007/978-1-4612-4566-7}.

\bibitem{ZhaH04}
{\sc H.~Zhang and W.~Hager}, {\em A nonmonotone line search technique and its
  application to unconstrained optimization}, SIAM J. Optim., 14 (2004),
  pp.~1043--1056, \url{https://doi.org/10.1137/S1052623403428208}.

\bibitem{ZhaoBJ15}
{\sc Z.~Zhao, Z.-J. Bai, and X.-Q. Jin}, {\em A {R}iemannian {N}ewton algorithm
  for nonlinear eigenvalue problems}, SIAM J. Matrix Anal. Appl., 36 (2015),
  pp.~752--774, \url{https://doi.org/10.1137/140967994}.

\end{thebibliography}
